\documentclass{amsart}

\usepackage{amsmath} 
\usepackage{amssymb}
\usepackage{paralist}
\usepackage[english]{babel}
\usepackage[none]{hyphenat}
\usepackage[colorlinks=true]{hyperref}
\usepackage{enumitem}
\usepackage{array}
\usepackage[latin1]{inputenc}
\usepackage{esint}
\usepackage{latexsym,amsfonts}
\usepackage{amsthm}
\usepackage{cleveref}
\usepackage{mathtools}
\usepackage{verbatim}
\usepackage{graphics}

\usepackage{tikz}
\usetikzlibrary{matrix}
\usetikzlibrary{decorations.pathreplacing}
\usepackage{pgfplots}
\usepackage{tikz-3dplot}
\usepackage[all]{xy}

\numberwithin{equation}{section}

\newtheorem{theorem}{Theorem}[section]
\newtheorem*{theorem*}{Theorem}
\newtheorem{lemma}[theorem]{Lemma}

\newtheorem{corollary}[theorem]{Corollary}

\newtheorem{example}[theorem]{Example}
\theoremstyle{definition}
\newtheorem{definition}[theorem]{Definition}
\newtheorem{remark}[theorem]{Remark}

\newcommand{\abs}[1]{\left\vert#1\right\vert} 
\newcommand{\norm}[1]{\left\|#1\right\|} 
\newcommand{\pare}[1]{\left(#1\right)} 
\newcommand{\braces}[1]{\left\{#1\right\}} 
\newcommand{\set}[2]{\left\{#1 \; :\; #2\right\}} 
\newcommand{\V}{B^{\tiny{\X_n}}}

\newcommand{\Om}{\Omega}

\DeclareMathOperator*{\diam}{diam} 
\DeclareMathOperator*{\dist}{dist} 
\DeclareMathOperator*{\trace}{Tr} 

\def\vint_#1{\mathchoice%
          {\mathop{\kern 0.2em\vrule width 0.6em height 0.69678ex depth -0.58065ex
                  \kern -0.8em \intop}\nolimits_{\kern -0.4em#1}}%
          {\mathop{\kern 0.1em\vrule width 0.5em height 0.69678ex depth -0.60387ex
                  \kern -0.6em \intop}\nolimits_{#1}}%
          {\mathop{\kern 0.1em\vrule width 0.5em height 0.69678ex
              depth -0.60387ex
                  \kern -0.6em \intop}\nolimits_{#1}}%
          {\mathop{\kern 0.1em\vrule width 0.5em height 0.69678ex depth -0.60387ex
                  \kern -0.6em \intop}\nolimits_{#1}}}
                  
                  \newcommand{\aveint}[2]{\mathchoice%
          {\mathop{\kern 0.2em\vrule width 0.6em height 0.69678ex depth -0.58065ex
                  \kern -0.8em \intop}\nolimits_{\kern -0.45em#1}^{#2}}%
          {\mathop{\kern 0.1em\vrule width 0.5em height 0.69678ex depth -0.60387ex
                  \kern -0.6em \intop}\nolimits_{#1}^{#2}}%
          {\mathop{\kern 0.1em\vrule width 0.5em height 0.69678ex depth -0.60387ex
                  \kern -0.6em \intop}\nolimits_{#1}^{#2}}%
          {\mathop{\kern 0.1em\vrule width 0.5em height 0.69678ex depth -0.60387ex
                  \kern -0.6em \intop}\nolimits_{#1}^{#2}}}

\newcommand{\R}{\mathbb R}
\newcommand{\N}{\mathbb N}
\newcommand{\Z}{\mathbb Z}

\newcommand{\M}{\mathcal{M}}

\newcommand{\eps}{\varepsilon}

\renewcommand{\O}{\mathcal{O}}
\renewcommand{\L}{\mathcal{L}}

\newcommand{\A}{\mathcal{A}}
\newcommand{\U}{\mathcal{U}}

\newcommand{\X}{\mathcal{X}}
\renewcommand{\P}{\mathbb{P}}
\newcommand{\E}{\mathbb{E}}

\renewcommand{\geq}{\geqslant}
\renewcommand{\leq}{\leqslant}

\def\Xint#1{\mathchoice
   {\XXint\displaystyle\textstyle{#1}}%
   {\XXint\textstyle\scriptstyle{#1}}%
   {\XXint\scriptstyle\scriptscriptstyle{#1}}%
   {\XXint\scriptscriptstyle\scriptscriptstyle{#1}}%
   \!\int}
\def\XXint#1#2#3{{\setbox0=\hbox{$#1{#2#3}{\int}$}
     \vcenter{\hbox{$#2#3$}}\kern-.49\wd0}}

\def\dashint{\ \Xint{\raisebox{-5.2pt}{$-$}}}

\DeclareSymbolFont{bbold}{U}{bbold}{m}{n}
\DeclareSymbolFontAlphabet{\mathbbold}{bbold}
\newcommand{\ind}{\mathbbold{1}}

\newcommand{\dens}{\phi}
\newcommand{\Omr}{{\Omega}} 
\newcommand{\card}[1]{\mathrm{card}(#1)}
\begin{document}

\title[Krylov-Safonov theory for random data clouds]{Krylov-Safonov theory for Pucci-type extremal inequalities on random data clouds}
\date{\today}
\author[Arroyo]{\'Angel Arroyo}
\address{Departamento de Matem\'aticas, Universidad de Alicante, 03690 San Vicente del Raspeig, Alicante, Spain}
\email{angelrene.arroyo@ua.es}
\author[Blanc]{Pablo Blanc}
\address{Departamento de Matem\'atica, FCEyN, Universidad de Buenos Aires, Pabell\'on I, Ciudad Universitaria (1428), Buenos Aires, Argentina.}
\email{pblanc@dm.uba.ar}
\author[Parviainen]{Mikko Parviainen}
\address{Department of Mathematics and Statistics, University of Jyv\"askyl\"a, PO~Box~35, FI-40014 Jyv\"askyl\"a, Finland}
\email{mikko.j.parviainen@jyu.fi}



\thanks{\'AA  is supported  by the MICIN/AEI through the Grant PID2021-123151NB-I00 and MP by Research Council of Finland, project 360184. Part of this research was done during a visit of \'AA to the University of Jyv\"askyl\"a in 2022 under the JYU Visiting Fellow Programme.}

\maketitle

\begin{abstract}\sloppy 
We establish Krylov-Safonov type H\"older regularity theory for solutions to quite general discrete dynamic programming equations or equivalently discrete stochastic processes on random geometric graphs. Such graphs arise for example from data clouds in graph-based machine learning. The results actually hold to functions satisfying Pucci-type extremal inequalities, and thus we cover many examples including tug-of-war games on random geometric graphs. As an application we show that under suitable assumptions when the number of data points increases, the graph functions converge to a solution of a partial differential equation.
\end{abstract}

\tableofcontents

\section{Introduction}\label{sec:intro}

\sloppy

In this paper, we develop a new method to establish H\"older type regularity theory on random geometric graphs.
These graphs commonly emerge for example from data clouds in graph-based machine learning.
It is shown that a function satisfying a certain dynamic programming principle  on such a graph is approximately H\"older regular.
These functions can arise as the expectation of a discrete stochastic process or as a result of a machine learning algorithm.
 The results are actually derived to functions satisfying Pucci-type extremal operators on random geometric graphs
 \begin{align*}
\L_{\X_n, \eps}^+ u\geq -\rho,\quad \L_{\X_n, \eps}^- u\leq  \rho,
\end{align*}
described in more detail in \Cref{sec:roadmap,sec:examples}.
Our main result is the H\"older estimate obtained in \Cref{thm:main-graph-holder-regularity}
 \begin{align*}
|u(Z_i)-u(Z_j)|
	\leq
	C(|Z_i-Z_j|^\gamma+\eps^\gamma).
\end{align*} 
As an application, such regularity results allow us to pass to a limit  (Theorems~\ref{thm:visc-sol} and \ref{thm:visc-sol-p}) when the number of data points increases, and if the functions $u$ are solutions to suitable graph problems, then the limit is a solution to a PDE. 

Such general results have turned out to be fruitful in the theory of partial differential equations as well as in the theory of nonlocal operators.  In the context of elliptic PDEs with merely bounded and measurable coefficients, the closest counterpart would be the original Krylov-Safonov method \cite{krylovs79} using a probabilistic proof and  Trudinger's proof \cite{trudinger80} using analytic techniques. In the context of viscosity solutions to fully nonlinear uniformly elliptic equations, the corresponding result was obtained by Caffarelli \cite{caffarelli89}. The results of this paper establish the Krylov-Safonov regularity theory in the context of random geometric graphs and stochastic processes, dynamic programming principles or equivalently with operators that can be compared to fully nonlinear uniformly elliptic PDE operators.  
Observe however that the regularity techniques in PDEs or in the nonlocal setting of for example Caffarelli and Silvestre \cite{caffarellis09} utilize, heuristically speaking, the fact that there is information available in all scales.  In our case, the step size and the graph sets a natural limit for the scale, and this limitation has some crucial effects.

Another inconvenience on a random geometric graph is its lack of symmetry. The symmetry is a crucial ingredient for the regularity estimates obtained for example to orthogonal grids in Kuo and Trudinger \cite{kuot90}, or in the nonlocal setting of \cite{caffarellis09}. Also more recently in \cite{arroyobp23} and \cite{arroyobp22} in $\R^N$, the symmetry was utilized. On the other hand, Kuo and Trudinger \cite{kuot96} consider a discrete graph not having the mentioned symmetry property, but this is done for a linear operator. Inspecting the proof in this paper one can see that also here the linear part of the operator does not require symmetry considerations.

Next we discuss our results in the machine learning context in the spirit of \cite{chapellesz06, elmoatazdt17, calder19b, caldergl22}.  In a typical machine learning problem, there is a large collection of data such as medical images, some of which is labeled by an expert. The task is to extend the labels to a larger data set in a reasonable meaningful way. In fully supervised learning, the algorithms make use of only the labeled data. Thus fully supervised algorithms usually need sufficiently diverse collection of labeled data to learn from.
In many applications, labeling data requires human annotation and this can be  expensive to obtain. On the other hand, unlabeled data is often of low cost and increasingly available. As a result, there is now significant interest in graph-based semi-supervised learning algorithms for example utilizing discretized partial differential equations that take advantage of both the labeled data as well as the structure of the unlabeled data to achieve improved performance. 

It is natural that in different applications in semi-supervised learning, different operators give the best practical results and thus developing methods applying to a general class is well motivated.  For example, it has  recently been discovered in certain cases that the game theoretic $p$-Laplacian on a graph gives advantage over commonly used methods in semi-supervised learning  with large amounts of unlabeled data. In particular, the problem is well-posed unlike in some more traditional approaches \cite{calder19b}. Thus developing methods that have wide applicability is well motivated.
In this context, the results of this paper can be interpreted to convey information on how the learning algorithm regularizes the problem and they also allow passing to a large data limit. For Lipschitz learning, convergence results and convergence rate results have been obtained for example in \cite{calder19}, \cite{bungertr23} and  \cite{bungertcr23}. In the case of Laplacian learning, Calder, Garc\'{\i}a Trillos and Lewicka, \cite{caldergl22}, utilize the idea of using nonlocal operator in the intermediate steps. In their setting, already in the case of the Laplacian, the proof of the regularity result requires a considerable amount of work. 

Since the results of this paper hold for Pucci-type extremal operators, the results immediately cover many examples such as tug-of-war with noise type stochastic games as described in \Cref{sec:examples}. Tug-of-war games in different contexts have been an object of a considerable recent interest, see for example \cite{peress08, peresssw09} and monographs \cite{blancr19b,  lewicka20, parviainen24}.

\subsection{Plan of the proof}
\label{sec:roadmap}

In this section, we present the plan of the proof of the Krylov-Safonov type H\"older regularity result, \Cref{thm:main-graph-holder-regularity}. We consider a random data cloud $\X_n=\{Z_1,\ldots,Z_n\}\subset\Omega\subset \R^N$ where each $Z_i$ is an independent and identically distributed random variable, $\V_\eps(x)=\X_n\cap B_\eps(x)$, and denote Pucci-type maximal operators (\Cref{def:L-Xn}) as
\begin{align*}
\L_{\X_n,\eps}^+u(Z_i)
	=
	\frac{1}{\eps^2}\bigg[&\alpha\max_{Z_j\in\V_{\Lambda\eps}(Z_i)}\max_{Z_k\in\V_{\tau\eps^2}(2Z_i-Z_j)}\frac{u(Z_j)+u(Z_k)}{2}
	\\
	~&
	+
	\beta\frac{1}{\card{\V_\eps(Z_i)}}\sum_{Z_j\in\V_\eps(Z_i)}u(Z_j)-u(Z_i)\bigg]
\end{align*}
defined on the data cloud, where $\Lambda,\tau\geq 1$ are fixed constants. Details will be given in \Cref{part1} but let us mention that the second maximum above gives a robust way to write down the extremal operator on nonsymmetric graphs. In the Euclidean space this part of the operator would simply contain a reflected point. Our task is then to establish a regularity result in \Cref{thm:main-graph-holder-regularity} for solutions or in general for functions that satisfy Pucci-type extremal inequalities on the data cloud. 

First we extend the function only defined on the data cloud to the whole domain in $\R^N$. Then we show that the extension approximately satisfies the Pucci-type extremal inequalities for the $\eps$-nonlocal operator (\Cref{def:L-Omega})
\begin{equation*}
	\L_{\Omega,\eps}^+v(x)
	=
	\frac{1}{\eps^2}\bigg[\alpha\sup_{|z|<\Lambda\eps}\sup_{|h|<\tau\eps^2}\frac{v(x+z)+v(x-z+h)}{2}
	+
	\beta\dashint_{B_\eps(x)}v\ d\mu(y)-v(x)\bigg].
\end{equation*}
Here the measure $\mu$ depends on the probability distribution of the data points.

Then the aim is to utilize the general Krylov-Safonov type regularity theory we have developed for $\eps$-nonlocal operators/stochastic processes/dynamic programming principles in \cite{arroyobp22, arroyobp23}. Observe however that the data cloud is random and also the limiting probability density is nonuniform and nonsymmetric. We need to take this into account and this introduces a kind of  a drift to the problem: since this is of independent interest, we present the Krylov-Safonov type H\"older regularity theory for $\eps$-nonlocal operators with nonsymmetric kernels in \Cref{part2} (the main theorems in that part are $\eps$-ABP estimate with nonsymmetric density, \Cref{eps-ABP-thm-2}, and H\"older regularity result \Cref{Holder}). To emphasize that the ABP estimate and regularity results obtained in \Cref{part2} are quite robust and that $\eps$-nonlocal definitions originally tailored for treating the loss of symmetry due to the graph allow extra flexibility, we present a further example at the end, \Cref{ex:drift}.

Regardless, a key is to pass from discrete  to nonlocal  as already suggested above. To this end, we need to extend the function to the whole domain. Moreover, as we want to cover not only solutions but also functions that satisfy the extremal inequalities, the extension should not depend on the operator.  The precise definition of the extension is given in \Cref{def:transport} but the idea is to construct a cover which corresponds to the data points and to extend the function as a constant from a data point to the related set in the cover. This will be denoted in terms of a transport map that transports a point to the related data point: so for a discrete function $u$ we define the extension as  $(u\circ T_\eps)(x)=u(T_\eps(x))$, where $x\in \R^n$ and $T_\eps$ denotes the transport map. Using this notation, the desired discrete  to $\eps$-nonlocal estimate roughly reads as
\begin{align}
\label{eq:roadmap-discrete-to-nonlocal}
	\L_{\X_n,\eps}^+u(T_\eps(x))
	\leq
	\L_{\Omega,\eps}^+[u\circ T_\eps](x)
	+
	C\|u\|_{L^\infty(\X_n)}.
\end{align}
Since data points are arriving at random, there is no hope that this would always hold but it is sufficient if it holds with high enough probability. 

The estimate is given in  \Cref{thm:discrete-to-nonlocal}, and the idea to prove it is as follows. A suitable estimation in \eqref{eq:max-from-local-to-nonlocal} for the nonlinear part in the discrete operator gives
\begin{align*}
	\max_{Z_i\in\V_{\Lambda\eps}(T_\eps(x))}&\max_{Z_j\in\V_{\tau\eps^2}(2T_\eps(x)-Z_i)}\frac{u(Z_i)+u(Z_j)}{2}
	\\
	&\leq \sup_{|z|<(\Lambda+\eps^2)\eps}\sup_{|h|<(\tau+2\eps)\eps^2}\frac{(u\circ T_\eps)(x+z)+(u\circ T_\eps)(x-z+h)}{2}.
\end{align*}
This estimate allows us to deal with the lack of symmetry in a graph but then we need to deal with the right hand side. By the previous estimate and the definitions of the operators, adjusting the parameters, we have
\begin{align}
\label{eq:roadmap-key-averages}
	\L_{\X_n,\eps}^+u(T_\eps(x))
	\leq
\L_{\Omega,\eps}^+[u\circ T_\eps](x)
	+
	\frac{\beta}{\eps^2}\bigg(\frac{1}{\card{\V_\eps(T_\eps(x))}}\sum_{Z_j\in\V_\eps(T_\eps(x))}u(Z_i)\nonumber
	\\
	-
	\dashint_{B_\eps(x)}(u\circ T_\eps)\ d\mu\bigg),
\end{align}
so it only remains to estimate the difference of the $\beta$-terms but this is the most delicate part of the proof given later as \Cref{thm:discrete-to-nonlocal-averages}. 
Let us also point out as explained in more detail in \Cref{sec:discrete-to-nonlocal}, that unlike in simpler cases, with our operator, we cannot use directly Berstein's concentration inequality to state that the discrete sum approximates the integral when the number of points increases. This is because the extended function is not independent of random data points.

\subsection{Examples of the operators}
\label{sec:examples}
The  main result \Cref{thm:main-graph-holder-regularity} gives the H\"older estimate for $u\in L^\infty(\X_n)$ satisfying the Pucci bounds
\begin{equation}\label{eq:LXn+f}
	\L_{\X_n,\eps}^+u\geq-\rho,
	\qquad
	\L_{\X_n,\eps}^-u\leq\rho,
\end{equation}
with $\rho>0$. Next we enlist some examples of discrete operators $\L_{\X_n,\eps}:L^\infty(\X_n)\to L^\infty(\X_n)$ covered by our work.
In both of them, if $u\in L^\infty(\X_n)$ is such that $\L_{\X_n,\eps}u=f$ in $\X_n$, then $u$ satisfies \eqref{eq:LXn+f} with $\rho=\|f\|_{L^\infty(\Omega)}$.

\begin{example}\label{example:discrete-operator}
Let $\beta=1-\alpha\in(0,1]$. Given any $Z_i\in\X_n$ and $x\in\Omega$, we define $\overline{Z_i}^x$ as the closest point in $\X_n$ to the reflection of $Z_i$ with respect to $x$. 
We consider
$\L_{\X_n,\eps}:L^\infty(\X_n)\to L^\infty(\X_n)$ defined by
\begin{equation*}
\begin{split}
	\L_{\X_n,\eps} u(Z_i)
	=
	~&
	\alpha\sum_{Z_j\in \V_{\Lambda\eps}(Z_i)}\nu_{Z_i}(Z_j)\frac{u(Z_j)+u(\overline{Z_j}^{Z_i})-2u(Z_i)}{2\eps^2}
	\\
	~&
	+
	\beta\,\frac{1}{\card{\V_\eps(Z_i)}}\sum_{Z_j\in \V_\eps(Z_i)}\frac{u(Z_j)-u(Z_i)}{\eps^2}
\end{split}
\end{equation*}
for each $i=1,\ldots,n$, where $\alpha = 1 - \beta \in [0, 1)$ and $\nu_{Z_i}:\V_{\Lambda\eps}(Z_i)\to[0,1]$ is such that
\begin{equation*}
	\sum_{Z_j\in \V_{\Lambda\eps}(Z_i)}\nu_{Z_i}(Z_j)
	=
	1.
\end{equation*}
Then
\begin{align*}
	\L_{\X_n,\eps}^-u(Z_i)
	\leq
	\L_{\X_n,\eps}u(Z_i)
	\leq
	\L_{\X_n,\eps}^+u(Z_i)
\end{align*}
for each $Z_i\in\X_n$.
\end{example}

\begin{example}[Tug-of-war type graph operators]\label{example:tug-of-war}
Let $p\geq 2$ and set $\beta=1-\alpha=\frac{N+2}{N+p}\in(0,1]$. Consider the graph operator $\L_{\X_n,\eps}:L^\infty(\X_n)\to L^\infty(\X_n)$ defined by
\begin{align*}
	\L_{\X_n,\eps}&u(Z_i)\eps^2+u(Z_i)
	\\
	=
	~&	
	\frac{\alpha}{2}\Bigg\{\inf_{Z_j\in \V_\eps(Z_i)}u(Z_j)+\sup_{Z_j\in \V_\eps(Z_i)}u(Z_j)\Bigg\}+
	\frac{\beta}{\card{\V_\eps(Z_i)}}\sum_{Z_j\in \V_\eps(Z_i)}u(Z_j)
\end{align*}
for each $u\in L^\infty(\X_n)$ and $Z_i\in\X_n$. This is related to the $p$-Laplace equation and tug-of-war games. In the Euclidean case, the counterpart of this was considered in \cite{manfredipr12} and from local regularity perspective, apart from references already mentioned, in \cite{luirops13} and \cite{luirop18}. On data clouds with $p>n$ in the context of semi-supervised learning such operators were considered in \cite{calder19b}.
\end{example}

\part{Operators on random geometric graphs}\label{part1}

\section{The random data cloud}

Let $\Omega\subset\R^N$ be an open bounded domain and $n\in\N$. We say that a finite set $\X_n=\{Z_1,\ldots,Z_n\}\subset \Omega$ is a \emph{random data cloud} in $\Omega$ if each $Z_i$ is an independent and identically distributed (i.i.d.) random variable according to a probability measure $\mu$ supported in $\Omega$. We denote by  $\dens:\Omega\to[\dens_0,\dens_1]$ the Lipschitz density function of $\mu$, where $0<\dens_0\leq\dens_1<\infty$. For simplicity, we assume that $\mu$ is a measure in the whole $\R^N$ by defining $\mu(A)=\mu(A\cap \Omega)$ for every measurable set $A\subset\R^N$, and thus $\dens\equiv 0$ in $\R^N\setminus \Omega$. Hence
\begin{equation*}
	\mu(A)=\int_A\dens(x)\ dx
\end{equation*}
for each measurable set $A$.

Observe that $\mu$ is absolutely continuous with respect to the Lebesgue measure. Moreover, $\mu$ is comparable to the Lebesgue measure, in the sense that
\begin{equation}\label{mu-bounds}
	\dens_0|A|\leq\mu(A)\leq\dens_1|A|
\end{equation}
for every measurable set $A\subset \Omega$. Furthermore, for any integrable function $f\geq 0$ in $A$ it holds that
\begin{align}
\label{eq:comparable}
	\frac{\dens_0}{\dens_1}\dashint_Af(x)\ dx
	\leq
	\dashint_Af(x)\ d\mu(x)
	\leq
	\frac{\dens_1}{\dens_0}\dashint_Af(x)\ dx.
\end{align}

Given a set $A\subset\Omega$, we denote by $A^{\X_n}$ the set of vertices $Z_i\in\X_n$ contained in $A$, ie.
\begin{equation*}
	A^{\X_n}
	=
	\X_n\cap A.
\end{equation*}
Moreover, we denote by
\begin{equation*}
	\card{A^{\X_n}}
	\geq
	0
\end{equation*}
the number of points in $A^{\X_n}$. In that way, $\V_r(x)=\X_n\cap B_r(x)$ with $r>0$. Observe that $\V_r(Z_i)\neq\emptyset$ for every $Z_i\in\X_n$ and any $r>0$.

We say that a constant $C>0$ is \emph{universal} if it only depends on $N$, $\beta$, $\Lambda$, $\tau$, $\dens_0$, $\dens_1$ and the Lipschitz constant of $\dens$. Additional dependencies are specified in parenthesis, for example $C(\diam\Omega)$ denotes a constant which depends on the universal constants and on $\diam\Omega$. For the scale parameter we will assume that $\eps\in(0,\eps_0)$, with
\begin{equation*}
	\eps_0
	\in
	(0,\tfrac{1}{2})
\end{equation*}
a universal constant fixed in such a way that it satisfies \eqref{eps0-bound-a}, \eqref{eps0-bound-b}, \eqref{eps0-bound-c}, \eqref{eps0-bound-1}, \eqref{eps-bound-r-R}, see \Cref{remarkeps0}. In addition, we assume that $\eps_0$ is such that $\Omega_{-\eps_0}:= \{x\in \Om\,:\,\dist(x,\R^N\setminus \Omega)>\eps_0\}\neq\emptyset$.

We say that a pairwise disjoint collection of sets $\U_1,\ldots,\U_n\subset\Omega$ with positive measure such that $\bigcup_{i=1}^n\U_i=\Omega$ is a \emph{partition} of $\Omega$.

Next we state the result that enables us to precisely describe the event of the points being roughly \textit{evenly} distributed according to the measure $\mu$. The approximate density $\dens_{\eps}$ appearing in the statement will simply be the histogram density estimator. For this, we assume a kind of an interior measure density condition for $\Omega$ which reads as follows (see \Cref{sec:histogram} for a more detailed description): there exists $c>0$ such that for every $\delta>0$ there exists $B_1,\dots,B_M$ a partition of $\Omega$ such that $B_i\subset B_\delta(x_i)$ for some $x_i\in B_i$ and $|B_i|\geq c\delta^N$. 
For simplicity, when no confusion arises, we denote by $\|\cdot\|_\infty$ the $L^\infty$-norm.

\begin{lemma}\label{lem:dens-eps}
Let $\dens$ be a Lipschitz continuous probability density in $\overline\Omega$ and $\eps>0$. For a random data cloud $\X_n=\{Z_1,\ldots,Z_n\}\subset\Omega$ with density $\dens$, there exists a probability measure $\mu_\eps$ with density $\dens_\eps\in L^\infty(\Omega)$ and a partition $\U_1,\ldots,\U_n$ of $\Omega$ such that
\begin{equation}\label{partition}
	Z_i
	\in
	\U_i\subset B_{\eps^3}(Z_i)
	\quad\text{and}\quad
	\mu_\eps(\U_i)=\int_{\U_i}\dens_\eps(y)\ dy=\frac{1}{n}
\end{equation}
and
\begin{equation}\label{event}
	\|\dens_\eps-\dens\|_\infty
	\leq
	C_0\eps^2
\end{equation}
with probability at least
\begin{equation}\label{probability}
	1-2n\exp\{-c_0n\eps^{3N+4}\}.
\end{equation}
where $C_0>0$ and $c_0=c_0(\Omega)>0$ are universal. 
\end{lemma}

\begin{proof} 
This result is taken from 
 \cite[Proposition 2.12]{calderg22}. For the convenience of the reader, we give the proof in \Cref{sec:histogram}. We take $\delta=\eps^3$ and $\lambda=\eps^2$ in \Cref{lem:apendix} and redefine each set $\U_i$ in the partition as $(\U_i\setminus\X_n)\cup\{Z_i\}$. 
\end{proof}

We remark that the previous result is only relevant when \eqref{probability} is strictly positive, and this gives a condition between $n$ and $\eps$. Moreover, $\U_i\cap\X_n=\{Z_i\}$ for each $i$.

As a direct consequence, assuming the event \eqref{event}, if $A$ is a measurable set in $\Omega$, then
\begin{equation}\label{mu-eps-mu}
	|\mu_\eps(A)-\mu(A)|
	\leq
	\int_A|\dens_\eps(y)-\dens(y)|\ dy
	\leq
	C_0|A|\eps^2.
\end{equation}
Since $\dens_0|A|\leq\mu(A)\leq\dens_1|A|$ by \eqref{mu-bounds}, this implies
\begin{equation*}
	(\dens_0-C_0\eps^2)|A|
	\leq
	\mu_\eps(A)
	\leq
	(\dens_1+C_0\eps^2)|A|.
\end{equation*}
We assume that $\eps>0$ is small enough so that the left hand side is positive. More precisely, $\eps\in(0,\eps_0)$ with
\begin{equation}\label{eps0-bound-a}
	\eps_0
	\leq
	\pare{\frac{\dens_0}{2C_0}}^{1/2},
\end{equation}
in which case
\begin{equation*}
	\frac{\dens_0}{2}|A|
	\leq
	\mu_\eps(A)
	\leq
	\frac{3\dens_1}{2}|A|.
\end{equation*}
As a consequence,
\begin{equation*}
	\bigg|\frac{\mu(A)}{\mu_\eps(A)}-1\bigg|
	=
	\frac{|\mu_\eps(A)-\mu(A)|}{\mu_\eps(A)}
	\leq
	\frac{2C_0}{\dens_0}\eps^2.
\end{equation*}
Moreover, if $A_1$ and $A_2$ are two measurable sets in $\Omega$ and $\eps$ satisfies \eqref{eps0-bound-a}, then
\begin{equation}\label{mu-eps-quotient}
	\frac{\mu_\eps(A_1)}{\mu_\eps(A_2)}
	\leq
	\frac{\dens_1+C_0\eps^2}{\dens_0-C_0\eps^2}\cdot\frac{|A_1|}{|A_2|}
	\leq
	3\frac{\dens_1}{\dens_0}\cdot\frac{|A_1|}{|A_2|}.
\end{equation}

Given a bounded function $u$ defined in the random data cloud $\X_n\subset\Omega$, the partition $\{\U_1,\ldots,\U_n\}$ of $\Omega$ from \Cref{lem:dens-eps} satisfying \eqref{partition} allows us to define a projection from $\Omega$ to $\X_n$.

\begin{definition}
\label{def:transport}
Let $\X_n=\{Z_1,\ldots,Z_n\}\subset\Omega$ and a partition $\U_1,\ldots,\U_n$ of $\Omega$. We define the \emph{transport map} 
\begin{equation*}
	T_\eps:\Omega\to\X_n
\end{equation*}
as the projection map satisfying that 
\begin{equation*}
	T_\eps(x)=Z_i
	\quad\text{ if and only if }\quad
	x\in\U_i.
\end{equation*}
\end{definition}

By means of the transport map, one can construct an extension of a given function in $\X_n$ by noting that $u\circ T_\eps\in L^\infty(\Omega)$ satisfies $(u\circ T_\eps)\big|_{\X_n}=u$ for each $u\in L^\infty(\X_n)$. Furthermore, the discrete average over $\X_n$ of a given function $u\in L^\infty(\X_n)$ can be related with the $\mu_\eps$-average over $\Omega$ of its extension $(u\circ T_\eps)\in L^\infty(\Omega)$ using \eqref{partition} and the fact that $\{Z_i\}=T_\eps(\U_i)$ as follows,
\begin{equation}\label{transport-integral}
\begin{split}
	\frac{1}{n}\sum_{Z_i\in\X_n}u(Z_i)
	=
	~&
	\sum_{Z_i\in\X_n}\mu_\eps(\U_i)u(Z_i)
	\\
	=
	~&
	\sum_{Z_i\in\X_n}\int_{\U_i}(u\circ T_\eps)\ d\mu_\eps
	\\
	=
	~&
	\int_\Omega(u\circ T_\eps)\ d\mu_\eps.
\end{split}
\end{equation}

We also make the following observation about the image of a set $A\subset\Omega$ under the transport map. If $r>0$ and we define
\begin{align}
\label{eq:fattening}
	A_r
	=
	\{x\in\R^N\,:\,\dist(x,A)<r\}
\end{align}
and
\begin{equation*}
	A_{-r}
	=
	\{x\in A\,:\,\dist(x,\R^N\setminus A)>r\},
\end{equation*}
then the inclusions
\begin{equation}\label{inclusions}
	A_{-\eps^3}
	\subset
	T_\eps^{-1}(A^{\X_n})
	\subset
	A_{\eps^3}
\end{equation}
hold, where $A^{\X_n}=\X_n\cap A$. To see this, let $x\in A_{-\eps^3}$. Since the sets $\U_1,\ldots,\U_n$ form a partition of $\Omega$, there exists some $i$ such that $x\in\U_i$ and by the definition of the transport map this is equivalent to $T_\eps(x)=Z_i$. Moreover, recalling \eqref{partition} we have that $\U_i\subset B_{\eps^3}(Z_i)$, so $Z_i$ is at a distance no greater than $\eps^3$ to $x$. Thus $Z_i\in A$, and in particular $Z_i\in A^{\X_n}$. Composing with the inverse map we get that $x\in T_{\eps}^{-1}(Z_i)\subset T_{\eps}^{-1}(A^{\X_n})$, from which the first inclusion in \eqref{inclusions} follows. Similarly for the other inclusion, let $x\in T_\eps^{-1}(A^{\X_n})$. Then $T_\eps(x)=Z_i$ for some $Z_i\in A^{\X_n}$. Again by the definition of the transport map and \eqref{partition}, it turns out that $x\in T_\eps^{-1}(Z_i)=\U_i\subset B_{\eps^3}(Z_i)$, then $x\in A_{\eps^3}$.

Furthermore, for convenience we write the second inclusion in \eqref{inclusions} in terms of indicator functions. More precisely, it holds that
\begin{equation}\label{indicator-A}
	\ind_A(T_\eps(y))
	=
	\ind_{T_\eps^{-1}(A^{\X_n})}(y)
	\leq
	\ind_{A_{\eps^3}}(y)
\end{equation}
for every $y\in\Omega$.

It is worth to mention that if $A\subset\Omega$ then it is not true in general that $A_{\eps^3}\subset\Omega$. However, this does not pose any problem.

\section{Discrete--to--$\varepsilon$-nonlocal estimates for linear averages}
\label{sec:discrete-to-nonlocal}

We explained the plan of the proof in \Cref{sec:roadmap}: the key idea is to pass from discrete Pucci-type extremal inequalities to $\eps$-nonlocal Pucci-type extremal inequalities. We will deal with the nonlinear part in the operators in \Cref{sec:nonlinear-part}, and apart from that as pointed out in \eqref{eq:roadmap-key-averages}, the key is to estimate the difference of the averages
\begin{align*}
\frac{1}{\card{\V_\eps(T_\eps(x))}}\sum_{Z_i\in\V_\eps(T_\eps(x))}u(Z_i)
	-
	\dashint_{B_\eps(x)}(u\circ T_\eps)\ d\mu.
\end{align*}
This is the task of this section, and the actual estimate will be given in \Cref{thm:discrete-to-nonlocal-averages}. Here is a rough sketch of the proof: adding and subtracting terms conveniently we can split the difference into the following parts
\begin{equation*}
\begin{split}
	\textbf{[\,A\,]}
	=
	~&
	\bigg|\frac{1}{\card{\V_\eps(x)}}\sum_{Z_j\in\V_\eps(x)}u(Z_j)
	-
	\frac{1}{\card{\V_\eps(T_\eps(x))}}\sum_{Z_j\in\V_\eps(T_\eps(x))}u(Z_j)\bigg|,
	\\
	\textbf{[\,B\,]}
	=
	~&
	\bigg|\frac{1}{\card{\V_\eps(x)}}\sum_{Z_j\in\V_\eps(x)}u(Z_j)
	-
	\dashint_{B_\eps(x)}(u\circ T_\eps)\ d\mu_\eps\bigg|,
	\\
	\textbf{[\,C\,]}
	=
	~&
	\bigg|\dashint_{B_\eps(x)}(u\circ T_\eps)\ d\mu
	-
	\dashint_{B_\eps(x)}(u\circ T_\eps)\ d\mu_\eps\bigg|.
\end{split}
\end{equation*}
In order to establish the discrete to nonlocal estimate in \eqref{eq:roadmap-discrete-to-nonlocal}, we estimate \textbf{A}, \textbf{B} and \textbf{C} above. The estimate for \textbf{A} is in  \Cref{lem:A}. 
The key estimate for \textbf{B} is given in \Cref{lem:B}. 
Finally, the estimate for \textbf{C} is stated in  \Cref{lem:C}. These estimates are based on different geometric observations and concentration of measure type stochastic approximations that utilize \Cref{lem:dens-eps}, and allow us to prove that such estimates hold with a high probability that tends to $1$ as $n\to \infty$ with fixed $\eps$.

But first, let us stress that, given $v\in L^\infty(\Omega)$, one can use Bernstein's inequality (see eq. (2.10) from \cite{boucheronlm13}) to obtain the same control of the difference between the arithmetic mean of $v\big|_{\X_n}$ in $\V_\eps(x)$ and the $\mu$-average of $v$ with high probability. However, this argument cannot be applied to $v=u\circ T_\eps$, because in that case $v$ depends on the extension map $T_\eps$, and thus on the random graph $\X_n$, so the independence of the random variables involved in the derivation of Bernstein's inequality fails. It is worth to mention that in the linear case in \cite{caldergl22}, this problem is solved by introducing the double convolution kernel, which allows to untie the function $v$ from the probability estimates. However, this approach does not seem to work in our setting, since the operators considered in this paper are nonlinear. 

We start with some auxiliary lemmas.  The first estimate allows us to employ \Cref{lem:dens-eps} to approximate the portion of points from $\X_n$ that fall inside a given set $A\subset\Omega$ by its measure $\mu(A)$ with high probability. This will be useful when approximating the difference between the discrete graph operator and $\eps$-nonlocal operator.

\begin{lemma}\label{lem:card/n-mu(A)}
Let $\eps\in(0,\eps_0)$ and suppose that the events \eqref{partition} and \eqref{event} hold and let $A\subset\Omega$ be any measurable set. There exists a universal constant $C>0$ such that
\begin{equation}\label{estimate-card/n-mu}
	\bigg|\frac{\card{A^{\X_n}}}{n}-\mu(A)\bigg|
	\leq
	C\big(|A_{\eps^3}\setminus A_{-\eps^3}|+|A_{\eps^3}|\eps^2\big),
\end{equation}
where $A^{\X_n}=\X_n\cap A$ and $A_{\eps^3}$ was defined in \eqref{eq:fattening}.
\end{lemma}

\begin{proof}
We start by adding and subtracting $\mu(T_\eps^{-1}(A^{\X_n}))$ so we can study the following differences separately,
\begin{equation}\label{card/n-mu}
\begin{split}
	\bigg|\frac{\card{A^{\X_n}}}{n}-\mu(A)\bigg|
	\leq
	~&
	\bigg|\frac{\card{A^{\X_n}}}{n}-\mu(T_\eps^{-1}(A^{\X_n}))\bigg|
	\\
	~&
	+
	\big|\mu(T_\eps^{-1}(A^{\X_n}))-\mu(A)\big|.
\end{split}
\end{equation}

For the first term in the right hand side, using \eqref{transport-integral} and \eqref{indicator-A} we can write
\begin{equation}\label{card(A)/n-identity}
\begin{split}
	\frac{\card{A^{\X_n}}}{n}
	=
	~&
	\frac{1}{n}\sum_{Z_i\in\X_n}\ind_A(Z_i)
	\\
	=
	~&
	\int_\Omega\ind_A(T_\eps(y))\ d\mu_\eps(y)
	\\
	=
	~&
	\int_\Omega\ind_{T_\eps^{-1}(A^{\X_n})}(y)\dens_\eps(y)\ dy,
\end{split}
\end{equation}
where $\dens_\eps$ is the density function of $\mu_\eps$ from \Cref{lem:dens-eps}. Then, using \eqref{event} and the inequality in \eqref{indicator-A} we have
\begin{align*}
	\bigg|\frac{\card{A^{\X_n}}}{n}-\mu(T_\eps^{-1}(A^{\X_n}))\bigg|
	\leq
	~&
	\int_\Omega\ind_{T_\eps^{-1}(A^{\X_n})}(y)|\dens_\eps(y)-\dens(y)|\ dy
	\\
	\leq
	~&
	C_0|A_{\eps^3}|\eps^2.
\end{align*}

For the remaining term in \eqref{card/n-mu}, recalling the inclusions in \eqref{inclusions} together with the fact that $\mu$ is a measure with density $\dens(y)\leq\dens_1$, we obtain
\begin{align*}
	\big|\mu(T_\eps^{-1}(A^{\X_n}))-\mu(A)\big|
	\leq
	~&
	\mu(A_{\eps^3})-\mu(A_{-\eps^3})
	\\
	=
	~&
	\mu(A_{\eps^3}\setminus A_{-\eps^3})
	\\
	\leq
	~&
	\dens_1|A_{\eps^3}\setminus A_{-\eps^3}|.
\end{align*}
Finally, \eqref{estimate-card/n-mu} is obtained after inserting these estimates in \eqref{card/n-mu}.
\end{proof}

In this paper we are interested in the particular case in which the set $A$ is a ball of radius $\eps$, so the following result follows as a corollary.

\begin{lemma}\label{lem:card/n-mu(B)}
Let $\eps\in(0,\eps_0)$ and suppose that the events \eqref{partition} and \eqref{event} hold. There exists a universal constant $C>0$ such that
\begin{equation}\label{card/n-mu(B)-eq1}
	\bigg|\frac{\card{\V_\eps(x)}}{n}-\mu(B_\eps(x))\bigg|
	\leq
	C\eps^{N+2},
\end{equation}
for every $x\in\Omega_{-\eps}$. In particular,
\begin{equation}\label{card/n-mu(B)-eq2}
	\bigg|\frac{\card{\V_\eps(x)}}{n\,\mu(B_\eps(x))}-1\bigg|
	\leq
	C\eps^2
\end{equation}
and
\begin{equation}\label{card/n-mu(B)-eq3}
	\frac{\card{\V_\eps(x)}}{n}
	\geq
	C\eps^N
\end{equation}
for every $x\in\Omega_{-\eps}$.
\end{lemma}

\begin{proof}
Apply \Cref{lem:card/n-mu(A)} with $A=B_\eps(x)$ to get
\begin{equation*}
	\bigg|\frac{\card{\V_\eps(x)}}{n}-\mu(B_\eps(x))\bigg|
	\leq
	C\big(|B_{\eps+\eps^3}\setminus B_{\eps-\eps^3}|+|B_{\eps+\eps^3}|\eps^2\big).
\end{equation*}
Let us focus on the terms in the right hand side separately. First, using the Mean Value Theorem and the convexity of $t\mapsto t^N$,
\begin{equation}\label{ring-estimate}
\begin{split}
	|B_{\eps+\eps^3}\setminus B_{\eps-\eps^3}|
	=
	~&
	|B_1|\big((1+\eps^2)^N-(1-\eps^2)^N\big)\eps^N
	\\
	\leq
	~&
	|B_1|\big(N(1+\eps^2)^{N-1}2\eps^2\big)\eps^N
	\\
	\leq
	~&
	|B_1|2^NN\eps^{N+2}
\end{split}
\end{equation}
for $\eps\in(0,\eps_0)$. Second,
\begin{align*}
	|B_{\eps+\eps^3}|
	=
	|B_1|(1+\eps^2)^N\eps^N
	\leq
	|B_1|2^N\eps^N.
\end{align*}
Replacing these two inequalities above we obtain \eqref{card/n-mu(B)-eq1}. Furthermore, since $\mu(B_\eps(x))\geq\dens_0|B_1|\eps^N$, this immediately implies \eqref{card/n-mu(B)-eq2}. On the other hand, from \eqref{card/n-mu(B)-eq1} we have
\begin{equation*}
	\frac{\card{\V_\eps(x)}}{n}
	\geq
	\mu(B_\eps(x))-C\eps^{N+2}
	\geq
	(\dens_0|B_1|-C\eps^2)\eps^N
\end{equation*}
for $x\in\Omega_{-\eps}$. The right hand side is uniformly bounded away from $0$ for any sufficiently small $\eps>0$. That is, if we select
\begin{equation}\label{eps0-bound-b}
	\eps_0
	\leq
	\pare{\frac{\dens_0|B_1|}{2C}}^{1/2}
\end{equation}
then \eqref{card/n-mu(B)-eq3} follows for any $\eps\in(0,\eps_0)$.
\end{proof}

Another instance for which \Cref{lem:card/n-mu(A)} is needed is the case in which $A$ is an annular set, that is, the symmetric difference of concentric balls. This is also of great help for obtaining an upper bound of the portion of points in $\X_n$ that fall in the difference of two non-concentric $\eps$-balls whose centers are at an $\eps^3$-distance, as we state in the following lemma.

\begin{lemma}\label{lem:card(B-B)/n}
Let $\eps\in(0,\eps_0)$ and suppose that the events \eqref{partition} and \eqref{event} hold. There exists a universal constant $C>0$ such that
\begin{equation*}
	\frac{\card{\V_\eps(x)\triangle\V_\eps(\widetilde x)}}{n}
	\leq
	C\eps^{N+2}
\end{equation*}
for every $x\in\Omega$ and $\widetilde x\in \Omega$ such that $|x-\widetilde x|\leq\eps^3$.
\end{lemma}

\begin{proof}
Taking into account that $|x-\widetilde x|\leq\eps^3$, by a basic geometric observation we have that
\begin{equation*}
	B_\eps(x)\triangle B_\eps(\widetilde x)
	\subset
	\big(B_{\eps+\eps^3}(x)\setminus B_{\eps-\eps^3}(x)\big)
	\cap\Omega.
\end{equation*}
Let us denote by $A$ the set in the right hand side. Then 
$$
	A_{\eps^3}\subset
	B_{\eps+2\eps^3}(x)\setminus \overline B_{\eps-2\eps^3}(x),
$$
and using the elementary inequality $|(t+h)^N-(t-h)^N|\leq2N(t+h)^{N-1}h$ for $0<h<t$ we estimate
\begin{align*}
	|A_{\eps^3}|
	\leq
	~&
	|B_{\eps+2\eps^3}\setminus \overline B_{\eps-2\eps^3}|
	\\
	=
	~&
	|B_1|\big((\eps+2\eps^3)^N-(\eps-2\eps^3)^N\big)
	\\
	\leq
	~&
	2N|B_1|(\eps+2\eps^2)^{N-1}\eps^3
	\\
	\leq
	~&
	C\eps^{N+2}
\end{align*}
for any $\eps\in(0,\eps_0)$. On the other hand, $|A_{-\eps^3}|=|\overline B_\eps \setminus B_\eps|=0$ so we have
\begin{align*}
	|A_{\eps^3}\setminus A_{-\eps^3}|
	\leq
	C\eps^{N+2}.
\end{align*}
Thus, \Cref{lem:card/n-mu(A)} yields
\begin{equation*}
	\bigg|\frac{\card{\V_{\eps+\eps^3}(x)\setminus\V_{\eps-\eps^3}(x)}}{n}-\mu(B_{\eps+\eps^3}(x)\setminus B_{\eps-\eps^3}(x))\bigg|
	\leq
	C\eps^{N+2}
\end{equation*}
for any $\eps\in(0,\eps_0)$. Hence we can estimate
\begin{align*}
	\frac{\card{\V_\eps(x)\triangle \V_\eps(\widetilde x)}}{n}
	\leq
	~&
	\frac{\card{\V_{\eps+\eps^3}(x)\setminus\V_{\eps-\eps^3}(x)}}{n}
	\\
	\leq
	~&
	\mu(B_{\eps+\eps^3}(x)\setminus B_{\eps-\eps^3}(x))
	+
	C\eps^{N+2}
	\\
	\leq
	~&
	C\big(|B_{\eps+\eps^3}\setminus B_{\eps-\eps^3}|+\eps^{N+2}\big)
	\\
	\leq
	~&
	C\eps^{N+2},
\end{align*}
so the proof is finished.
\end{proof}

Let us recall the identity \eqref{transport-integral} stating that for any function $u\in L^\infty(\X_n)$ it holds that
\begin{equation*}
	\frac{1}{n}\sum_{Z_i\in\X_n}u(Z_i)
	=
	\int_\Omega(u\circ T_\eps)\ d\mu_\eps.
\end{equation*}
Therefore, in view of \Cref{lem:card/n-mu(B)}, which relates the portion of points in $\X_n$ that fall inside a ball $B_\eps(x)$ with $\mu(B_\eps(x))$, and \eqref{event} (and more precisely \eqref{mu-eps-mu}), that bounds the difference between the measures $\mu$ and $\mu_\eps$, it is reasonable to expect that one can relate the discrete average over $\V_\eps(x)$ of any function $u\in L^\infty(\X_n)$ with the $\mu$-average of its extension $u\circ T_\eps\in L^\infty(\Omega)$. This is indeed the statement of the main result of this section.

\begin{theorem}\label{thm:discrete-to-nonlocal-averages}
Let $\X_n=\{Z_1,\ldots,Z_n\}$ be a random graph in $\Omega$ and $u\in L^\infty(\X_n)$. Let $\eps\in(0,\eps_0)$ and assume that the events \eqref{partition} and \eqref{event} are satisfied. There exists a universal constant $C>0$ such that
\begin{equation}\label{eq:discrete-to-nonlocal-averages}
	\bigg|\frac{1}{\card{\V_\eps(T_\eps(x))}}\sum_{Z_i\in\V_\eps(T_\eps(x))}u(Z_i)
	-
	\dashint_{B_\eps(x)}(u\circ T_\eps)\ d\mu\bigg|
	\leq
	C\|u\|_{L^\infty(\X_n)}\eps^2,
\end{equation}
for every $x\in \Omega_{-\eps}$.
\end{theorem}
The idea of the proof  was explained at the beginning of the section, and is split into several lemmas.

\begin{lemma}\label{lem:A}
Let $\eps\in(0,\eps_0)$ and assume that the events \eqref{partition} and \eqref{event} are satisfied. There exists a universal constant $C>0$ such that
\begin{equation*}
	\bigg|\frac{1}{\card{\V_\eps(x)}}\sum_{Z_i\in\V_\eps(x)}u(Z_i)
	-
	\frac{1}{\card{\V_\eps(T_\eps(x))}}\sum_{Z_i\in\V_\eps(T_\eps(x))}u(Z_i)\bigg|
	\leq
	C\|u\|_{L^\infty(\X_n)}\eps^2
\end{equation*}
for every $x\in\Omega_{-\eps}$.
\end{lemma}

\begin{proof}
For simplicity, we denote $\widetilde x=T_\eps(x)\in\Omega$, so $|x-\widetilde x|\leq\eps^3$. Adding and subtracting $\frac{1}{\card{\V_\eps(x)}}\sum_{Z_i\in\V_\eps(\widetilde x)}u(Z_i)$ we can write
\begin{multline}\label{ineq-Aa}
	\bigg|\frac{1}{\card{\V_\eps(x)}}\sum_{Z_i\in\V_\eps(x)}u(Z_i)
	-
	\frac{1}{\card{\V_\eps(\widetilde x)}}\sum_{Z_i\in\V_\eps(\widetilde x)}u(Z_i)\bigg|
	\\
	\leq
	\frac{1}{\card{\V_\eps(x)}}\bigg|\sum_{Z_i\in\V_\eps(x)}u(Z_i)-\sum_{Z_i\in\V_\eps(\widetilde x)}u(Z_i)\bigg|
	\\
	+
	\bigg|\frac{1}{\card{\V_\eps(x)}}-\frac{1}{\card{\V_\eps(\widetilde x)}}\bigg|\sum_{Z_i\in \V_\eps(\widetilde x)}|u(Z_i)|.
\end{multline}
For the first term in \eqref{ineq-Aa}, recalling that $(A_1\setminus A_2)\cup(A_2\setminus A_1)=A_1\triangle A_2$,
\begin{multline*}
	\bigg|\sum_{Z_i\in\V_\eps(x)}u(Z_i)-\sum_{Z_i\in\V_\eps(\widetilde x)}u(Z_i)\bigg|
	\\
\begin{split}
	=
	~&
	\bigg|\sum_{Z_i\in\V_\eps(x)\setminus\V_\eps(\widetilde x)}u(Z_i)-\sum_{Z_i\in\V_\eps(\widetilde x)\setminus\V_\eps(x)}u(Z_i)\bigg|
	\\
	\leq
	~&
	\sum_{Z_i\in\V_\eps(x)\triangle\V_\eps(\widetilde x)}|u(Z_i)|
	\\
	\leq
	~&
	\card{\V_\eps(x)\triangle\V_\eps(\widetilde x)}\|u\|_{L^\infty(\X_n)}.
\end{split}
\end{multline*}
For the other term, using that $|\card{A_1}-\card{A_2}|\leq\card{A_1\triangle A_2}$,
\begin{equation*}
	\bigg|\frac{1}{\card{\V_\eps(x)}}-\frac{1}{\card{\V_\eps(\widetilde x)}}\bigg|
	\leq
	\frac{\card{\V_\eps(x)\triangle\V_\eps(\widetilde x)}}{\card{\V_\eps(x)}\,\card{\V_\eps(\widetilde x)}},
\end{equation*}
and since
\begin{equation*}
	\sum_{Z_i\in \V_\eps(\widetilde x)}|u(Z_i)|
	\leq
	\card{\V_\eps(\widetilde x)}\|u\|_{L^\infty(\X_n)},
\end{equation*}
we also have
\begin{multline*}
	\bigg|\frac{1}{\card{\V_\eps(x)}}-\frac{1}{\card{\V_\eps(\widetilde x)}}\bigg|\sum_{Z_i\in \V_\eps(\widetilde x)}|u(Z_i)|
	\\
	\leq
	\frac{\card{\V_\eps(x)\triangle\V_\eps(\widetilde x)}}{\card{\V_\eps(x)}}\|u\|_{L^\infty(\X_n)}.
\end{multline*}
Inserting these estimates in \eqref{ineq-Aa} we obtain
\begin{multline*}
	\bigg|\frac{1}{\card{\V_\eps(x)}}\sum_{Z_i\in\V_\eps(x)}u(Z_i)
	-
	\frac{1}{\card{\V_\eps(\widetilde x)}}\sum_{Z_i\in\V_\eps(\widetilde x)}u(Z_i)\bigg|
	\\
	\leq
	2\frac{\card{\V_\eps(x)\triangle\V_\eps(\widetilde x)}}{\card{\V_\eps(x)}}\|u\|_{L^\infty(\X_n)}.
\end{multline*}
Finally, we conclude the proof by noting that, by \Cref{lem:card/n-mu(B),lem:card(B-B)/n},
\begin{equation*}
	\frac{\card{\V_\eps(x)\triangle\V_\eps(\widetilde x)}}{\card{\V_\eps(x)}}
	\leq
	C\eps^2,
\end{equation*}
for any  $\eps\in(0,\eps_0)$.
\end{proof}

\begin{lemma}\label{lem:B}
Let $\eps\in(0,\eps_0)$ and assume that the events \eqref{partition} and \eqref{event} are satisfied. There exists a universal constant $C>0$ such that
\begin{equation*}
	\bigg|\frac{1}{\card{\V_\eps(x)}}\sum_{Z_i\in\V_\eps(x)}u(Z_i)
	-
	\dashint_{B_\eps(x)}(u\circ T_\eps)\ d\mu_\eps\bigg|
	\leq
	C\|u\|_{L^\infty(\X_n)}\eps^2
\end{equation*}
for every $x\in\Omega_{-\eps}$.
\end{lemma}

\begin{proof}
We add and subtract $\frac{1}{n\,\mu_\eps(B_\eps(x))}\sum_{Z_i\in\V_\eps(x)}u(Z_i)$ to estimate
\begin{multline}\label{ineq-B}
	\bigg|\frac{1}{\card{\V_\eps(x)}}\sum_{Z_i\in\V_\eps(x)}u(Z_i)
	-
	\dashint_{B_\eps(x)}(u\circ T_\eps)\ d\mu_\eps\bigg|
	\\
\begin{split}
	\leq
	~&
	\frac{1}{\mu_\eps(B_\eps(x))}\bigg|\frac{1}{n}\sum_{Z_i\in\V_\eps(x)}u(Z_i)
	-\int_{B_\eps(x)}(u\circ T_\eps)\ d\mu_\eps\bigg|
	\\
	~&
	+
	\bigg|\frac{1}{n\,\mu_\eps(B_\eps(x))}-\frac{1}{\card{\V_\eps(x)}}\bigg|\sum_{Z_i\in\V_\eps(x)}|u(Z_i)|.
\end{split}
\end{multline}

We focus first on the first term in the right hand side in \eqref{ineq-B}. Similarly to \eqref{card(A)/n-identity}, using \eqref{transport-integral} and \eqref{indicator-A}, we deduce
\begin{align*}
	\frac{1}{n}\sum_{Z_i\in\V_\eps(x)}u(Z_i)
	=
	~&
	\frac{1}{n}\sum_{Z_i\in\X_n}\ind_{B_\eps(x)}(Z_i)u(Z_i)
	\\
	=
	~&
	\int_\Omega\ind_{B_\eps(x)}(T_\eps(y))(u\circ T_\eps)(y)\ d\mu_\eps(y)
	\\
	=
	~&
	\int_\Omega\ind_{T_\eps^{-1}(\V_\eps(x))}(y)(u\circ T_\eps)(y)\ d\mu_\eps(y),
\end{align*}
so
\begin{multline*}
	\frac{1}{\mu_\eps(B_\eps(x))}\bigg|\frac{1}{n}\sum_{Z_i\in\V_\eps(x)}u(Z_i)
	-\int_{B_\eps(x)}(u\circ T_\eps)\ d\mu_\eps\bigg|
	\\
\begin{split}
	=
	~&
	\frac{1}{\mu_\eps(B_\eps(x))}\bigg|
	\int_\Omega\big(\ind_{T_\eps^{-1}(\V_\eps(x))}(y)-\ind_{B_\eps(x)}(y)\big)(u\circ T_\eps)(y)\ d\mu_\eps(y)
	\bigg|
	\\
	\leq
	~&
	\|u\|_{L^\infty(\X_n)}\frac{\mu_\eps\big(T_\eps^{-1}(\V_\eps(x))\triangle B_\eps(x)\big)}{\mu_\eps(B_\eps(x))}
	\\
	\leq
	~&
	C\|u\|_{L^\infty(\X_n)}\frac{\big|T_\eps^{-1}(\V_\eps(x))\triangle B_\eps(x)\big|}{|B_\eps(x)|},
\end{split}
\end{multline*}
where we have used \eqref{mu-eps-quotient} in the last inequality. Note that, by \eqref{inclusions} with $A=B_\eps(x)$ and by \eqref{ring-estimate},
\begin{align*}
	\big|T_\eps^{-1}(\V_\eps(x))\triangle B_\eps(x)\big|
	\leq
	|B_{\eps+\eps^3}(x)\setminus B_{\eps-\eps^3}(x)|
	\leq
	C\eps^{N+2},
\end{align*}
so we immediately have
\begin{equation*}
	\frac{\big|T_\eps^{-1}(\V_\eps(x))\triangle B_\eps(x)\big|}{|B_\eps(x)|}
	\leq
	C\eps^2.
\end{equation*}
Thus the first term in the right hand side of \eqref{ineq-B} is bounded by $C\|u\|_{L^\infty(\X_n)}\eps^2$ as desired.

For the remaining term, we recall \eqref{card/n-mu(B)-eq2} together with \eqref{mu-eps-mu}, so we obtain the same bound, and putting all these estimates together we finish the proof.
\end{proof}

\begin{lemma}\label{lem:C}
Let $\eps\in(0,\eps_0)$ and assume that the events \eqref{partition} and \eqref{event} are satisfied. There exists a universal constant $C>0$ such that
\begin{equation*}
	\bigg|\dashint_{B_\eps(x)}(u\circ T_\eps)\ d\mu-\dashint_{B_\eps(x)}(u\circ T_\eps)\ d\mu_\eps\bigg|
	\leq
	C\|u\|_{L^\infty(\X_n)}\eps^2
\end{equation*}
for every $x\in\Omega_{-\eps}$.
\end{lemma}

\begin{proof}
Using that $\dens$ and $\dens_\eps$ are the densities of $\mu$ and $\mu_\eps$, respectively, we can write
\begin{multline*}
	\bigg|\dashint_{B_\eps(x)}(u\circ T_\eps)\ d\mu-\dashint_{B_\eps(x)}(u\circ T_\eps)\ d\mu_\eps\bigg|
	\\
\begin{split}
	=
	~&
	\bigg|\int_{B_\eps(x)}(u\circ T_\eps)(y)\bigg(\frac{\dens(y)}{\mu(B_\eps(x))}-\frac{\dens_\eps(y)}{\mu_\eps(B_\eps(x))}\bigg)\ dy\bigg|
	\\
	\leq
	~&
	\|u\|_{L^\infty(\X_n)}\int_{B_\eps(x)}\bigg|\frac{\dens(y)}{\mu(B_\eps(x))}-\frac{\dens_\eps(y)}{\mu_\eps(B_\eps(x))}\bigg|\ dy.
\end{split}
\end{multline*}

Let us focus on the functions inside the integral. If we add and subtract $\frac{\dens_\eps(y)}{\mu(B_\eps(x))}$, then
\begin{multline*}
	\bigg|\frac{\dens(y)}{\mu(B_\eps(x))}-\frac{\dens_\eps(y)}{\mu_\eps(B_\eps(x))}\bigg|
	\\
\begin{split}
	\leq
	~&
	\frac{|\dens(y)-\dens_\eps(y)|}{\mu(B_\eps(x))}
	+
	\dens_\eps(y)\bigg|\frac{1}{\mu(B_\eps(x))}-\frac{1}{\mu_\eps(B_\eps(x))}\bigg|
	\\
	=
	~&
	\frac{1}{\mu(B_\eps(x))}\bigg(|\dens(y)-\dens_\eps(y)|
	+
	\frac{\dens_\eps(y)}{\mu_\eps(B_\eps(x))}\big|\mu_\eps(B_\eps(x))-\mu(B_\eps(x))\big|\bigg).
\end{split}
\end{multline*}
Recalling \eqref{mu-bounds}, \eqref{event} and \eqref{mu-eps-mu} we get
\begin{equation*}
	\bigg|\frac{\dens(y)}{\mu(B_\eps(x))}-\frac{\dens_\eps(y)}{\mu_\eps(B_\eps(x))}\bigg|
	\leq
	\frac{C_0}{\dens_0|B_\eps|}\bigg(1+\frac{\dens_\eps(y)|B_\eps|}{\mu_\eps(B_\eps(x))}\bigg)\eps^2.
\end{equation*}
Hence, integrating with respect to $y$ over $B_\eps(x)$,
\begin{equation*}
	\int_{B_\eps(x)}\bigg|\frac{\dens(y)}{\mu(B_\eps(x))}-\frac{\dens_\eps(y)}{\mu_\eps(B_\eps(x))}\bigg|\ dy
	\leq
	\frac{2C_0}{\dens_0}\eps^2,
\end{equation*}
so the desired inequality follows for any $\eps\in(0,\eps_0)$.
\end{proof}

Together the previous lemmas prove \Cref{thm:discrete-to-nonlocal-averages}.
\begin{proof}[Proof of \Cref{thm:discrete-to-nonlocal-averages}]
First we add and subtract suitable terms
\begin{align*}
	\bigg|&\frac{1}{\card{\V_\eps(T_\eps(x))}}\sum_{Z_i\in\V_\eps(T_\eps(x))}u(Z_i)
	-
	\dashint_{B_\eps(x)}(u\circ T_\eps)\ d\mu\bigg|
	\\ &\leq
	\bigg|\frac{1}{\card{\V_\eps(x)}}\sum_{Z_j\in\V_\eps(x)}u(Z_j)
	-
	\frac{1}{\card{\V_\eps(T_\eps(x))}}\sum_{Z_j\in\V_\eps(T_\eps(x))}u(Z_j)\bigg|
	\\
	&\hspace{1 em}+
	\bigg|\frac{1}{\card{\V_\eps(x)}}\sum_{Z_j\in\V_\eps(x)}u(Z_j)
	-
	\dashint_{B_\eps(x)}(u\circ T_\eps)\ d\mu_\eps\bigg|,
	\\
	&\hspace{1 em}+\bigg|\dashint_{B_\eps(x)}(u\circ T_\eps)\ d\mu
	-
	\dashint_{B_\eps(x)}(u\circ T_\eps)\ d\mu_\eps\bigg|.
\end{align*}
The estimate for the first term on the right is in  \Cref{lem:A}, for the second in \Cref{lem:B} and for the third in  \Cref{lem:C}. 
\end{proof}

\section{Discrete--to--$\varepsilon$-nonlocal inequalities for nonlinear extremal operators}
\label{sec:nonlinear-part}

\subsection{Discrete and $\varepsilon$-nonlocal extremal operators}
One of the inconveniences of the random graph $\X_n$ is its lack of symmetry. The symmetry is a crucial ingredient for the regularity estimates obtained for example to orthogonal grids in  \cite{kuot90}, or in the nonlocal setting of \cite{caffarellis09}. Also more recently with stochastic processes or dynamic programming principles with a fixed step size in \cite{arroyobp23} and \cite{arroyobp22} the symmetry was utilized. On the other hand, Kuo and Trudinger \cite{kuot96} consider a discrete graph not having the mentioned symmetry property, but this is done for a linear operator. One can immediately  observe that also here the $\beta$-term of the operator, which is its linear part, does not require the symmetry considerations.

It can be deduced from \Cref{lem:dens-eps} that, under the event \eqref{partition}, it holds that
\begin{equation*}
	\dist(x,\X_n)
	<
	\eps^3
\end{equation*}
for any $x\in\Omega$. With this in mind, if given any $x,z\in\Omega$, we define $\overline z^x$ as the closest point in $\X_n$ to the reflection of $z$ with respect to $x$, that is, $2x-z$, then
\begin{equation*}
	|\overline z^x-(2x-z)|
	=
	\dist(2x-z,\X_n)
	<
	\eps^3
	<
	\tau\eps^2,
\end{equation*}
where $\tau\geq 1$. In fact, we can approximate $2x-z$ with any vertex in a larger ball $\V_{\tau\eps^2}(2x-z)$.
Using this relaxed notion of symmetry, we can define the second difference of $u$ in $\X_n$ centered at $Z_i$ as
\begin{equation*}
	\delta_{\X_n}u(Z_i,Z_j)
	=
	u(Z_j)+u(\overline Z_j^{Z_i})-2u(Z_i)
\end{equation*}
for every $Z_j\in\X_n$, where $\overline Z_j^{Z_i}$ is as explained above. Moreover, given $\tau\geq 1$, we define the \emph{extremal (maximal and minimal) second differences in $\X_n$} of $u\in L^\infty(\X_n)$ centered at $Z_i$ as
\begin{equation}\label{extremal-delta-Xn}
\begin{split}
	\delta_{\X_n,r}^+u(Z_i,Z_j)
	=
	~&
	u(Z_j)+\max_{Z_k\in\V_{r}(2Z_i-Z_j)}u(Z_k)-2u(Z_i),
	\\
	\delta_{\X_n,r}^-u(Z_i,Z_j)
	=
	~&
	u(Z_j)+\min_{Z_k\in\V_{r}(2Z_i-Z_j)}u(Z_k)-2u(Z_i).
\end{split}
\end{equation}
Then
\begin{equation*}
	\delta_{\X_n,\tau\eps^2}^-u(Z_i,Z_j)
	\leq
	\delta_{\X_n}u(Z_i,Z_j)
	\leq
	\delta_{\X_n,\tau\eps^2}^+u(Z_i,Z_j)
\end{equation*}
in the case that \eqref{partition} holds.

Now we define the extremal operators both in the discrete and the $\varepsilon$-nonlocal settings.

\begin{definition}[Discrete extremal operators]
\label{def:L-Xn}
Let $\Lambda,\tau\geq 1$. We define the \emph{discrete maximal operator} $\L_{\X_n,\Lambda,\tau,\eps}^+$ by
\begin{multline}\label{L-Xn}
	\L_{\X_n,\Lambda,\tau,\eps}^+u(Z_i)
	=
	\frac{1}{\eps^2}\bigg[\alpha\max_{Z_j\in\V_{\Lambda\eps}(Z_i)}\max_{Z_k\in\V_{\tau\eps^2}(2Z_i-Z_j)}\frac{u(Z_j)+u(Z_k)}{2}
	\\
	\hfill
	+
	\beta\frac{1}{\card{\V_\eps(Z_i)}}\sum_{Z_j\in\V_\eps(Z_i)}u(Z_j)-u(Z_i)\bigg]
	\\
	=
	\alpha\max_{Z_j\in\V_{\Lambda\eps}(Z_i)}\frac{\delta_{\X_n,\tau\eps^2}^+u(Z_i,Z_j)}{2\eps^2}
	+
	\frac{\beta}{\card{\V_\eps(Z_i)}}\sum_{Z_j\in\V_\eps(Z_i)}\frac{u(Z_j)-u(Z_i)}{\eps^2}
\end{multline}
for each $u\in L^\infty(\X_n)$ and every $Z_i\in\X_n$. When no confusion arises we use the shorthand $		\L_{\X_n,\eps}^+=\L_{\X_n,\Lambda,\tau,\eps}^+$. The \emph{discrete minimal operator} $\L_{\X_n,\eps}^-=\L_{\X_n,\Lambda,\tau,\eps}^-$ is defined in an analogous way replacing the $\max$ by $\min$.
\end{definition}

Analogously, extremal operators can be defined in the $\eps$-nonlocal setting. To that end, we introduce the following notation: for $r>0$, the \emph{extremal (maximal and minimal) second differences in $\Omega$} of $v\in L^\infty(\Omega)$ centered at $x$ are
\begin{equation}\label{extremal-delta-Omega}
\begin{split}
	\delta_{\Omega,r}^+v(x,z) 
	=
	~&
	v(x+z)+\sup_{|h|<r}v(x-z+h)-2v(x),
	\\
	\delta_{\Omega,r}^-v(x,z)
	=
	~&
	v(x+z)+\inf_{|h|<r}v(x-z+h)-2v(x).
\end{split}
\end{equation}

\begin{definition}[$\eps$-nonlocal extremal operators]
\label{def:L-Omega}
Let $\Lambda,\tau\geq 1$. We define the \emph{maximal $\eps$-nonlocal operator} $\L_{\Omega,\Lambda,\tau,\eps}^+$ by
\begin{multline}\label{L-Omega}
	\L_{\Omega,\Lambda,\tau,\eps}^+v(x)
	=
	\frac{1}{\eps^2}\bigg[\alpha\sup_{|z|<\Lambda\eps}\sup_{|h|<\tau\eps^2}\frac{v(x+z)+v(x-z+h)}{2}
	\\
	\hfill
	+
	\beta\dashint_{B_\eps(x)}v\ d\mu-v(x)\bigg]
	\\
	=
	\alpha\sup_{|z|<\Lambda\eps}\frac{\delta_{\Omega,\tau\eps^2}^+v(x,z)}{2\eps^2}
	+
	\beta\dashint_{B_\eps(x)}\frac{v(y)-v(x)}{\eps^2}\ d\mu(y)
\end{multline}
for each $v\in L^\infty(\Omega)$ and every $x\in\Omega$. When no confusion arises we use the shorthand $		\L_{\Omega,\eps}^+=\L_{\Omega,\Lambda,\tau,\eps}^+$. The \emph{minimal $\eps$-nonlocal operator} $\L_{\Omega,\eps}^-=\L_{\Omega,\Lambda,\tau,\eps}^-$ is defined in an analogous way replacing the $\sup$ by $\inf$.
\end{definition}

\subsection{Discrete--to--$\varepsilon$-nonlocal inequalities for extremal operators and the regularity}

In the next result we show the connection between the maximal discrete and $\varepsilon$-nonlocal operators, which is the core of this work.

\begin{theorem}[Discrete to nonlocal]\label{thm:discrete-to-nonlocal}
Let $\eps\in(0,\eps_0)$. There exists a universal constant $C>0$ such that
\begin{equation*}
	\L_{\X_n,\Lambda,\tau,\eps}^+u(T_\eps(x))
	\leq
	\L_{\Omega,\Lambda+\eps^2,\tau+2\eps,\eps}^+[u\circ T_\eps](x)
	+
	C\|u\|_{L^\infty(\X_n)}
\end{equation*}
for every $x\in\Omega_{-\eps}$ and each $u\in L^\infty(\X_n)$
with probability at least
\begin{equation*}
	1-2n\exp\{-c_0n\eps^{3N+4}\},
\end{equation*}
where $c_0=c_0(\diam\Omega)>0$.
\end{theorem}

\begin{proof}
We assume \eqref{partition} and \eqref{event}.
By \Cref{lem:dens-eps} we have that those events hold with at least the probability in the statement.

By the definition of the discrete maximal operator \eqref{L-Xn} and the estimate for the difference between the arithmetic mean of $u\in L^\infty(\X_n)$ over $\V_\eps(T_\eps(x))$ and the $\mu$-average of its extension $u\circ T_\eps\in L^\infty(\Omega)$ from \Cref{thm:discrete-to-nonlocal-averages}, we immediately have
\begin{align}
	\notag
	\L_{\X_n,\Lambda,\tau,\eps}^+u(T_\eps(x))
	\leq
	~&
	\frac{\alpha}{2}\max_{Z_j\in\V_{\Lambda\eps}(T_\eps(x))}\frac{\delta_{\X_n,\tau\eps^2}^+u(T_\eps(x),Z_j)}{2\eps^2}
	\\
	\label{LXn-ineq1}
	~&
	+
	\beta\dashint_{B_\eps(x)}\frac{(u\circ T_\eps)(y)-(u\circ T_\eps)(x)}{\eps^2}\ d\mu(y)
	+
	C\|u\|_{L^\infty(\X_n)}.
\end{align}

We focus our attention on the $\alpha$-term above. That is, we have to show that 
\begin{equation*}
	\max_{Z_j\in\V_{\Lambda\eps}(T_\eps(x))}\frac{\delta_{\X_n,\tau\eps^2}^+u(T_\eps(x),Z_j)}{2\eps^2}
	\leq
	\sup_{|z|<\Lambda\eps}\frac{\delta_{\Omega,\tau\eps^2}^+(u\circ T_\eps)(x,z)}{2\eps^2}.
\end{equation*}
Since $|T_\eps(x)-x|\leq\eps^3$, we have the inclusion
\begin{equation*}
	\V_{\Lambda\eps}(T_\eps(x))
	\subset
	B_{\Lambda\eps+\eps^3}(x)
	=B_{(\Lambda+\eps^2)\eps}(x)
\end{equation*}
hold for any $\eps\in(0,\eps_0)$. In addition, 
\begin{align*}
	\V_{\tau\eps^2}(2T_\eps(x)-z)
	\subset
	B_{\tau\eps^2+2\eps^3}(2x-z)
	=B_{(\tau+2\eps)\eps^2}(2x-z)
\end{align*}
for any $\eps\in(0,\eps_0)$. Hence, recalling \eqref{extremal-delta-Xn} and \eqref{extremal-delta-Omega} with the fact that $u(Z_j)=(u\circ T_\eps)(Z_j)$ for every $Z_j\in\X_n$ we obtain
\begin{align}
\label{eq:max-from-local-to-nonlocal}
	\nonumber \max_{Z_i\in\V_{\Lambda\eps}(T_\eps(x))}&\max_{Z_j\in\V_{\tau\eps^2}(2T_\eps(x)-Z_i)}\frac{u(Z_i)+u(Z_j)}{2}
	\\
	&\leq
	\sup_{z\in B_{(\Lambda+\eps^2)\eps}(x)}\sup_{\widetilde z\in B_{(\tau+2\eps)\eps^2}(2x-z)}
	\frac{(u\circ T_\eps)(z)+(u\circ T_\eps)(\widetilde z)}{2}
	\\
	&=
	\sup_{|z|<(\Lambda+\eps^2)\eps}\sup_{|h|<(\tau+2\eps)\eps^2}
	\frac{(u\circ T_\eps)(x+z)+(u\circ T_\eps)(x-z+h)}{2}. \nonumber
\end{align}
We finish the proof by inserting this in \eqref{LXn-ineq1} and recalling the definition of the maximal $\eps$-nonlocal operator \eqref{L-Omega}.
\end{proof}

Next we state and prove the main regularity result for the graph functions satisfying the extremal inequalities. The idea is to pass from discrete to $\eps$-nonlocal extremal inequalities by the previous theorem and then use the regularity result for $\eps$-nonlocal extremal inequalities obtained in \Cref{part2}. The given probability arises from the fact that we are dealing with a random data cloud. 

\begin{theorem}[Asymptotic H\"older regularity for $u$ in $\X_n$]\label{thm:main-graph-holder-regularity}
Let $R\in(0,1]$ and fix any $\eps\in(0,R\eps_0)$. If $u\in L^\infty(\V_{2R}(0))$ is a function satisfying
\begin{equation*}
	\L^{+}_{\X_n,\eps}u\geq -\rho,\quad \L^{-}_{\X_n,\eps}u\leq \rho
\end{equation*}
in $\V_{2R}(0)$ for some $\rho>0$, then there exists $\gamma\in(0,1]$ such that
\begin{equation*}
	|u(Z_i)-u(Z_j)|
	\leq
	C(|Z_i-Z_j|^\gamma+\eps^\gamma)
\end{equation*}
for every $Z_i,Z_j\in \V_{R}(0)$ with probability at least
\begin{equation*}
	1-2n\exp\{-c_0n\eps^{3N+4}\},
\end{equation*}
where $C=C(\rho,\|u\|_{L^\infty(\V_{2R}(0)}))>0$ and $c_0=c_0(R)>0$.
\end{theorem}

\begin{proof}
Using \Cref{thm:discrete-to-nonlocal} with $\L_{\Omega,\Lambda+\eps^2,\tau+2\eps,\eps}^+\leq\L_{\Omega,\Lambda+1,\tau+1,\eps}^+$ and the assumptions of the theorem, we have
\begin{equation*}
	-\rho\leq\L_{\X_n,\Lambda,\tau,\eps}^+u(T_\eps(x))
	\leq
	\L_{\Omega,\Lambda+1,\tau+1,\eps}^+[u\circ T_\eps](x)
	+
	C\|u\|_{L^\infty(\X_n)},
\end{equation*}
and similarly for the minimal operator. Thus $u\circ T_\eps$ satisfies the $\eps$-nonlocal extremal inequalities, and thus we may use the H\"older regularity result of  $\eps$-nonlocal operators below in \Cref{Holder}. Since $(u\circ T_\eps)(Z_i)=u(Z_i)$ for $Z_i\in \X_n$, the result immediately follows.
\end{proof}

\part{Asymptotic regularity for functions satisfying the Pucci bounds}\label{part2}

The key point in the previous section was to pass from a graph dependent operator to a graph independent operator with the aid of \Cref{thm:discrete-to-nonlocal}. Now, our goal is to prove asymptotic H\"older regularity for functions satisfying the extremal inequalities
\begin{equation*}
	\L_{\Omega,\eps}^+v\geq-\rho
	\qquad \text{and} \qquad
	\L_{\Omega,\eps}^-v\leq\rho
\end{equation*}
with $\rho>0$.

We follow the same path as in \cite{arroyobp22,arroyobp23}, where a similar result was obtained for a less general operator. Here our maximal and minimal operators have the supremum and the infimum in $h$, respectively, allowing to control the lack of symmetry in the random graph in the previous part.  Also observe that the operators in this paper are with respect to measure $\mu$ in the $\beta$-term instead of the Lebesgue measure compared to the earlier papers. This introduces a drift, and is taken care of in the proof of \Cref{eps-ABP-thm-2} at the end of \Cref{sec:abp}.

A key intermediate result towards
H\"older regularity (\Cref{Holder})
is a decay estimate for the distribution function of $v$, which is obtained as an immediate consequence of \Cref{lem:main}. The proof is based on the measure
estimates \Cref{first} and \Cref{second}, as well as a discrete version of the Calder\'on-Zygmund decomposition, \Cref{CZ} below.

In this part of the paper, our standing assumption is that $v$ is a bounded measurable function. Eventually we apply these results on $v=u \circ T_\eps$, which satisfies the extremal inequalities in a smaller domain as shown in \Cref{thm:discrete-to-nonlocal}. However, for the sake of simplicity and since the results in \Cref{part2} are stated for general functions, we assume directly that $v$ is a bounded measurable function in a bounded domain $\Omega\subset\R^N$.

\section{An $\eps$-ABP estimate for $\L_{\Omega,\eps}^\pm$}
\label{sec:abp}

Let $v$ be a Borel measurable bounded function. 
First we perform an extension of $\Omega$ to a bigger domain $\Omega_{\Lambda\eps+\tau\eps^2}=\{x\in\R^N\,:\,\dist(x,\Omega)<\Lambda\eps+\tau\eps^2\}$ containing $x\in\Omega$ together with both $x+z$ and $x-z+h$ for any $|z|<\Lambda\eps$ and $|h|<\tau\eps^2$. Hence we assume that $v$ is also defined in $\R^N$ instead of $\Omega$, even though it is enough to assume that is defined in $\Omega_{\Lambda\eps+\tau\eps^2}$.

Next we define $\Gamma:\R^N\to\R$, the \emph{concave envelope} of $\hat v:\,=\max\{v,\sup_{\R^N\setminus\Om}v\}$ in $\Omega_{\Lambda\eps+\tau\eps^2}$. To be more precise, 
\begin{equation*}
		\Gamma(x)
		:\,=
		\left\{
		\begin{array}{ll}
			\inf\{p(x)\,:\,\mbox{ for all hyperplanes } p\geq \hat v \mbox{ in } \Omega_{\Lambda\eps+\tau\eps^2}\}
			&
			\mbox{ if } x\in\Omega_{\Lambda\eps+\tau\eps^2},
			\\
			\sup_{\R^N\setminus\Omega}v
			&
			\mbox{ if } x\notin\Omega_{\Lambda\eps+\tau\eps^2}.
		\end{array}
		\right.
\end{equation*}
We call the set
\begin{equation*}
	\nabla\Gamma(x)
	:\,=
	\set{ \xi\in\R^N }{\Gamma(z)\leq\Gamma(x)+\xi\cdot(z-x) \ \mbox{ for all } z\in\Omr}.
\end{equation*}
as the \emph{superdifferential} of $\Gamma$ at $x\in\Omr$. In addition, if $S\subset\Omr$, then we denote $\nabla\Gamma(S)=\bigcup_{x\in S}\nabla\Gamma(x)$. 
Observe that the concavity of $\Gamma$ yields that $\nabla\Gamma(x)\neq\emptyset$ for every $x\in\Omega$.

We define the \emph{contact set} of $v$ as
\begin{equation*}
		K_v
		:\,=
		\Big\{x\in\overline\Omega\,:\,\limsup_{y\to x}\hat v(y)=\Gamma(x)\Big\},
\end{equation*}
which is a non-empty compact subset of $\overline\Omega$. If $v$ is upper semicontinuous then the $\limsup$ can be replaced by the direct evaluation of $\hat v$ at $x$, in which case $K_v$ represents the set of points where $\Gamma$ touches $v$ from above. However $v$ is not necessarily upper semicontinuous and there might not be an actual contact between the graphs of $v$ and $\Gamma$.

Since we are dealing with the $\eps$-nonlocal operators, we use covering arguments instead of infinitesimal ones. For a more detailed explanation, we refer to the beginning of Section 4 in \cite{arroyobp23}. By the compactness of $K_v$, there exist a finite collection of points $\{x_1,\ldots,x_m\}\subset K_v$ such that the family of balls $\{B_{\eps/4}(x_i)\}_{i=1}^m$ covers $K_v$ with finite overlap:
\begin{equation}\label{covering}
	K_v\subset\bigcup_{i=1}^mB_{\eps/4}(x_i).
\end{equation}
Indeed, by Besicovitch's Covering Theorem \cite[Theorem 1.27]{evansg15}, such family can be taken with finite maximum overlap $\ell=\ell(N)\in\N$. Note that $m=m(\eps)\to\infty$ as $\eps\to 0$, but the maximum overlap remains constant, meaning that every $x\in K_v$ is contained in at most $\ell$ balls of the form $B_{\eps/4}(x_i)$, independently of $\eps$. As a direct consequence of this, and since $B_{\eps/4}(x_i)\subset\Omega_{\eps/4}$ for each $i=1,\ldots,m$, we can estimate
\begin{equation}\label{measure-balls}
	\sum_{i=1}^m|B_{\eps/4}(x_i)|
	\leq
	\ell\bigg|\bigcup_{i=1}^mB_{\eps/4}(x_i)\bigg|
	\leq
	\ell|\Omega_{\eps_0}|
	\leq
	C(\diam\Omega+\eps_0)^N
\end{equation}
for any $\eps\in(0,\eps_0)$, where $C>0$ depends only on $N$.
Using this covering, we can state the main result of this section.

\begin{theorem}[$\eps$-ABP estimate with nonsymmetric density]\label{eps-ABP-thm-2}
Let $\eps\in(0,\eps_0)$. For $v:\R^N\to\R$ a Borel measurable bounded function, there exists a finite collection of points $\{x_1,\ldots,x_m\}\subset K_v$ satisfying \eqref{covering} and \eqref{measure-balls} such that
\begin{equation}\label{eps-ABP-2}
	\sup_{\Om} v
	\leq
	\sup_{\R^N\setminus\Om} v
	+
	C 
	\bigg(\sum_{i=1}^m\Big(\sup_{B_{\eps/4}(x_i)}(-\L^+_{\Omega,\eps}v)^+\Big)^N\,|B_{\eps/4}|\bigg)^{1/N},
\end{equation}
where $C=C(\diam\Omega)>0$ is universal.
\end{theorem}

Observe that if $\sup_\Omega v \leq\sup_{\R^N\setminus\Omega}v$ then \eqref{eps-ABP-2} holds trivially. Therefore, without loss of generality by subtracting a constant in what follows we may assume that 
\begin{equation*}
	\sup_{\R^N\setminus\Omega}v=0
	\qquad\text{ and }\qquad
	\sup_\Omega v>0.
\end{equation*}
Then $\hat v=v^+=\max\{v,0\}$.

\subsection{An estimate for $\L^+_{\Omega,\eps}v$ at a contact point}

The next lemma bounds the $\alpha$-term.
It estimates the error caused originally by the fact that points in the data cloud do not necessarily lie symmetrically. In the symmetric case (that is, when $h=0$) one would recover the second differences centered at a point and use the concavity properties to discard that term. 

\begin{lemma}[Estimate for nonsymmetry]\label{lem:contact-point-estimate}
If $x_0\in K_v$, then
\begin{equation*}
	\liminf_{x\to x_0}
	\sup_{|z|<\Lambda\eps}\frac{\delta^+_{\Omega,\tau\eps^2} v(x,z)}{\eps^2}
	\leq
	\tau|\xi|
\end{equation*}
for any $\xi\in\nabla\Gamma(x_0)$, where $\delta_{\Omega,\tau\eps^2}^+$ stands for the maximal second differences defined in \eqref{extremal-delta-Omega}.
\end{lemma}

\begin{proof}
Let any $|z|<\Lambda\eps$. For any $x\in\Om$, since $v\leq\Gamma$ and $\xi\in\nabla\Gamma(x_0)$, we have
\begin{align*}
	\delta_{\Omega,\tau\eps^2}^+ v(x,z)
	=
	~&
	v(x+z)+\sup_{|h|<\tau\eps^2}v(x-z+h)-2v(x)
	\\
	\leq
	~&
	\Gamma(x+z)+\sup_{|h|<\tau\eps^2}\Gamma(x-z+h)-2v(x)
	\\
	\leq
	~&
	\xi\cdot(x+z-x_0)+\sup_{|h|<\tau\eps^2}\left(\xi\cdot(x-z+h-x_0)\right)+2(\Gamma(x_0)-v(x))
	\\
	=
	~&
	\sup_{|h|<\tau\eps^2}\xi\cdot h+2\xi\cdot(x-x_0)+2(\Gamma(x_0)-v(x))
	\\
	\leq
	~&
	\tau|\xi|\eps^2+2\left(|\xi|\,|x-x_0|+\Gamma(x_0)-v(x)\right).
\end{align*}
The proof is concluded after taking the $\liminf$ as $x\to x_0$ and recalling that, since $x_0\in K_v$, then $\limsup_{x\to x_0}v(x)=\Gamma(x_0)$.
\end{proof}

Next, we estimate the $\beta$-term analogously to the previous lemma.

\begin{lemma}\label{integral-estimate}
If $x_0\in K_v$, then
\begin{equation}\label{eq:key-abp-2a}
	\liminf_{x\to x_0}\dashint_{B_\eps(x)}\frac{v(y)-v(x)}{\eps^2}\ d\mu(y)
	=
	\dashint_{B_\eps(x_0)}\frac{v(y)-\Gamma(x_0)}{\eps^2}\ d\mu(y).
\end{equation}
Moreover,
\begin{equation}\label{eq:key-abp-2b}
	\liminf_{x\to x_0}\dashint_{B_\eps(x)}\frac{v(y)-v(x)}{\eps^2}\ d\mu(y)
	\leq
	\frac{\|\nabla\dens\|_\infty}{\dens_0}\,|\xi|
\end{equation}
for any $\xi\in\nabla\Gamma(x_0)$.
\end{lemma}

\begin{proof}
Adding and subtracting terms,
\begin{align*}
\dashint_{B_\eps(x)}\left(v(y)-v(x)\right)\ d\mu(y)
	=
	~&
	\dashint_{B_\eps(x_0)}\left(v(y)-\Gamma(x_0)\right)\ d\mu(y)
	+
	\Gamma(x_0)-v(x)
	\\
	~&
	+
	\dashint_{B_\eps(x)}v(y)\ d\mu(y)-\dashint_{B_\eps(x_0)}v(y)\ d\mu(y).
\end{align*}
Since $x_0 \in K_v$, it holds that $-\limsup_{x\to x_0}v^+(y)=-\Gamma(x_0)$, and thus if we take the $\liminf_{x\to x_0}$ on both sides above, then the term $\Gamma(x_0)-v(x)$ will vanish. Moreover, since $\mu$ is absolutely continuous with respect to Lebesgue measure, the last difference converges to zero as $x\to x_0$, so \eqref{eq:key-abp-2a} follows.

Next, let any $\xi\in\nabla\Gamma(x_0)$. To see \eqref{eq:key-abp-2b}, we use that $v\leq\Gamma$, recall \eqref{eq:comparable}, and that $\dens$ (the density function of $\mu$) is Lipschitz and for simplicity we denote by $\|\nabla\dens\|_\infty$ its Lipschitz constant. It follows that 
\begin{equation}
\label{eq:int-est-proof}
\begin{split}
	\dashint_{B_\eps(x_0)}\frac{v(y)-\Gamma(x_0)}{\eps^2}\ d\mu(y)
	\leq
	~&
	\dashint_{B_\eps(x_0)}\frac{\Gamma(y)-\Gamma(x_0)}{\eps^2}\ d\mu(y)
	\\
	\leq
	~&
	\dashint_{B_\eps(x_0)}\frac{\xi\cdot(y-x_0)}{\eps^2}\ d\mu(y)
	\\
	\leq
	~&
	\frac{1}{\dens_0}\left|\dashint_{B_\eps(x_0)}\frac{\xi\cdot(y-x_0)}{\eps^2}\,\dens(y)\ dy\right|
	\\
	=
	~&
	\frac{1}{\dens_0}\left|\dashint_{B_1}(\xi\cdot y)\,\frac{\dens(x_0+\eps y)-\dens(x_0)}{\eps}\ dy\right|
	\\
	\leq
	~&
	\frac{|\xi|}{\dens_0}\dashint_{B_1}\left|\frac{\dens(x_0+\eps y)-\dens(x_0)}{\eps}\right|\ dy
	\\
	\leq
	~&
	\frac{\|\nabla\dens\|_\infty}{\dens_0}\,|\xi|,
\end{split}
\end{equation}
where in the last equality we added a term integrating to zero, since $B_\eps$ is symmetric and $y\mapsto\xi\cdot y$ is an odd function. Then the proof is finished.
\end{proof}

Combining the previous lemmas we get the following estimate. It tells us that in the contact set we have additional information in terms of concavity. This is quite natural since, roughly speaking, $v$ is `touched' from above by the concave envelope and thus by a supporting hyperplane. However, we had to take into account that $v$ is not necessarily continuous, real touching does not take place, and the equation is not infinitesimal.

\begin{lemma}\label{liminf-estimate}
If $x_0\in K_v$ then
\begin{equation*}
	\liminf_{x\to x_0}\L^+_{\Omega,\eps}v(x)
	\leq
	\Big(\tau+\frac{\|\nabla\dens\|_\infty}{\dens_0}\Big)\,|\xi|
\end{equation*}
for any $\xi\in\nabla\Gamma(x_0)$.
\end{lemma}

\subsection{An estimate for the difference between $v$ and its concave envelope}

Now we employ the previous results to give a quantitative estimate on the the portion of the points in a neighborhood of a contact point where the difference between $v$ and its concave envelope is large.

\begin{lemma}\label{measure-estimate}
If $x_0\in K_v$ and $\xi\in\nabla\Gamma(x_0)$, then
\begin{equation*}
	\frac{\abs{B_\eps(x_0)\cap\braces{\Gamma-v>t\eps^2}}}{|B_\eps(x_0)|}
	\leq
	\frac{1}{t}\cdot\frac{\dens_1}{\beta\dens_0}\left[-\liminf_{x\to x_0}\L^+_{\Omega,\eps}v(x)+\left(\tau+\frac{\|\nabla\dens\|_\infty}{\dens_0}\right)|\xi|\right]
\end{equation*}
for $t>0$.
\end{lemma}

\begin{proof}
Let $t>0$. By concavity we have $\Gamma(y)-v(y)\leq\Gamma(x_0)-v(y)+\xi\cdot(y-x_0)$ and thus
\begin{multline}
\label{eq:int-estimate-proof}
	\frac{\abs{\braces{y\in B_\eps(x_0)\,:\,\Gamma(y)-v(y)>t\eps^2}}}{|B_\eps(x_0)|}
	\\
\begin{split}
	\leq
	~&
	\frac{\abs{\braces{y\in B_\eps(x_0)\,:\,\Gamma(x_0)-v(y)+\xi\cdot(y-x_0)>t\eps^2}}}{|B_\eps(x_0)|}
	\\
	\leq
	~&
	\frac{1}{t\eps^2}\dashint_{B_\eps(x_0)}\pare{\Gamma(x_0)-v(y)+\xi\cdot(y-x_0)}\ dy
	\\
	\leq
	~&
	\frac{1}{t}\cdot\frac{\dens_1}{\dens_0}\dashint_{B_\eps(x_0)}\frac{\Gamma(x_0)-v(y)+\xi\cdot(y-x_0)}{\eps^2}\ d\mu(y)
	\\
	=
	~&
	\frac{1}{t}\cdot\frac{\dens_1}{\dens_0}\braces{-\dashint_{B_\eps(x_0)}\frac{v(y)-\Gamma(x_0)}{\eps^2}\ d\mu(y)+\dashint_{B_\eps(x_0)}\frac{\xi\cdot(y-x_0)}{\eps^2}\ d\mu(y)},
\end{split}
\end{multline}
where in the third inequality we used that the integrals with respect to these two measures are comparable as stated in \eqref{eq:comparable}.

We focus next on the first integral in the right hand side above. By \eqref{eq:key-abp-2a} from \Cref{integral-estimate} and \Cref{lem:contact-point-estimate} we get
\begin{align*}
	-\dashint_{B_\eps(x_0)}\frac{v(y)-\Gamma(x_0)}{\eps^2}\ d\mu(y)
	=
	~&
	-\liminf_{x\to x_0}\dashint_{B_\eps(x)}\frac{v(y)-v(x)}{\eps^2}\ d\mu(y)
	\\
	\leq
	~&
	-\liminf_{x\to x_0}\dashint_{B_\eps(x)}\frac{v(y)-v(x)}{\eps^2}\ d\mu(y)
	\\
	~&
	-\frac{\alpha}{\beta}\bigg[\liminf_{x\to x_0}\sup_{|z|<\Lambda\eps}\frac{\delta_{\Omega,\tau\eps^2}^+ v(x,z)}{2\eps^2}
	-\tau|\xi|\bigg]
	\\
	=
	~& 
	\frac{1}{\beta}\Big(-\liminf_{x\to x_0}\L^+_{\Omega,\eps}v(x)+\tau|\xi|\Big),
\end{align*}
where we used the definition of $\L^+_{\Omega,\eps}v$ in the last line. For the second integral on the last line of \eqref{eq:int-estimate-proof} we use \eqref{eq:int-est-proof}, which gives 
\begin{equation*}
\begin{split}
	\dashint_{B_\eps(x_0)}\frac{\xi\cdot(y-x_0)}{\eps^2}\ d\mu(y)
	\leq
	~&
	\frac{\norm{\nabla\dens}_{L^\infty}}{\dens_0}\,|\xi|.
\end{split}
\end{equation*}
Putting all these together we finish the proof.
\end{proof}

As an immediate corollary we get the following result saying that the difference between $v$ and its concave envelope $\Gamma$ on $B_{\eps/4}$ is sufficiently small in at least half of the points in terms of measure. The choice of the radius $\eps/4$ is justified later and it is due to a geometric argument using the concavity of $\Gamma$.

\begin{corollary}\label{estimate Q}
There exist universal constants $C_1,C_2>0$ such that
\begin{align*}
	\abs{B_{\eps/4}(x_0)\cap\braces{\Gamma-v\leq-C_2\eps^2\liminf_{x\to x_0}\L^+_{\Omega,\eps}v(x)+C_1\eps\sup_\Omega v}}
	\geq
	\frac{1}{2}|B_{\eps/4}(x_0)|
\end{align*}
for each $x_0\in K_v$.
\end{corollary}

\begin{proof}
Let $x_0\in K_v$ and any $\xi\in\nabla\Gamma(x_0)$. Since $\Gamma$ is the concave envelope of $v$ in $\Omega_{\Lambda\eps+\tau\eps^2}$, with $v\leq0$ outside $\Omega$, then the slope of any supporting hyperplane of $v$ in $\Omega_{\Lambda\eps+\tau\eps^2}$ is controlled from above by
\begin{align*}
	|\xi|
	\leq
	\frac{1}{\Lambda\eps+\tau\eps^2}\sup_\Omega v
	\leq
	\frac{1}{\eps}\sup_\Omega v,
\end{align*}
where $\sup_\Omega v>0$ by assumption.
Moreover, if we set 
\begin{align*}
	t
	=
	2\cdot4^N\frac{\dens_1}{\beta\dens_0}\left(-\liminf_{x\to x_0}\L^+_{\Omega,\eps}v(x)+\left(\tau+\frac{\|\nabla\dens\|_\infty}{\dens_0}+1\right)\frac{1}{\eps}\sup_\Omega v\right)
	\\
	\geq
	2\cdot4^N\frac{\dens_1}{\beta\dens_0}\,\frac{1}{\eps}\sup_\Omega v
	>0,
\end{align*}
where in the last line we have used \Cref{liminf-estimate}, then \Cref{measure-estimate} yields
\begin{align*}
	\abs{B_{\eps/4}(x_0)\cap\braces{\Gamma-v>t\eps^2}}
	\leq
	~&
	\abs{B_\eps(x_0)\cap\braces{\Gamma-v>t\eps^2}}
	\\
	\leq
	~&
	\frac{1}{2}\frac{|B_\eps(x_0)|}{4^N}
	=
	\frac{1}{2}|B_{\eps/4}(x_0)|.
\end{align*}
That is,
\begin{align*}
	\abs{B_{\eps/4}(x_0)\cap\braces{\Gamma-v\leq t\eps^2}}
	\geq
	\frac{1}{2}|B_{\eps/4}(x_0)|,
\end{align*}
and the proof is concluded for the chosen value of $t>0$.
\end{proof}

\subsection{Measure estimates for $\nabla\Gamma$ and proof of \Cref{eps-ABP-thm-2}}

The  proof of the following lemma is the same as that of \cite[ Lemma 4.2]{arroyobp23}, since the operator does not appear in the proof. The lemma states that the image $\nabla\Gamma(K_v)$ contains a ball whose radius depends on the supremum of $v$. The idea of the proof is well-known and relies on a geometric intuition: since $\Gamma$ is defined as the infimum of hyperplanes    above $ v^+$, then for each $\xi$, for which $\abs \xi$ is not too big, one can find a supporting hyperplane for $\Gamma$ at $x_0\in K_v$ of the form $a+\xi\cdot z$, that is $\xi \in \nabla\Gamma(x_0)$.  Since $v$ is not necessarily continuous, and `touching' is defined as $\limsup_{y\to x}v^+(y)=\Gamma(x)$, the actual proof \cite[ Lemma 4.2]{arroyobp23} requires more care compared to the usual PDE proof.
\begin{lemma}
\label{lem:inclusion}
Let $v\leq 0$ in $\R^N\setminus\Omega$. Then the inclusion
\begin{equation*}
	B_M
	\subset
	\nabla\Gamma(K_v)
\end{equation*}
holds for 
\begin{equation*}
	M
	=
	\frac{\sup_\Omega v^+}{\diam\Omega+\Lambda\eps+\tau\eps^2}.
\end{equation*}
\end{lemma}

\Cref{lem:inclusion} will be used to derive an estimate for $\sup_\Omega v$ in terms of $\abs{\nabla\Gamma(K_v)}$. Then we estimate  $\abs{\nabla \Gamma(K_v)}$. This is done by using the cover of $K_v$ by balls of radius $\eps/4$ constructed in \eqref{covering} having finite overlap \eqref{measure-balls}. Thus we just need to estimate the diameter of $\nabla\Gamma(B_{\eps/4}(x_0))$ with $x_0\in K_v$. It is natural that this can be estimated in terms of the oscillation of $\Gamma$: roughly, after reducing the situation so that at the center of the ball the slope of a supporting hyperplane is $0$, then the slope of a supporting hyperplane touching somewhere in a ball of radius $\eps/4$ cannot be larger than the oscillation in $\eps/2$ divided by $\eps/2$. Again the operator does not appear in this lemma and the proof can be found in \cite[Lemma 4.3]{arroyobp23}.

\begin{lemma}
\label{lem:inclusion2}
Let  $\Gamma:\Omega_{\Lambda\eps+\tau\eps^2}\to\R$ be a concave bounded function. Then
\begin{equation*}
	\nabla\Gamma(B_{\eps/4}(x_0))
	\subset
	\overline B_R(\xi)
\end{equation*}
for every $x_0\in\overline\Omega$ and $\xi\in\nabla\Gamma(x_0)$, where
\begin{equation}\label{R}
	R
	=
	\frac{2}{\eps}\sup_{y\in B_{\eps/2}(x_0)}\braces{\Gamma(x_0)-\Gamma(y)+\xi\cdot(y-x_0)}.
\end{equation}
\end{lemma}

\begin{lemma}\label{lemma-Gamma}
Let $v\leq 0$ in $\R^N\setminus\Omega$ and $\Gamma:\Omega_{\Lambda\eps+\tau\eps^2}\to\R$ be the concave envelope of $v$. Then
\begin{equation}\label{diam-Gamma}
	\diam\nabla\Gamma(B_{\eps/4}(x_0))
	\leq
	C\left[\sup_{B_{\eps/4}(x_0)}(-\L_{\Omega,\eps}^+v)^++\left(\tau+\frac{\|\nabla\dens\|_\infty}{\dens_0}\right)|\xi|\right]\eps
\end{equation}
and
\begin{equation}\label{measure-Gamma}
	\frac{|\nabla\Gamma(B_{\eps/4}(x_0))|}{|B_{\eps/4}|}
	\leq
	C\left[\sup_{B_{\eps/4}(x_0)}(-\L_{\Omega,\eps}^+v)^++\left(\tau+\frac{\|\nabla\dens\|_\infty}{\dens_0}\right)|\xi|\right]^N
\end{equation}
for every $x_0\in K_v$ and $\xi\in\nabla\Gamma(x_0)$, where $C>0$ is a universal constant.
\end{lemma}

\begin{proof}
We follow \cite[Lemma 4.4]{arroyobp23} with modifications. The idea of the proof is as follows: 
for $x_0\in K_v$ and $\xi\in\nabla\Gamma(x_0)$, using the estimate from \Cref{lem:inclusion2}, it holds that $\nabla\Gamma(B_{\eps/4}(x_0))\subset\overline B_R(\xi)$ with $R$ given in \eqref{R}. Thus
\begin{equation}\label{R1}
	\diam\nabla\Gamma(B_{\eps/4}(x_0))
	\leq
	2R,
\end{equation}
and
\begin{equation}\label{R2}
	\frac{|\nabla\Gamma(B_{\eps/4}(x_0))|}{|B_{\eps/4}|}
	\leq
	\bigg(\frac{4R}{\eps}\bigg)^N
\end{equation}
so our aim is to obtain an estimate of $R$ in terms of the maximal operator and $|\xi|$. In order to do that, let us start by letting any $y\in B_{\eps/2}(x_0)$. We have that $B_{\eps/2}(y)\subset B_\eps(x_0)$, so using \Cref{measure-estimate} with
\begin{align*}
	t
	=
	\frac{2^{N+2}\dens_1}{\beta\dens_0}\left[-\liminf_{x\to x_0}\L^+_{\Omega,\eps}v(x)+\left(\tau+\frac{\|\nabla\dens\|_\infty}{\dens_0}\right)|\xi|\right],
\end{align*}
we have that the portion of the $z$'s in $B_{\eps/2}$ for which
\begin{align*}
	\Gamma(x_0)-v(y+z)+\xi\cdot(y+z-x_0)>t\eps^2
\end{align*}
is less than $1/4$. Recall that $t\geq 0$ by \Cref{liminf-estimate}, we assume $t>0$ in order to apply \Cref{measure-estimate} and will address the case $t=0$ afterward. Hence, the reversed inequality holds in $3$ quarters of $B_{\eps/2}$, and thus there exists at least one $z\in B_{\eps/2}$ for which we simultaneously have
\begin{align*}
	\Gamma(x_0)-\Gamma(y+z)+\xi\cdot(y+z-x_0)
	\leq
	~&
	t\eps^2
	\\
	\text{ and }\quad
	\Gamma(x_0)-\Gamma(y-z)+\xi\cdot(y-z-x_0)
	\leq
	~&
	t\eps^2.
\end{align*}
Finally, by this and concavity
\begin{align*}
	\Gamma(x_0)-\Gamma(y)+\xi\cdot(y-x_0)
	\leq
	\Gamma(x_0)-\frac{\Gamma(y+z)+\Gamma(y-z)}{2}+\xi\cdot(y-x_0)
	\leq
	~&
	t\eps^2,
\end{align*}
and this holds for every $y\in B_{\eps/2}(x_0)$, so
\begin{equation*}
	R
	\leq
	2t\eps
	=
	\frac{2^{N+3}\dens_1}{\beta\dens_0}\left[-\liminf_{x\to x_0}\L^+_{\Omega,\eps}v(x)+\left(\tau+\frac{\|\nabla\dens\|_\infty}{\dens_0}\right)|\xi|\right]\eps.
\end{equation*}
Observe that if $t=0$ we can repeat the argument for any $t>0$ obtaining that $R=0$, so the desired inequality follows immediately.

In addition to this, we can use the rough estimate
\begin{equation*}
	-\liminf_{x\to x_0}\L_{\Omega,\eps}^+v(x)
	\leq
	\sup_{B_{\eps/4}(x_0)}(-\L_{\Omega,\eps}^+v)^+,
\end{equation*}
so that \eqref{diam-Gamma} and \eqref{measure-Gamma} follow immediately after replacing this in \eqref{R1} and \eqref{R2}, respectively.
\end{proof}

Note that the inequality \eqref{measure-Gamma} depends on the norm of a vector $\xi$ in $\nabla\Gamma(x_0)$. However, as it will be made clear later in the proof of \Cref{eps-ABP-thm-2}, the nonlocal nature of the problem forces us to consider superdifferentials other than those in $\nabla\Gamma(x_0)$. More precisely, we later need to estimate the left hand side in \eqref{measure-Gamma} using a superdifferential in the bigger set $\nabla\Gamma(B_{\eps/4}(x_0))$ (in particular, the one minimizing the Euclidean norm). To this end, in the following lemma we use the control on the diameter of $\nabla\Gamma(B_{\eps/4}(x_0))$ stated in \eqref{diam-Gamma} to estimate the difference between $\xi\in\nabla\Gamma(x_0)$ and any other vector $\zeta\in\nabla\Gamma(B_{\eps/4}(x_0))$.

\begin{lemma}\label{lemma-Gamma-2}
Let $v\leq 0$ in $\R^N\setminus\Omega$ and $\Gamma:\Omega_{\Lambda\eps+\tau\eps^2}\to\R$ be the concave envelope of $v$. For $x_0\in K_v$, let $\zeta$ be a vector minimizing the norm among all vectors in the closure of $\nabla\Gamma(B_{\eps/4}(x_0))$. Then
\begin{equation}\label{measure-Gamma-2}
	\frac{|\nabla\Gamma(B_{\eps/4}(x_0))|}{|B_{\eps/4}|}
	\leq
	C\bigg[\Big(\sup_{B_{\eps/4}(x_0)}(-\L_{\Omega,\eps}^+v)^+\Big)^N+|\zeta|^N\bigg]
\end{equation}
for some universal constant $C>0$.
\end{lemma}

\begin{proof}
Let $\xi\in\nabla\Gamma(x_0)$ and use the inequality
\begin{equation}\label{diam-bound}
	|\xi|
	\leq
	|\zeta|+|\xi-\zeta|
	\leq
	|\zeta|+\diam\nabla\Gamma(B_{\eps/4}(x_0))
\end{equation}
in \eqref{diam-Gamma} to get the estimate
\begin{equation*}
	\diam\nabla\Gamma(B_{\eps/4}(x_0))
	\leq
	C\eps_0\bigg[\sup_{B_{\eps/4}(x_0)}(-\L_{\Omega,\eps}^+v)^++|\zeta|+\diam\nabla\Gamma(B_{\eps/4}(x_0))\bigg],
\end{equation*}
where we used that $\tau\geq 1$ and also that $\eps\in(0,\eps_0)$. Rearranging terms we get
\begin{equation*}
	(1-C\eps_0)\diam\nabla\Gamma(B_{\eps/4}(x_0))
	\leq
	C\eps_0\bigg[\sup_{B_{\eps/4}(x_0)}(-\L_{\Omega,\eps}^+v)^++|\zeta|\bigg],
\end{equation*}
so by selecting small enough $\eps_0$ we ensure that the left hand side is strictly positive. More precisely, if we impose
\begin{equation}\label{eps0-bound-c}
	\eps_0
	\leq
	\frac{1}{2C},
\end{equation}
then
\begin{equation*}
	\diam\nabla\Gamma(B_{\eps/4}(x_0))
	\leq
	C\eps_0\bigg[\sup_{B_{\eps/4}(x_0)}(-\L_{\Omega,\eps}^+v)^++|\zeta|\bigg].
\end{equation*}
Inserting this into \eqref{diam-bound} we obtain
\begin{equation*}
	|\xi|
	\leq
	C\bigg[\sup_{B_{\eps/4}(x_0)}(-\L_{\Omega,\eps}^+v)^++|\zeta|\bigg]
\end{equation*}
for some universal constant $C>0$. Finally, using this in \eqref{measure-Gamma} and applying the elementary inequality $(a+b)^N\leq2^{N-1}(a^N+b^N)$ for every $a,b\geq0$ we reach the desired estimate \eqref{measure-Gamma-2}.
\end{proof}

Now we have all the tools in order to prove the main theorem of the section.

\begin{proof}[Proof of \Cref{eps-ABP-thm-2}.]
We denote by $\sigma$ the $(n-1)$-dimensional Lebesgue measure. Let $M\geq 0$ and $\eta>0$. Then
\begin{equation}\label{eq-04}
\begin{split}
	\log\pare{\frac{M^N+\eta^N}{\eta^N}}
	=
	~&
	\int_0^M\frac{Nt^{N-1}}{t^N+\eta^N}\ dt
	\\
	=
	~&
	\frac{1}{|B_1|}
	\int_0^M\frac{\sigma(\partial B_t)}{t^N+\eta^N}\ dt
	\\
	=
	~&
	\frac{1}{|B_1|}
	\int_0^M\int_{\partial B_t}\frac{1}{|z|^N+\eta^N}\ d\sigma(z)\ dt
	\\
	=
	~&
	\frac{1}{|B_1|}
	\int_{B_M}\frac{1}{|z|^N+\eta^N}\ dz,
\end{split}
\end{equation}
where we have used that $\sigma(\partial B_t)=\sigma(\partial B_1)t^{N-1}=N|B_1|t^{N-1}$, which follows by integrating $\sigma(\partial B_t)$ to get $|B_1|$. Our aim is to show that this integral is bounded by a constant $C>0$ independent of $\eps>0$, from which it follows that
\begin{align}
\label{eq:M-bound}
	M
	\leq
	(e^C-1)^{1/N}\eta,
\end{align}
and select $M$ in such a way that the desired estimate follows from this.
To be more precise, we select 
\begin{equation*}
	M
	= 
	\frac{\sup_\Omega v^+}{\diam\Omega+\Lambda\eps+\tau\eps^2}.
\end{equation*}
Thus we start estimating the right hand side of  \eqref{eq-04}. The value of $M$ allows us to use \Cref{lem:inclusion}, and  we have $B_M \subset \nabla\Gamma(K_v)$, so
\begin{align}
\label{eq-05}
	\int_{B_M}\frac{1}{|z|^N+\eta^N}\ dz
	\leq
	~&
	\int_{\nabla\Gamma(K_v)}\frac{1}{|z|^N+\eta^N}\ dz\nonumber 
	\\
	\leq
	~&
	\sum_{i=1}^m\int_{\nabla\Gamma(B_{\eps/4}(x_i))}\frac{1}{|z|^N+\eta^N}\ dz,
\end{align}
where we use the covering in \eqref{covering}.

In order to make a rough estimate, let us select $\zeta_i$ in the closure of $\nabla\Gamma(B_{\eps/4}(x_i))$ minimizing $|z|^N$. Thus we obtain
\begin{equation*}
	\int_{\nabla\Gamma(B_{\eps/4}(x_i))}\frac{1}{|z|^N+\eta^N}\ dz
	\leq
	\frac{|\nabla\Gamma(B_{\eps/4}(x_i))|}{|\zeta_i|^N+\eta^N}.
\end{equation*}
We recall \Cref{lemma-Gamma-2} to estimate the right hand side and then we use it in \eqref{eq-05} and further in \eqref{eq-04} to get
\begin{equation*}
\begin{split}
	\log\pare{\frac{M^N+\eta^N}{\eta^N}}
	\leq
	~&
	C\sum_{i=1}^m\frac{
	|\zeta_i|^N+\Big(\displaystyle\sup_{B_{\eps/4}(x_i)}(-\L_{\Omega,\eps}^+v)^+\Big)^N}{|\zeta_i|^N+\eta^N}|B_{\eps/4}|
	\\
	\leq
	~&
	C\bigg(\sum_{i=1}^m|B_{\eps/4}|+\frac{1}{\eta^N}\sum_{i=1}^m\Big(\sup_{B_{\eps/4}(x_i)}(-\L^+_{\Omega,\eps}v)^+\Big)^N|B_{\eps/4}|\bigg).
\end{split}
\end{equation*}
Suppose now that the following quantity is strictly positive
\begin{equation}\label{mu}
	\eta
	:=
	\bigg(\sum_{i=1}^m\Big(\sup_{B_{\eps/4}(x_i)}(-\L^+_{\Omega,\eps}v)^+\Big)^N|B_{\eps/4}|\bigg)^{1/N},
\end{equation}
where $x_1,\ldots,x_m\in K_v$ are such that \eqref{covering} and \eqref{measure-balls} hold. Then by using this choice in the previous estimate we obtain 
\begin{align*}
	\log\pare{\frac{M^N+\eta^N}{\eta^N}}
	\leq 	C\big((\diam\Omega+\eps_0)^N+1\big),
\end{align*}
and we can finish the proof as explained before \eqref{eq:M-bound}. Otherwise, if the quantity in \eqref{mu} equals zero, then we have
\begin{align*}
\log\pare{\frac{M^N+\eta^N}{\eta^N}}
	\leq C \sum_{i=1}^m|B_{\eps/4}|\leq C\big((\diam\Omega+\eps_0)^N.
\end{align*}
Thus we get \eqref{eq:M-bound} for any $\eta>0$ and letting $\eta\to 0$ we again obtain the result.
\end{proof}

\section{A barrier function for $\L_{\Om,\eps}^-$}
\label{sec:barrier}

In what follows we construct an auxiliary function called barrier function for example by Cabr\'e and Caffarelli, which will be used in the proof of \Cref{first}.
The shape of this barrier function ensures that when added to a subsolution $v$ the assumptions of the $\eps$-ABP estimate (\Cref{eps-ABP-thm-2}) are fulfilled. We will also use the barrier in \Cref{lem:L+barrier}, which is why we perform the computation within a ball of general radius.

In order to construct a proper explicit subsolution for
\begin{align*}
	\L_{\Om,\eps}^-\Psi(x)
	=
	~&
	\alpha\inf_{|z|<\Lambda\eps}\frac{\delta_{\Omega,\tau\eps^2}^-\Psi(x,z)}{2\eps^2}
	+
	\beta\dashint_{B_\eps(x)}\frac{\Psi(y)-\Psi(x)}{\eps^2}\ d\mu(y).
\end{align*}
where
\begin{equation*}
	\delta_{\Omega,\tau\eps^2}^-\Psi(x,z)
	=
	\Psi(x+z)+\inf_{|h|<\tau\eps^2}\Psi(x-z+h)-2\Psi(x),
\end{equation*}
we first need to show the following technical lemmas which will be used later in the main \Cref{barrier-0,barrier}. We are using the same barrier function as in \cite{arroyobp22}, but we need to account for the loss of symmetry due to the density function $\dens$ and the presence of the parameter $h$ in the nonlinear part of the operator.

First, we state an auxilary result taken from \cite[Lemma 3.4]{arroyobp22}, whose proof is based on estimating Taylor's expansion and using concavity properties.

\begin{lemma}
Let $\sigma>0$. If $a,b>0$ and $c\in\R$ are such that
\begin{align*}
	2(\sigma+2)b
	\leq
	a
	\qquad\text{ and }\qquad
	|c|<a+b,
\end{align*}
then
\begin{equation}\label{ineq:abc}
	(a+b+c)^{-\sigma}+(a+b-c)^{-\sigma}-2a^{-\sigma}
	\geq
	2\sigma a^{-\sigma-1}\Big(-b+(\sigma+1)\frac{c^2}{4a}\Big).
\end{equation}
\end{lemma}

Next we use the previous lemma to obtain a lower estimate for the second differences of the barrier function similarly as in \cite{arroyobp22} taking into account the error caused by the non symmetry. It is worth to mention that this result introduces an extra dependence of $\eps_0$ on the exponent $\sigma$ and a radius $R>0$. This is not problematic since later in \Cref{barrier} the exponent $\sigma$ and the radius $R$ are selected according to the universal constants, see \Cref{remarkeps0}.
We introduce the following shorthand notation for simplicity,
\begin{equation}\label{delta}
	\delta v(x,z;h)=v(x+z)+v(x-z+h)-2v(x).
\end{equation}

\begin{lemma}
For a fixed $\sigma>0$, let $\varphi(x)=(1+|x|^2)^{-\sigma}$ for $x\in\R^N$. Let $\eps\in(0,\eps_0)$ with $\eps_0=\eps_0(R)$ satisfying \eqref{eps0-bound-1}. There exists a universal constant $C_1>0$ such that
\begin{equation}\label{estimate-Psi}
	\frac{\delta\varphi(x,\eps z;\eps^2 h)}{2\eps^2}
	\geq
	\sigma\,\frac{\varphi(x)}{1+|x|^2}\bigg[-C_1(R+1)+(\sigma+1)\frac{|x|^2}{1+|x|^2}\left(\frac{x}{|x|}\cdot z\right)^2\bigg]
\end{equation}
for every $|x|\leq R$, $|z|<\Lambda$ and $|h|<\tau$. Moreover, as an immediate consequence,
\begin{equation}\label{estimate-inf-Psi}
	\inf_{|z|<\Lambda\eps}\frac{\delta_{\Omega,\tau\eps^2}^-\varphi(x,z)}{2\eps^2}
	\geq
	\sigma \varphi(x)[-C_1(R+1)]
\end{equation}
for every $|x|\leq R$.
\end{lemma}

\begin{proof}
For any $\eps\in(0,\eps_0)$ we obtain the following estimate by direct evaluation,
\begin{align*}
	1+|x\pm\eps z+\eps^2 h|^2
	=
	~&
	1+|x|^2+\eps^2|z|^2+\eps^4|h|^2\pm 2\eps\,x\cdot z+2\eps^2 x\cdot h\pm2\eps^3z\cdot h
	\\
	\leq
	~&
	1+|x|^2+((|z|+|h|)^2+2|x||h|)\eps^2\pm2\eps\,x\cdot z
	\\
	\leq
	~&
	1+|x|^2+((\Lambda+\tau)^2+2R\tau)\eps^2\pm2\eps\,x\cdot z
	\\
	=
	~&
	1+|x|^2+C_1(R+1)\eps^2\pm2\eps\,x\cdot z
\end{align*}
for every $|x|\leq R$, $|z|<\Lambda$ and $|h|<\tau$, where $C_1>0$. So, since $\sigma>0$, using the previous estimate for $1+|x+\eps z|^2$ (with $h=0$) and $1+|x-\eps z+\eps^2 h|^2$ we obtain
\begin{align*}
	\delta\varphi(x,\eps z;\eps^2 h)
	=
	~&
	\varphi(x+\eps z)+\varphi(x-\eps z+\eps^2h)-2\varphi(x)
	\\
	\geq
	~&
	(1+|x|^2+C_1(R+1)\eps^2+2\eps\,x\cdot z)^{-\sigma}
	\\
	~&
	+(1+|x|^2+C_1(R+1)\eps^2-2\eps\,x\cdot z)^{-\sigma}
	\\
	~&
	-2(1+|x|^2)^{-\sigma}.
\end{align*}

Next we recall inequality \eqref{ineq:abc} with
\begin{align*}
	a
	=
	1+|x|^2,
	\qquad
	b
	=
	C_1(R+1)\eps^2
	\qquad\text{ and }\qquad
	c
	=
	2\eps x\cdot z.
\end{align*}
Indeed, choosing
\begin{align}\label{eps0-bound-1}
	\eps_0
	\leq
	\frac{1}{\sqrt{2(\sigma+2)C_1(R+1)}\,},
\end{align}
it holds that $2(\sigma+2)b\leq a$ and $|c|<a+b$ for any $\eps\in(0,\eps_0)$, so the inequality \eqref{ineq:abc} yields \eqref{estimate-Psi}. Furthermore, bounding the second term in brackets in \eqref{estimate-Psi} by $0$ we immediately get that \eqref{estimate-inf-Psi} holds.
\end{proof}

In the next lemma is where the lack of symmetry caused by the density function $\dens$ has more impact causing an extra negative term depending on $\eps$ and $\sigma$.

\begin{lemma}
For a fixed $\sigma>0$, let $\varphi(x)=(1+|x|^2)^{-\sigma}$ for $x\in\R^N$. Let $\eps\in(0,\eps_0)$ with $\eps_0=\eps_0(R)$ satisfying \eqref{eps0-bound-1}. There exist universal constants $C_2,C_3>0$ such that
\begin{equation}\label{estimate-integral-Psi}
	\dashint_{B_\eps(x)}\frac{\varphi(y)-\varphi(x)}{\eps^2}\ d\mu(y)
	\geq
	\sigma\varphi(x)\bigg[C_2\frac{\sigma+1}{R^2+1}\,\frac{|x|^2}{1+|x|^2}-C_3(R+1)(1-\eps)^{-\sigma-1}\bigg]
\end{equation}
for every $|x|\leq R$.
\end{lemma}

\begin{proof}
Since $\dens$ is the density function of the measure $\mu$, adding and subsctracting we get
\begin{equation}\label{estimate-integral-0-Psi}
\begin{split}
	\dashint_{B_\eps(x)}\frac{\varphi(y)-\varphi(x)}{\eps^2}\ d\mu(y)
	=
	~&
	\frac{|B_\eps|}{\mu(B_\eps(x))}\dashint_{B_\eps(x)}\frac{\varphi(y)-\varphi(x)}{\eps^2}\,\dens(y)\ dy
	\\
	=
	~&
	\frac{|B_\eps|}{\mu(B_\eps(x))}\bigg[\dens(x)\dashint_{B_\eps(x)}\frac{\varphi(y)-\varphi(x)}{\eps^2}\ dy
	\\
	~&
	+\dashint_{B_\eps(x)}\frac{\varphi(y)-\varphi(x)}{\eps}\cdot\frac{\dens(y)-\dens(x)}{\eps}\ dy\bigg].
\end{split}
\end{equation}
We study the two integrals in the right-hand side separately. For the first integral we use the symmetry of the ball to write the integrand using the second differences of $\varphi$, that is
\begin{multline}\label{estimate-integral-1-Psi}
	\dashint_{B_\eps(x)}\frac{\varphi(y)-\varphi(x)}{\eps^2}\ dy
	\\
\begin{split}
	=
	~&
	\dashint_{B_1}\frac{\delta\varphi(x,\eps y;0)}{2\eps^2}\ dy
	\\
	\geq
	~&
	\sigma\,\frac{\varphi(x)}{1+|x|^2}\bigg[-C_1(R+1)+(\sigma+1)\frac{|x|^2}{1+|x|^2}\dashint_{B_1}\Big(\frac{x}{|x|}\cdot y\Big)^2\ dy\bigg]
	\\
	\geq
	~&
	\sigma\,\frac{\varphi(x)}{1+|x|^2}\bigg[-C_1 (R+1)+\frac{\sigma+1}{N+2}\cdot\frac{|x|^2}{1+|x|^2}\bigg],
\end{split}
\end{multline}
where we have used \eqref{estimate-Psi} with $h=0$ in order to recover the second differences of $\varphi$ centered at $x$. Also, we have used that the averaged integral of $y_1^2$ over $B_1$ is equal to $1/(N+2)$. 

For the other integral, we estimate
\begin{align}\label{ineq:dens-gradient}
	\dashint_{B_\eps(x)}\frac{\varphi(y)-\varphi(x)}{\eps}\frac{\dens(y)-\dens(x)}{\eps}\ dy
	\geq
	~&
	-\|\nabla\dens\|_\infty\dashint_{B_\eps(x)}\left|\frac{\varphi(y)-\varphi(x)}{\eps}\right|\ dy\nonumber
	\\
	=
	~&
	-\|\nabla\dens\|_\infty\dashint_{B_1}\left|\frac{\varphi(x+\eps y)-\varphi(x)}{\eps}\right|\ dy.
\end{align}

By the Mean Value Theorem, for $|x|\leq R$ and $|y|<1$, there exists $\lambda\in[0,1]$ such that
\begin{equation*}
	\Big|\frac{\varphi(x+\eps y)-\varphi(x)}{|x+\eps y|^2-|x|^2}\Big|
	=
	\sigma(1+|x+\lambda\eps y|^2)^{-\sigma-1}
	\leq
	\sigma(1+|x|^2)^{-\sigma-1}(1-\eps)^{-\sigma-1},
\end{equation*}
where in the inequality we have estimated
\begin{align*}
	1+|x+\lambda\eps y|^2
	\geq
	~&
	1+|x|^2+|\lambda\eps y|^2-2|x||\lambda\eps y|
	\\
	\geq
	~&
	1+|x|^2-2|x|\eps
	\\
	\geq
	~&
	1+|x|^2-\eps(1+|x|^2)
	\\
	=
	~&
	(1+|x|^2)(1-\eps)
\end{align*}
for every $|y|<1$ and any $\lambda\in[0,1]$, which is bounded from below by $1-\eps_0>0$ for any $\eps\in(0,\eps_0)$. On the other hand, taking into account that $\big||x+\eps y|^2-|x|^2\big|=|2\eps x\cdot y+\eps^2|y|^2|\leq(2|x|+1)\eps$ for every $|y|<1$ and $\eps\in(0,\eps_0)$ we get
\begin{align*}
	\Big|\frac{\varphi(x+\eps y)-\varphi(x)}{\eps}\Big|
	\leq
	~&
	\sigma(1+|x|^2)^{-\sigma-1}(1-\eps)^{-\sigma-1}\Big|\frac{|x+\eps y|^2-|x|^2}{\eps}\Big|
	\\
	\leq
	~&
	\sigma(1+|x|^2)^{-\sigma-1}(1-\eps)^{-\sigma-1}(2|x|+1)
	\\
	=
	~&
	\sigma\varphi(x)(1-\eps)^{-\sigma-1}\,\frac{2|x|+1}{1+|x|^2}
	\\
	\leq
	~&
	2\sigma\varphi(x)(1-\eps)^{-\sigma-1}.
\end{align*}
Replacing this estimate in \eqref{ineq:dens-gradient},
\begin{align}\label{estimate-integral-2-Psi}
	\dashint_{B_\eps(x)}\frac{\varphi(y)-\varphi(x)}{\eps}\frac{\dens(y)-\dens(x)}{\eps}\ dy
	\geq
	-2\sigma\varphi(x)\|\nabla\dens\|_\infty(1-\eps)^{-\sigma-1}.
\end{align}

Combining \eqref{estimate-integral-1-Psi} and \eqref{estimate-integral-2-Psi} in \eqref{estimate-integral-0-Psi} and using that $\dens_0\leq\dens(x)\leq\dens_1$ and $\dens_0|B_\eps|\leq\mu(B_\eps(x))\leq\dens_1|B_\eps|$, we get
\begin{multline*}
	\dashint_{B_\eps(x)}\frac{\varphi(y)-\varphi(x)}{\eps^2}\ d\mu(y)
	\\
\begin{split}
	~&\geq
	\sigma\varphi(x)\frac{|B_\eps|}{\mu(B_\eps(x))}\bigg[
	\frac{\dens(x)}{1+|x|^2}\cdot\frac{\sigma+1}{N+2}\cdot\frac{|x|^2}{1+|x|^2}
	\\
	~&\hspace{85pt}
	-\frac{\dens(x)}{1+|x|^2}C_1(R+1)-2\|\nabla\dens\|_\infty(1-\eps)^{-\sigma-1}
	\bigg]
	\\
	~&
	\geq
	\sigma \varphi(x)\bigg[
	\frac{\dens_0}{\dens_1}\cdot\frac{1}{1+|x|^2}\cdot\frac{\sigma+1}{N+2}\cdot\frac{|x|^2}{1+|x|^2}
	\\
	~&\hspace{80pt}
	-\frac{1}{\dens_0}\bigg(\frac{C_1\dens_1(R+1)}{1+|x|^2}+2\|\nabla\dens\|_\infty(1-\eps)^{-\sigma-1}\bigg)
	\bigg]
\end{split}
\end{multline*}
for every $|x|\leq R$ and any $\eps\in(0,\eps_0)$. 
Using that $|x|\leq R$ by assumption, we can estimate
\begin{equation*}
	\frac{1}{1+|x|^2}
	\geq
	\frac{1}{R^2+1}
\end{equation*}
in the positive term in brackets, while we can bound the negative terms in the following way,
\begin{multline*}
	\frac{C_1\dens_1(R+1)}{1+|x|^2}+2\|\nabla\dens\|_\infty(1-\eps)^{-\sigma-1}
	\\
\begin{split}
	=
	~&
	\bigg(\frac{C_1\dens_1}{1+|x|^2}(1-\eps)^{\sigma+1}+\frac{2\|\nabla\dens\|_\infty}{R+1}\bigg)(R+1)(1-\eps)^{-\sigma-1}
	\\
	\leq
	~&
	(C_1\dens_1+2\|\nabla\dens\|_\infty)(R+1)(1-\eps)^{-\sigma-1}
\end{split}
\end{multline*}
for $\eps\in(0,\eps_0)$.
Thus
\begin{align*}
	\dashint_{B_\eps(x)}\frac{\varphi(y)-\varphi(x)}{\eps^2}\ d\mu(y)
	\geq
	~&
	\sigma \varphi(x)\bigg[
	\frac{1}{R^2+1}\cdot\frac{\dens_0}{\dens_1}\cdot\frac{\sigma+1}{N+2}\cdot\frac{|x|^2}{1+|x|^2}
	\\
	~&\hspace{30pt}
	-\frac{1}{\dens_0}(C_1\dens_1+2\|\nabla\dens\|_\infty)(R+1)(1-\eps)^{-\sigma-1}
	\bigg]
\end{align*}
for every $|x|\leq R$ and any $\eps\in(0,\eps_0)$. Hence \eqref{estimate-integral-Psi} follows after letting $C_2=\frac{\dens_0}{(N+2)\dens_1}$ and $C_3=\frac{C_1\dens_1+2\|\nabla\dens\|_\infty}{\dens_0}$.

\end{proof}

In the next result we combine the previous lemmas to derive a lower bound for the minimal operator applied to $\varphi$ far away from zero.

\begin{lemma}\label{barrier-0}
For a fixed $\sigma>0$, let $\varphi(x)=(1+|x|^2)^{-\sigma}$ for $x\in\R^N$. Let $\eps\in(0,\eps_0)$ with $\eps_0=\eps_0(R)$ satisfying \eqref{eps0-bound-1}. There exist universal constants $C_1,C_2,C_3>0$ such that
\begin{equation}\label{estimate-L-Psi}
	\L_{\Om,\eps}^-\varphi(x)
	\geq
	\sigma\varphi(x)\bigg[
	\frac{\beta C_2}{R^2+1}\,\frac{|x|^2}{1+|x|^2}(\sigma+1)
	-(C_1+C_3)(R+1)(1-\eps)^{-\sigma-1}
	\bigg]
\end{equation}
for every $|x|\leq R$. Furthermore, if $0<r<R$, there exists $\sigma_0=\sigma_0(r,R)$ such that for every $\sigma>\sigma_0$ it holds that
\begin{equation}\label{bound-L-Psi}
	\L_{\Omega,\eps}^-\varphi(x)
	\geq
	\sigma\varphi(x)
\end{equation}
for every $r\leq|x|\leq R$ and $\eps\in(0,\eps_0)$, for certain $\eps_0=\eps_0(r,R,\sigma)>0$.
\end{lemma}

\begin{proof}
The inequality \eqref{estimate-L-Psi} is an immediate consequence of \eqref{estimate-inf-Psi} and \eqref{estimate-integral-Psi}. Indeed, replacing in the definition of $\L_{\Om,\eps}^-\Psi(x)$ we have
\begin{align*}
	\L_{\Om,\eps}^-\varphi(x)
	\geq
	~&
	\sigma\varphi(x)\bigg[
	-\alpha C_1(R+1)
	\\
	~&\hspace{30pt}
	+\beta C_2\frac{\sigma+1}{R^2+1}\,\frac{|x|^2}{1+|x|^2}
	-\beta C_3(R+1)(1-\eps)^{-\sigma-1}
	\bigg]
\end{align*}
for every $|x|\leq R$, so \eqref{estimate-L-Psi} follows after using that $\beta=1-\alpha\in(0,1]$ and $(1-\eps)^{-\sigma-1}\geq 1$.

To show \eqref{bound-L-Psi}, we use the positive term inside the brackets in \eqref{estimate-L-Psi} to compensate the negativity of the other term. 
To do that, for clarity we state the following elementary fact. If $a,b>0$ and we choose any $\sigma>\frac{1+b}{a}$, then it is possible to select $\eps_0=\eps_0(a,b,\sigma)$ satisfying
\begin{equation*}
	0
	<
	\eps_0
	<
	1-\bigg(\frac{1+b}{a\sigma}\bigg)^{1/\sigma},
\end{equation*}
and so it holds that
\begin{equation*}
	a\sigma-b(1-\eps)^{-\sigma}
	>
	(1-\eps)^{-\sigma}
	>
	1
\end{equation*}
for every $\eps\in(0,\eps_0)$. Observe that $\frac{|x|^2}{1+|x|^2}\geq \frac{r^2}{1+r^2}$. Using this with
\begin{equation*}
	a
	=
	\frac{\beta C_2}{R^2+1}\,\frac{r^2}{1+r^2}
	\qquad
	\text{and}
	\qquad
	b
	=
	(C_1+C_3)(R+1),
\end{equation*}
we can fix
\begin{equation}\label{sigma-bound-r-R}
	\sigma_0
	=
	\sigma_0(r,R)
	=
	\frac{1+b}{a},
\end{equation}
such that for any $\sigma>\sigma_0$ we can then select $\eps_0=\eps_0(r,R,\sigma)$ satisfying
\begin{equation}\label{eps-bound-r-R}
	0
	<
	\eps_0
	<
	1-\bigg(\frac{1+b}{a\sigma}\bigg)^{1/\sigma},
\end{equation}
and hence \eqref{bound-L-Psi} follows for $r\leq|x|\leq R$ and $\eps\in(0,\eps_0)$.
\end{proof}

Finally we state the main result of this section.

\begin{lemma}\label{barrier}
There exist smooth radial functions $\Psi,\psi:\R^N\to\R$ (depending only on universal constants) such that
\begin{equation*}
	\begin{cases}
	\L_{\Om,\eps}^-\Psi+\psi\geq 0 & \text{ in } \R^N,
	\\
	\Psi\geq 2 & \text{ in } Q_3,
	\\
	\Psi\leq 0 & \text{ in } \R^N\setminus B_{\frac{5}{2}\sqrt{N}},
	\end{cases}
	\qquad\text{ and }\qquad
	\begin{cases}
	\psi\leq\psi(0) & \text{ in } \R^N,
	\\
	\psi\leq 0 & \text{ in } \R^N\setminus B_{1/4},
	\end{cases}
\end{equation*}
for every $\eps\in(0,\eps_0)$, where $Q_3$ denotes the cube in $\R^N$ of side-length $3$ centered at the origin and sides parallel to the coordinate axis.
\end{lemma}

\begin{proof} 
Let $\sigma>0$ be a fixed constant to be determined later. Let
\begin{equation*}
	\Psi(x)
	=
	A\varphi(x)-B
\end{equation*}
for $x\in\R^N$, where $A,B>0$ are fixed constants depending on $\sigma$ such that $\Psi\geq2$ in $B_{\frac{3}{2}\sqrt{N}}\supset Q_3$ and $\Psi\leq0$ in $\R^N\setminus B_{\frac{5}{2}\sqrt{N}}$.
We show that $\L_{\Om,\eps}^-\Psi$ can be bounded from below by a function $-\psi$, where $\psi:\R^N\to\R$ is a smooth function in $\R^N$ having the desired properties, for which both $\sigma$ and $\eps_0$ have to be chosen conveniently.

We apply \Cref{barrier-0} with $R=\frac{5}{2}\sqrt{N}$. In particular, by \eqref{estimate-L-Psi} with $\eps_0$ satisfying \eqref{eps0-bound-1}, we can  define
\begin{align*}
	\psi(x)
	:\,=
	A\sigma \varphi(x)\bigg[
	(C_1+C_3)(1-\eps)^{-\sigma-1}-\beta C_2(\sigma+1)\frac{|x|^2}{1+|x|^2}
	\bigg]
\end{align*}
for every $x\in\R^N$, so that
\begin{align*}
	\L_{\Om,\eps}^-\Psi+\psi
	\geq
	0
\end{align*}
in $B_{\frac{5}{2}\sqrt{N}}$.
Since $\varphi$ is radially decreasing, it is clear that
\begin{align*}
	\psi(x)
	\leq
	\psi(0)
	=
	A\sigma(C_1+C_3)(1-\eps)^{-\sigma-1}
	\leq
	A\sigma(C_1+C_3)(1-\eps_0)^{-\sigma-1}
\end{align*}
for every $x\in\R^N$.
Finally, we fix $r=\frac{1}{4}$ and recall that by \eqref{bound-L-Psi} with a fixed universal constant $\sigma>\sigma_0$ it holds that
\begin{equation*}
	\L_{\Omega,\eps}^-\Psi(x)
	\geq
	-\psi(x)
	\geq
	A\sigma\varphi(x)
\end{equation*}
for every $\frac{1}{4}\leq|x|\leq\frac{5}{2}\sqrt{N}$ and $\eps\in(0,\eps_0)$ with a universal $\eps_0$ satisfying \eqref{eps-bound-r-R}, so in particular $\psi\leq 0$ in $\R^N\setminus B_{1/4}$.
That is we have proved the result for every $\eps\in(0,\eps_0)$ with $\eps_0>0$ satisfying \eqref{eps0-bound-1} and \eqref{eps-bound-r-R}.

\end{proof}

\begin{remark}
\label{remarkeps0}
The constant $\eps_0$ is universal since in the 
\eqref{eps0-bound-1} and \eqref{eps-bound-r-R} we have employed $R=\frac{5}{2}\sqrt{N}$, $r=\frac{1}{4}$ and a fixed universal $\sigma$.
\end{remark}

\section{Estimate for the distribution function of $v$}

In this section, we provide different controls of the measure of the set where a given subsolution is large depending on the scale of the parameter $\eps$. More precisely,
in \Cref{first} we focus in the case in which $\eps$ can be arbitrary small, while in \Cref{second} we deal with a range of the $\eps$'s bounded away from zero.

The proof of the next result follows the lines of \cite[Lemma 3.6]{arroyobp22}, the main difference being the use of \Cref{estimate Q} and the appearance of an extra term due to the nonsymmetry of the density function.

\begin{lemma}
\label{first}
Let $\eps\in(0,\eps_0)$. There exist universal constants $\rho>0$, $M\geq 1$ and $\theta\in(0,1)$ such that if $v$ is a bounded measurable function satisfying
\begin{equation*}
	\begin{cases}
	\L_{\Om,\eps}^-v\leq\rho & \text{ in } B_{\frac{5}{2}\sqrt{N}},
	\\
	v\geq 0 & \text{ in } \R^N,
	\end{cases}
\end{equation*}
and
\begin{equation*}
	\inf_{Q_3}v
	\leq
	1,
\end{equation*}
then
\begin{equation*}
	|Q_1\cap \{v>M\}|
	\leq 
	\theta.
\end{equation*}
\end{lemma}

\begin{proof}
The basic idea consists on using the auxiliary functions $\Psi$ and $\psi$ from \Cref{barrier} to define $w=\Psi-v$, which satisfies
\begin{equation*}
	\sup_{B_{\frac{5}{2}\sqrt{N}}}w
	\geq
	\sup_{Q_3}w
	\geq
	\inf_{Q_3}\Psi-\inf_{Q_3}v
	\geq
	1,
\end{equation*}
and since $\L_{\Om,\eps}^+(\Psi-v)\geq\L_{\Om,\eps}^-\Psi-\L_{\Om,\eps}^-v$,
\begin{equation*}
	\begin{cases}
	\L_{\Om,\eps}^+w+\psi^++\rho\geq 0 & \text{ in } B_{\frac{5}{2}\sqrt{N}},
	\\
	w\leq 0 & \text{ in } \R^N\setminus B_{\frac{5}{2}\sqrt{N}}.
	\end{cases}
\end{equation*}
Since $\eps\in(0,\eps_0)$, the $\eps$-ABP estimate (\Cref{eps-ABP-thm-2}) yields that 
\begin{align*}
	1
	\leq
	~&
	\sup_{B_{\frac{5}{2}\sqrt{N}}}w
	\leq
	C\bigg(\sum_{i=1}^m(\sup_{B_{\eps/4}(x_i)}\psi^++\rho)^N|B_{\eps/4}|\bigg)^{1/N}
	\\
	\leq
	~&
	C\bigg[\bigg(\sum_{i=1}^m(\sup_{B_{\eps/4}(x_i)}\psi^+)^N|B_{\eps/4}|\bigg)^{1/N}
	+
	\rho\bigg(\sum_{i=1}^m|B_{\eps/4}|\bigg)^{1/N}\bigg],
\end{align*}
where we also used Minkowski's inequality. Now observe that the negativity of $\psi$ outside $B_{1/4}$ together with the fact that $\psi\leq\psi(0)$ allows to bound the first term inside the brackets by
\begin{align*}
	\psi(0)\bigg(\sum_{|x_i|<\frac{\eps+1}{4}}|B_{\eps/4}(x_i)|\bigg)^{1/N},
\end{align*}
where the sum is taken over all indexes $i=1,\ldots,m$ such that $B_{\eps/4}(x_i)\cap B_{1/4}\neq\emptyset$, while the other term can be directly estimated by $C\rho$ using \eqref{measure-balls} with $\Omega=B_{\frac{5}{2}\sqrt{N}}$ and taking into account that $\eps_0<1$, where $C>0$ is a universal constant. Then choosing small enough $\rho>0$ and rearranging terms we get
\begin{equation*}
	\frac{C}{\psi(0)^N}
	\leq
	\sum_{|x_i|<\frac{\eps+1}{4}}|B_{\eps/4}|.
\end{equation*}
Let $\Gamma$ be the concave envelope of $w$. Now \Cref{estimate Q} with $\sup_{B_2}w\leq\sup_{B_2}\Psi=\Psi(0)$ yields
\begin{align*}
	\frac{C}{\psi(0)^N}
	\leq
	~&
	\sum_{|x_i|<\frac{\eps+1}{4}}\big|B_{\eps/4}(x_i)\cap\big\{\Gamma-w\leq \widetilde C(\psi(0)+\rho+\Psi(0))\big\}\big|
\end{align*}
using $\eps\in(0,\eps_0)$, where $\widetilde C>0$ is universal. Moreover, since $B_{\eps/4}(x_i)\subset B_1\subset Q_1$ for every $i=1,\ldots,m$ such that $B_{\eps/4}(x_i)\cap B_{1/4}\neq\emptyset$ and that the family $\{B_{\eps/4}(x_i)\}_{i=1}^m$ has maximum overlap $\ell=\ell(N)$, we have
\begin{align*}
	\frac{C}{\psi(0)^N}
	\leq
	~&
	\ell\big|Q_1\cap\big\{\Gamma-w\leq \widetilde C(\psi(0)+\rho+\Psi(0))\big\}\big|
\end{align*}

To finish the proof just observe that since $w=\Psi-v\leq\Psi(0)-v$, $\Gamma\geq0$ and $\eps\in(0,\eps_0)$, then
\[
\begin{split}
	\frac{C}{\psi(0)^N\ell}
	&\leq
	\big|Q_1\cap\big\{\Gamma-w\leq \widetilde C(\psi(0)+\rho+\Psi(0))\big\}\big|\\
	&\leq
	\big|Q_1\cap\big\{v\leq \Psi(0)+\widetilde C(\psi(0)+\rho+\Psi(0))\big\}\big|,
\end{split}
\]
from which the result follows by conveniently taking the constants $M$ and $\theta$. 
\end{proof}

Next we state a more qualitative version of the previous lemma provided that $\eps$ is at a controlled distance to $0$, similarly saying that if $v$ is such that $\inf_{Q_1}v\leq 1$ then $|Q_1\cap\{v>K\}|\leq\frac{c}{K}$ for every $K>0$. This result will be useful later in the proof of \Cref{lem:main}. However, we state it in a slightly different way which is more conveniente for our purposes.

\begin{lemma}
\label{second}
Let $\eps\in[\frac{\eps_0}{2},\eps_0)$ and $\rho>0$. Suppose that $v$ is a bounded measurable function satisfying
\begin{equation*}
	\begin{cases}
	\L_{\Omega,\eps}^-v\leq\rho & \text{ in } B_{\frac{5}{2}\sqrt{N}},
	\\
	v\geq 0 & \text{ in } \R^N.
	\end{cases}
\end{equation*}
There exists a universal constant $c>0$ such that if
\begin{equation*}
	|Q_1\cap\{v> K\}|
	>
	\frac{c}{K}
\end{equation*}
holds for some $K>0$, then
\begin{equation*}
	v> 1 \quad \text{ in }Q_1.
\end{equation*}
\end{lemma}

\begin{proof}
By the definition of the minimal Pucci-type operator $\L_{\Omega,\eps}^-$ and since $\L_{\Omega,\eps}^- v(x)\leq \rho$ for every $x\in B_{\frac{5}{2}\sqrt{N}}$ by assumption, 
\begin{equation*}
	v(x)+\eps^2\rho
	\geq
	v(x)+\eps^2\L_{\Omega,\eps}^-v(x)
	\geq
	\beta\dashint_{B_\eps(x)}v(y)\ d\mu(y)
	\geq
	\beta\frac{\dens_0}{\dens_1}\dashint_{B_\eps(x)}v(y)\ dy
\end{equation*}
where using that $v\geq 0$ we have discarded the $\alpha$-term and we have recalled \eqref{eq:comparable}.
Then, by considering $f=\frac{\chi_{B_1}}{|B_1|}$, we can rewrite this inequality as
\begin{equation*}
	v (x) 
	\geq
	\frac{1}{\eps^N}\beta\frac{\dens_0}{\dens_1}\int_{\R^N} f\Big(\frac{y-x}{\eps}\Big) v(y)\ dy-\eps^2\rho,
\end{equation*}
which holds for every $x\in B_{\frac{5}{2}\sqrt{N}}$. Let us fix $n=\left\lfloor\frac{2\sqrt{N}}{\eps_0}\right\rfloor\in\N$ so that $Q_1\subset B_{\frac{5}{2}\sqrt{N}-(n-1)\eps}$ and $\diam Q_1\leq\frac{n\eps}{2}$ for any $\eps\in[\frac{\eps_0}{2},\eps_0)$.

Repeating the steps in \cite[Lemma 4.2]{arroyobp22} with $\beta\frac{\dens_0}{\dens_1}$ instead of just $\beta$, we get that
\begin{equation*}
	v(x)
	\geq
	\frac{1}{\eps^N}\Big(\beta\frac{\dens_0}{\dens_1}\Big)^n\int_{\R^N}f^{*n}\Big(\frac{y-x}{\eps}\Big) v(y)\ dy-\frac{\eps^2\rho}{1-\beta\frac{\dens_0}{\dens_1}}
\end{equation*}
for every $x\in Q_1\subset B_{\frac{5}{2}\sqrt{N}-(n-1)\eps}$. Here $f^{*n}$ denotes the convolution of $f$ with itself $n$ times, and by the choice of $n\in\N$ it turns out that $Q_1$ is contained in the support of $y\mapsto f^{*n}\big(\frac{y-x}{\eps}\big)$ for every $x\in Q_1$. This allows to estimate the previous integral from below and to conclude the proof as in \cite[Lemma 4.2]{arroyobp22}.
\end{proof}

\section{H\"older regularity}

Next we prove  the main result of this part, the H\"older regularity result in the context of the $\eps$-nonlocal operators. The main ingredients are  the Calder\'on-Zygmund decomposition that allows us to obtain one iteration step between the measures of the level sets or in other words a kind of expansion of positivity and the De Giorgi oscillation estimate obtained after a full iteration. The $\eps$-ABP-estimates are used in checking that the conditions of the Calder\'on-Zygmund decomposition are satisfied.

The discrete nature of the DPP does not allow us to use the rescaling argument in arbitrary small scales. Since the previous estimates require the bound $\eps_0>0$ for the scale-size, in a dyadic cube of generation $\ell$, the rescaling argument only works whenever $2^\ell\eps<\eps_0$. Thus the dyadic splitting in the Calder\'on-Zygmund decomposition has to be stopped at a certain generation, and  thus we need an additional criterion for selecting cubes in order to control the error caused by stopping the process.

The proof of the following version of the Calder\'on-Zygmund decomposition  is given in \cite[Lemma 4.1]{arroyobp23}.
We consider dyadic cubes and use the following notation. Let $\mathcal{D}_\ell$ be the family of dyadic open subcubes of $Q_1$ of generation $\ell\in\N$, where $\mathcal{D}_0=\{Q_1\}$, $\mathcal{D}_1$ is the family of $2^N$ dyadic cubes obtained by dividing $Q_1$, and so on. Given $Q\in\mathcal D_\ell$, the predecessor of $Q$, which we denote by $\mathrm{pre}(Q)$, is the unique dyadic cube in $\mathcal{D}_{\ell-1}$ that contains $Q$.

\begin{lemma}[Calder\'on-Zygmund dyadic decomposition]\label{CZ}
Let $A\subset B\subset Q_1$ be measurable sets, $\delta_1,\delta_2\in (0,1)$ and $L\in\N$. If the following assumptions hold,
\begin{enumerate}
\item $|A|\leq\delta_1$;
\item \label{item:includedB}if $Q\in\mathcal D_\ell$ with $\ell\leq L$ satisfies $|Q\cap A|>\delta_1|Q|$ then $\mathrm{pre}(Q)\subset B$;
\item \label{item:includedB2} if $Q\in\mathcal D_L$ satisfies $|Q\cap A|>\delta_2|Q|$ then $Q\subset B$;
\end{enumerate}
then,
\begin{align*}
	|A|
	\leq
	\delta_1|B|+\delta_2.
\end{align*}
\end{lemma}

In order to utilize the Calder\'on-Zygmund decomposition in deriving an estimate between the superlevel sets of $u$ below, we will use \Cref{first} to verify condition \eqref{item:includedB}, and \Cref{second} in connection to \eqref{item:includedB2} for the dyadic cubes in $\mathcal{D}_L$ where $\eps$ is `relatively large'.

\subsection{De Giorgi oscillation lemma}

One of the crucial steps in the proof of the H\"older estimate consists on a rescaling of the function $v$ to cubes of different sizes so that the measure estimates from \Cref{first} and \Cref{second} can be applied. This is done in \Cref{lem:main} below following the same argument from \cite{arroyobp22}. However, the presence in the extremal operators of a non-uniform measure $\mu$, as well as the supremum of $h$ in the $\alpha$-term, implies that its rescaling has a deeper effect than in \cite{arroyobp22}. Indeed, the rescaling of $\L_{\Omega,\eps}^\pm$ gives rise to new extremal operators (which we denote as $\widetilde\L_{\Omega,\eps}^\pm$) depending on a rescaled version of the original measure, and in turn on its rescaled density function $\tilde\dens$. It is worth to mention here that the rescalings considered in this section preserve the constants $\dens_0$, $\dens_1$ and $\|\nabla\dens\|_\infty$. To be more precise, it holds that $\tilde\dens_0=\dens_0$, $\tilde\dens_1=\dens_1$ and $\|\nabla\tilde\dens\|_\infty\leq\|\nabla\dens\|_\infty$. We state the details in the following result
where we use the notation \eqref{delta}.

\begin{lemma}[Rescaling of the extremal operators]\label{lem:rescaling}
Let $R\in(0,1]$ and $v:B_R(x_0)\to\R$ be a bounded measurable function and consider the rescaled function $\tilde v:B_1\to\R$ given by
\begin{equation*}
	\tilde v(y)
	=
	Cv(x_0+Ry),
\end{equation*}
where $C>0$. Let $\mu$ be a probability measure supported in $\Omega$ with Lipschitz density $\dens:\Omega\to[\dens_0,\dens_1]$. If $\L^\pm_{\Omega,\eps}$ are the extremal operators with probability measure $\mu$, then
\begin{align*}
	\widetilde\L^-_{\Omega,\frac{\eps}{R}}\tilde v(y)
	\leq
	~&
	CR^2\L^-_{\Omega,\eps}v(x_0+Ry)
	\\
	\widetilde\L^+_{\Omega,\frac{\eps}{R}}\tilde v(y)
	\geq
	~&
	CR^2\L^+_{\Omega,\eps}v(x_0+Ry)
\end{align*}
for every $y\in B_1$, where $\widetilde\L^\pm_{\Omega,\frac{\eps}{R}}$ are the rescaled extremal operators related to the density function $\tilde\dens$ satisfying $\tilde\dens(\xi)=\dens(x_0+R\xi)$.
\end{lemma}

\begin{proof}
By the affine invariance of the extremal operators, we can assume without loss of generality that $C=1$. Let $\tilde\eps=\frac{\eps}{R}$. By the definition of $\tilde v$ and \eqref{delta} we have
\begin{equation*}
	\delta\tilde v(y,\tilde\eps z;\tilde\eps^2h)
	=
	\delta v\Big(x_0+Ry,\eps z;\eps^2\frac{h}{R}\Big)
\end{equation*}
for $y\in B_1$, $|z|<\Lambda$ and $|h|<\tau$. Using that
\begin{equation*}
	\inf_{|h|<\tau}f\Big(\frac{h}{R}\Big)
	\leq
	\inf_{|h|<\tau}f(h)
	\quad
	\text{and}
	\quad
	\sup_{|h|<\tau}f\Big(\frac{h}{R}\Big)
	\geq
	\sup_{|h|<\tau}f(h)
\end{equation*}
for $0<R\leq 1$ we get
\begin{equation*}
	\inf_{|z|<\Lambda}\inf_{|h|<\tau}\frac{\delta\tilde v(y,\tilde\eps z;\tilde\eps^2h)}{2\tilde\eps^2}
	\leq
	R^2\inf_{|z|<\Lambda}\inf_{|h|<\tau}\frac{\delta v(x_0+Ry,\eps z;\eps^2h)}{2\eps^2}
\end{equation*}
and
\begin{equation*}
	\sup_{|z|<\Lambda}\sup_{|h|<\tau}\frac{\delta\tilde v(y,\tilde\eps z;\tilde\eps^2h)}{2\tilde\eps^2}
	\geq
	R^2\sup_{|z|<\Lambda}\sup_{|h|<\tau}\frac{\delta v(x_0+ty,\eps z;\eps^2h)}{2\eps^2}.
\end{equation*}

On the other hand, let $\tilde\mu$ be the push-forward measure of $\mu$ given by
\begin{align*}
	\tilde\mu(A)
	=
	\int_A\tilde\dens(\xi)\ d\xi
	=
	\int_A\dens(x_0+R\xi)\ d\xi
\end{align*}
for every Borel measurable set $A$. Observe that $\dens_0\leq\tilde\dens\leq\dens_1$ and $\|\nabla\tilde\dens\|_\infty=R\|\nabla\dens\|_\infty\leq \|\nabla\dens\|_\infty$ for every $x_0$ and $0<R\leq 1$.
Then
\begin{align*}
	\dashint_{B_{\tilde\eps}(y)}\frac{\tilde v(\xi)-\tilde v(y)}{\tilde\eps^2}\ d\tilde\mu(\xi)
	=
	~&
	R^2\dashint_{B_{\tilde\eps}(y)}\frac{v(x_0+R\xi)-v(x_0+Ry)}{\eps^2}\ d\tilde\mu(\xi)
	\\
	=
	~&
	R^2\frac{\displaystyle\dashint_{B_{\tilde\eps}(y)}\frac{v(x_0+R\xi)-v(x_0+Ry)}{\eps^2}\,\dens(x_0+R\xi)\ d\xi}{\displaystyle\dashint_{B_{\tilde\eps}(y)}\dens(x_0+R\xi)\ d\xi}
	\\
	=
	~&
	R^2\frac{\displaystyle\dashint_{B_\eps(x_0+Ry)}\frac{v(\xi)-v(x_0+Ry)}{\eps^2}\,\dens(\xi)\ d\xi}{\displaystyle\dashint_{B_\eps(x_0+Ry)}\dens(\xi)\ d\xi}
	\\
	=
	~&
	R^2\dashint_{B_\eps(x_0+Ry)}\frac{v(\xi)-v(x_0+Ry)}{\eps^2}\ d\mu(\xi).
\end{align*}
Combining this identity with the inequalities for the infima and the suprema we obtain the desired results.
\end{proof}

Next we state De Giorgi iteration for the measures of the superlevel sets that is later needed in the power decay lemma. Here we use the $\eps$-nonlocal version of the  Calder\'on-Zygmund decomposition from \Cref{CZ} in order to expand the positivity together with the preliminary measure estimates from \Cref{first} and \Cref{second}  to check the conditions  Calder\'on-Zygmund decomposition. Since the explicit  operators are only used  in the lemmas but not in the actual proof, it is essentially the same as \cite[Lemma 4.3]{arroyobp22}. Nonetheless for the convenience of the reader, we carefully write down the proof including the use of the rescaling property for operators of this paper in \eqref{eq:DG-rescaling}.

\begin{lemma}[De Giorgi iteration]
\label{lem:main}
Let $\eps\in(0,\eps_0)$, $\rho>0$, $M\geq 1$ and $\theta\in(0,1)$ as in \Cref{first} and $c>0$ as in \Cref{second}. If $v$ is a bounded measurable function satisfying
\begin{equation*}
	\begin{cases}
	\L_{\Omega,\eps}^-v\leq\rho & \text{ in } B_{\frac{5}{2}\sqrt{N}},
	\\
	v\geq 0 & \text{ in } \R^N,
	\end{cases}
\end{equation*}
and
\begin{equation*}
	\inf_{Q_3}v
	\leq
	1,
\end{equation*}
then
\begin{equation*}
	|Q_1\cap\{v> K^m\}|
	\leq
	\frac{c}{(1-\theta)K}+\theta^m,
\end{equation*}
holds for every $K\geq M$ and $m\in\N$.
\end{lemma}

\begin{proof}
For simplicity, we define
\begin{equation*}
	A_m=Q_1\cap\{v>K^m\}
\end{equation*}
for each $m\in\N$. We start by observing that the desired inequality is immediately satisfied for $m=1$. Indeed, by \Cref{first}, we have
\begin{equation*}
	|A_1|
	=
	|Q_1\cap\{v>K\}|
	\leq
	|Q_1\cap\{v> M\}|
	\leq
	\theta
	\leq
	\frac{c}{K}+\theta
\end{equation*}
for every $K\geq M$. In order to prove that the inequality holds for every $m\in\N$, we claim that
\begin{equation}\label{claim}
	|A_{m+1}|
	\leq
	\frac{c}{K}+\theta|A_m|
\end{equation}
for every $m\in\N$, since in that case we immediately get that
\begin{equation*}
	|A_m|
	\leq
	\frac{c}{K}(1+\theta+\cdots+\theta^{m-1})+\theta^m
	\leq
	\frac{c}{(1-\theta)K}+\theta^m
\end{equation*}
for each $m\in\N$.

We now show \eqref{claim}. Fix $L=L(\eps)\in\N$ the unique integer such that $2^L\eps<\eps_0\leq 2^{L+1}\eps$. Our aim is to apply \Cref{CZ} with $\delta_1=\theta$ and $\delta_2=\frac{c}{K}$. It is clear that $A_{m+1}\subset A_m\subset Q_1$ and $|A_{m+1}|\leq\theta$, so the first assumption in \Cref{CZ} is satisfied. Next we check that the two remaining conditions also hold.

To check assumption \eqref{item:includedB}, given any dyadic cube $Q\in\mathcal{D}_\ell$ of generation $\ell\leq L$ satisfying
\begin{equation}\label{eq:meas-assump}
	|Q\cap A_{m+1}|>\theta|Q|,
\end{equation}
we need to prove that $\mathrm{pre}(Q)\subset A_m$. Suppose on the contrary that $\mathrm{pre}(Q)\not\subset A_m$. Then there exists $\tilde x\in\mathrm{pre}(Q)$ such that $v(\tilde x)\leq K^m$. Denote by $x_0$ the center of $Q$ and define the rescaled function $\tilde v:Q_1\to\R$ as
\begin{equation}\label{tilde v}
	\tilde v(y)
	=
	\frac{1}{K^m}\,v(x_0+2^{-\ell}y)
\end{equation}
for $y\in Q_1$. By \Cref{lem:rescaling} with $R=2^{-\ell}$ and $C=\frac{1}{K^m}$ we have that
\begin{align}
\label{eq:DG-rescaling}
	\widetilde\L^-_{\Omega,2^\ell\eps}\tilde v(y)
	\leq
	\frac{1}{2^{2\ell}K^m}\L^-_{\Omega,\eps}v(x_0+2^{-\ell}y)
	\leq
	\frac{\rho}{2^{2\ell}K^m}
	\leq
	\rho
\end{align}
where we have used that $K\geq M\geq1$ together with the assumption $\L_{\Om,\eps}^-v\leq\rho$. Observe also that $\widetilde\L_{\Omega,2^\ell\eps}^-$ is the minimal Pucci operator with the measure $\tilde\mu$ instead of $\mu$ and that $2^\ell\eps\leq 2^L\eps<\eps_0$. Moreover, $\inf_{Q_3}\tilde v\leq 1$ since $v(\tilde x)\leq K^m$. Thus $\tilde v\geq 0$ is a bounded measurable function satisfying $\widetilde\L_{\Om,2^\ell\eps}^-\tilde v\leq\rho$, so recalling \Cref{first} we obtain that
\begin{equation*}
	|Q_1\cap\{\tilde v> K\}|
	\leq
	\theta.
\end{equation*}
On the other hand, using \eqref{tilde v} we have
\begin{equation*}
	|Q_1\cap\{\tilde v> K\}|
	=
	2^{N\ell}|Q\cap\{u> K^{m+1}\}|
	=
	\frac{|Q\cap A_{m+1}|}{|Q|},
\end{equation*}
so we reach a contradiction with \eqref{eq:meas-assump}. Thus $\mathrm{pre}(Q)\subset A_m$ and the assumption \eqref{item:includedB} from \Cref{CZ} holds.

Next we check that the assumption \eqref{item:includedB2} also holds. That is, given any dyadic cube $Q\in\mathcal{D}_L$ such that
\begin{equation*}
	|Q\cap A_{m+1}|>\frac{c}{K}|Q|,
\end{equation*}
we need to show that $Q\subset A_m$. Using the same rescaled function constructed in \eqref{tilde v} we immediately have that
\begin{equation*}
	|Q_1\cap\{\tilde v>K\}|
	=
	\frac{|Q\cap A_{m+1}|}{|Q|}
	>
	\frac{c}{K},
\end{equation*}
so similarly recalling \Cref{second} we get that $\tilde v>1$ in $Q_1$. In terms of the function $v$, this reads as $v>K^m$ in $Q$, which in particular means that $Q\subset A_m$ as wanted, and so \eqref{item:includedB2} is satisfied.

Since the assumptions in \Cref{CZ} are fulfilled with $A=A_{m+1}$ and $B=A_m$, we conclude that \eqref{claim} holds, so the proof is finished.
\end{proof}

As an immediate consequence of \Cref{lem:main}, by appropriately selecting the constants $K\geq M$ and $m\in\N$, we obtain a decay estimate for the distribution function of $v$ which reads as
\begin{equation*}
	|Q_1\cap\{v>t\}|
	\leq
	d\exp\bigg(-\sqrt{\frac{\log t }{a}}\bigg)
\end{equation*}
for every $t\geq1$. This is proven in \cite[Lemma 4.4]{arroyobp22}.

Repeating the arguments in \cite[Lemmas 4.5 and 4.6]{arroyobp22} (see also \cite[Lemma 5.8]{arroyobp23}), one can show in a straightforward manner the following:
Let $R>0$ and $\eps\in(0,R\eps_0)$. There exists $c>1$ such that if $v\geq0$ is a bounded measurable function satisfying $\L_{\Omega,\eps}^+v\geq-\rho$ in $B_{cR}$ for some $\rho>0$ then
\begin{equation*}
	|B_R\cap\{v\leq M\}|
	\geq
	\theta|B_R|
\end{equation*}
for given $\theta\in(0,1)$ and $M\in\R$ implies that
\begin{equation*}
	\sup_{B_R}v
	\leq
	(1-\eta)\sup_{B_{cR}}v+\eta M+CR^2\rho
\end{equation*}
for certain $\eta=\eta(\theta)\in(0,1)$. This is the so-called De Giorgi oscillation lemma, and it is the key in the proof of the asymptotic H\"older regularity estimate that we state next. Interested readers may find the proof from Section 6 in \cite{arroyobp23}.

\begin{theorem}[Local asymptotic H\"older regularity for $v$]\label{Holder}
Let $R>0$ and fix any $\eps\in(0,R\eps_0)$. If $v\in L^\infty(B_R)$ is a function satisfying
\begin{equation*}
	\L_{\Omega,\eps}^+v
	\geq
	-\rho,
	\qquad
	\L_{\Omega,\eps}^-v
	\leq
	\rho,
\end{equation*}
in $B_{2R}$ for some $\rho>0$, then there exists $\gamma\in(0,1]$ such that
\begin{equation*}
	|v(x)-v(y)|
	\leq
	\frac{C}{R^\gamma}(\|v\|_{L^\infty(B_R)}+R^2\rho)(|x-y|^\gamma+\eps^\gamma)
\end{equation*}
for every $x,y\in B_R$, where $C>0$ and $\gamma\in(0,1]$ are universal constants.
\end{theorem}

Observe that since this result gives H\"older regularity for the operators with drift terms, this is of independent interest apart from the graph context and also compared to the results in \cite{arroyobp22,arroyobp23}. In the PDE setting this would correspond to operators with first order gradient dependent terms.

\begin{example}[$\eps$-nonlocal operators with a drift term]
\label{ex:drift}
To emphasize that the ABP and regularity results obtained in \Cref{part2} are quite robust and that definitions originally tailored for treating the loss of symmetry in the graph operators allow extra flexibility, we present a further   example. We show that the results apply to an operator with a drift term. 
 Let $\tau>0$ and $\beta>0$, $\alpha,\gamma\geq 0$, and define
\begin{align*}
	\hat \L_{\Omega,\eps}^+v(x)
	=
	~&
	\alpha\sup_{|z|<\Lambda\eps}\frac{\delta_{\Omega,\tau\eps^2}^+v(x,z)}{2\eps^2}
	+
	\beta\dashint_{B_\eps(x)}\frac{v(y)-v(x)}{\eps^2}\ d\mu(y)
	\\
	~&
	+
	\gamma \sup_{|h|<\tau\eps^2}\frac{v(x+h)-v(x)}{\eps^2}.
\end{align*}
Here the last term can be used to include many kind of drift terms for example a constant drift to the first coordinate direction $e_1$   
\begin{align*}
\frac{v(x+\eps^2e_1)-v(x)}{\eps^2}
\end{align*}
as a simplest case.
It turns out that the  last term can be absorbed by the $\alpha$-term:
\begin{align*}
	\sup_{|h|<\tau\eps^2}\frac{v(x+h)-v(x)}{\eps^2}
	=
	~&
	\sup_{|h|<\tau\eps^2}\frac{v(x)+v(x+h)-2v(x)}{2\eps^2}
	\\
	\leq
	~&
	\sup_{|z|<\Lambda\eps}\sup_{|h|<\tau\eps^2}\frac{v(x+z)+v(x-z+h)-2v(x)}{2\eps^2}
	\\
	=
	~&
	\sup_{|z|<\Lambda\eps}\frac{\delta_{\Omega,\tau\eps^2}^+v(x,z)}{2\eps^2}.
\end{align*}
This means that the drift term of this form can be absorbed into the nonlinear part of the maximal $\eps$-nonlocal operator \eqref{L-Omega}.
\end{example}

\part{Connection to partial differential equations}\label{part3}

Our aim in this part is to show that the discrete operators can be associated to partial differential equations.
To do that, we begin by obtaining a convergent subsequence with an Ascoli-Arzel\`a type argument.
Here the asymptotic H\"older regularity obtained in \Cref{thm:main-graph-holder-regularity} plays a crucial role. 
Then we compute asymptotic expansions for twice-differentiable functions to take limits as $\eps\to 0$. In particular we consider an extremal partial differential operator related to the classical Pucci operator and the weighted $p$-Laplacian.

\section{Convergence}

We consider a sequence of random data clouds $\{\X_{n_k}^{(k)}\}_{k\in\N}$  such that $\{n_k\}_{k\in\N}$ is sequence of natural numbers.  Moreover, for each $k$ we consider a function $u_k\in L^\infty(\X_{n_k}^{(k)})$ such that
\begin{equation*}
	\L_{\X_{n_k}^{(k)},\eps_k}^+u_k
	\geq
	-\rho,
	\quad
	\L_{\X_{n_k}^{(k)},\eps_k}^-u_k
	\leq
	\rho
\end{equation*}
in $\X_{n_k}^{(k)}$ with some fixed $\rho>0$ and $\eps_k\in(0,\eps_0)$. 
In addition, we assume that $u_k$ has prescribed ``boundary'' values. More precisely, if we consider the inner boundary strip of width $\eps$, that is
\begin{equation*}
	\Gamma_\eps
	=
	\{x\in\Omega\,:\,\dist(x,\partial\Omega)\leq \eps\},
\end{equation*}
and $g\in C(\Gamma_{\eps_0})$ is a given function playing the role of the continuous boundary data, then we assume that $u_k=g$ in $\O_{\eps_k}$, where we use the notation
\[
\O_\eps=\X_n\cap\Gamma_{\eps}.
\]

Our aim is to use the established regularity results. However, there is an inconvenience when trying to apply directly this approach, and it is due to the fact that the discrete-to-nonlocal inequality from \Cref{thm:discrete-to-nonlocal} introduces a dependence on the supremum norm of the solutions. In fact,
\begin{equation}\label{d-t-n-uk}
	\L_{\Omega,\Lambda+\eps^2,\tau+\eps^2,\eps_k}^+[u_k\circ T_{\eps_k}](x)
	\geq
	-\rho-C\|u_k\|_{L^\infty(\X_{n_k}^{(k)})},
\end{equation}
so the regularity estimates for $u_k\circ T_{\eps_k}$ depend \emph{a priori} on the supremum norm of $u_k$, which does not allow to apply an Ascoli-Arzel\`a type result.

In order to bypass this problem, we show that a function $u_k\in L^\infty(\X_{n_k}^{(k)})$ satisfying the Pucci-type inequalities in $\X_{n_k}^{(k)}$ has bounds which are independent of $k$, and so the sequence $\{u_k\}_k$ is uniformly bounded. In consequence, the Pucci bounds in \eqref{d-t-n-uk} are uniform, an then the asymptotic H\"older regularity estimates for both $u_k$ and its extension $u_k\circ T_{\eps_k}$ are uniform in $k$.

But even more, the discrete-to-nonlocal inequality \eqref{d-t-n-uk} is not enough for the purposes of this part of the article even though we can obtain a uniform control on $\|u_k\|_{L^\infty(\X_{n_k}^{(k)})}$. The reason for this is that, despite being controlled by a uniform constant for every $k$, the size of the second term in the right hand side of \eqref{d-t-n-uk} adds an error in the estimates which affects negatively.

Fortunately, it is possible to improve the size of this error by appropriate selecting the constants $\delta$ and $\lambda$ in the proof of \Cref{lem:dens-eps}. More precisely, if we choose $\delta=\eps^{3+a}$ and $\lambda=\eps^{2+a}$ with some $a>0$, then the events \eqref{partition} and \eqref{event} become
\begin{equation}\label{partition-a>0}
	Z_i
	\in
	\U_i\subset B_{\eps^{3+a}}(Z_i)
	\quad\text{and}\quad
	\mu_\eps(\U_i)=\int_{\U_i}\dens_\eps(y)\ dy=\frac{1}{n}
\end{equation}
and
\begin{equation}\label{event-a>0}
	\|\dens_\eps-\dens\|_\infty
	\leq
	C_0\eps^{2+a},
\end{equation}
and they hold with probability at least
\begin{equation}\label{probability-a>0}
	1-2n\exp\{-c_0n\eps^{3N+4+(N+2)a}\}.
\end{equation}

Repeating the proofs in \Cref{sec:discrete-to-nonlocal} assuming that \eqref{partition-a>0} and \eqref{event-a>0} hold with some fixed $a>0$ leads to an analogous version of the estimate \eqref{eq:discrete-to-nonlocal-averages}
from \Cref{thm:discrete-to-nonlocal-averages} which reads as
\begin{equation}\label{eq:discrete-to-nonlocal-averages-a}
	\bigg|\frac{1}{\card{\V_\eps(T_\eps(x))}}\sum_{Z_i\in\V_\eps(T_\eps(x))}u(Z_i)
	-
	\dashint_{B_\eps(x)}(u\circ T_\eps)\ d\mu\bigg|
	\leq
	C\|u\|_{L^\infty(\X_n)}\eps^{2+a}.
\end{equation}
By introducing this change in the proof of the discrete-to-nonlocal inequality (\Cref{thm:discrete-to-nonlocal}) we get that
\begin{equation}\label{discrete-to-nonlocal-a>0}
	\L_{\X_n,\Lambda,\tau,\eps}^+u(T_\eps(x))
	\leq
	\L_{\Omega,\Lambda+\eps^2,\tau+2\eps,\eps}^+[u\circ T_\eps](x)
	+
	C\|u\|_{L^\infty(\X_n)}\eps^a
\end{equation}
holds for every $x\in\Omega_{-\eps}$ and each $u\in L^\infty(\X_n)$
with probability at least \eqref{probability-a>0}. Then, for $u_k$ as above, instead of \eqref{d-t-n-uk} we have that
\begin{equation*}
	\L_{\Omega,\Lambda+\eps^2,\tau+\eps^2,\eps_k}^+[u_k\circ T_{\eps_k}](x)
	\geq
	-\rho-C\|u_k\|_{L^\infty(\X_{n_k}^{(k)})}\eps^a,
\end{equation*}
so the second term in the right hand side can be as small as desired by setting a sufficiently small $\eps>0$ uniformly for every $k$. This allows to mitigate the bad effect produced by this extra term in the estimates.

The rest of the section is organized as follows. In \Cref{sec:uniform-bound} we show that the barrier constructed in \Cref{sec:barrier} can be adapted to the discrete setting, and thus with the aid of a comparison principle in the random data cloud, we obtain the uniform bound. Next in \Cref{sec:boundary-continuity}, using the explicit barriers from the preceding section, equicontinuity estimates are obtained for $u_k\in L^\infty(\X_{n_k}^{(k)})$. Finally, by using a variant of the Ascoli-Arzel\`a theorem, we prove that $\{u_k\circ T_{\eps_k}\}_k$ converge (up to a subsequence) to a function with probability $1$.

\subsection{Uniform boundedness for Pucci-type inequalities on $\X_n$}
\label{sec:uniform-bound}

We start by proving a comparison principle for functions defined in $\X_n$, for which we use a relaxed notion of ``boundary'' consisting on the nonempty subset of points from the data cloud $\O_\eps=\X_n\cap\Gamma_\eps$.
It is worth remarking that even thought $\O_\eps$ is a collection of points in the random data cloud $\X_n$, the points in $\O_\eps$ could be regarded as non-random, that is, we can assume that $\O_\eps$ is fixed, as long as the graph $\X_n$ satisfies \eqref{partition}, or in particular \eqref{partition-a>0} with some fixed $a>0$.

\begin{lemma}[Comparison principle]\label{lem:comparison-pple}
Let $\X_n\subset\Omega$ be a random data cloud and let $\O_\eps=\X_n\cap\Gamma_\eps$. Suppose that the event \eqref{partition-a>0} holds for some fixed $a>0$ and that $u,v\in L^\infty(\X_n)$ satisfy
\begin{equation*}
	\begin{dcases}
	\L_{\X_n,\eps}^+ u \geq \L_{\X_n,\eps}^+ v & \text{ in } \X_n\setminus\O_\eps,
	\\
	u \leq v & \text{ in } \O_\eps.
	\end{dcases}
\end{equation*}
Then
\begin{equation*}
	u
	\leq
	v
	\qquad
	\text{ in } \X_n.
\end{equation*}
\end{lemma}

\begin{proof}
The proof follows a standard argument. For the sake of simplicity, we introduce the following shorthand notations,
\begin{align*}
	\A u(Z_i)
	=
	~&
	\frac{1}{\card{\V_\eps(Z_i)}}\sum_{Z_j\in\V_\eps(Z_i)}u(Z_j),
	\\
	\M^+u(Z_i)
	=
	~&
	\max_{Z_j\in\V_{\Lambda\eps}(Z_i)}\max_{Z_k\in\V_{\tau\eps^2}(2Z_i-Z_j)}\frac{u(Z_j)+u(Z_k)}{2}.
\end{align*}
Note that these operators are monotone in $L^\infty(\X_n)$. In this way, by the definition of the maximal operator $\L_{\X_n,\eps}^+$ in \eqref{def:L-Xn} and by assumption,
\begin{equation*}
	\alpha\M^+u+\beta\A u-u
	=
	\eps^2\L_{\X_n,\eps}^+u
	\geq
	\eps^2\L_{\X_n,\eps}^+v
	=
	\alpha\M^+v+\beta\A v-v
\end{equation*}
in $\X_n\setminus\O_\eps$. Moreover, since $\A:L^\infty(\X_n)\to L^\infty(\X_n)$ is a linear operator, we can write
\begin{equation}\label{comparison-ineq}
	u-v
	\leq
	\alpha(\M^+u-\M^+v)+\beta\A[u-v].
\end{equation}
Let us define 
\begin{equation*}
	m=\max_{\X_n}(u-v).
\end{equation*}
In order to show that $m\leq 0$, we assume on the contrary that $m>0$ and define the set $S$ of points in $\X_n$ for which the maximum is attained, that is
\begin{equation*}
	S
	=
	\{Z_i\in\X_n\,:\,u(Z_i)-v(Z_i)=m\}.
\end{equation*}
Then the contradiction with the assumption $u-v\leq 0$ in $\O_\eps$ will follow after showing that $S\cap\O_\eps\neq\emptyset$.

Let any $Z_i\in S$. Then $u(Z_i)=m+v(Z_i)$ and by \eqref{comparison-ineq} we have
\begin{align*}
	m
	\leq
	~&
	\alpha\big(\M^+u(Z_i)-\M^+v(Z_i)\big)+\beta\A[u-v](Z_i)
	\\
	\leq
	~&
	\alpha m+\beta\A[u-v](Z_i),
\end{align*}
where in the second inequality we have used that $u\leq v+m$ together with the monotonicity of $\M^+$. Since $\beta=1-\alpha>0$, recalling again that $u-v\leq m$ in $\X_n$ using the monotonicity of $\A$ we obtain
\begin{equation*}
	\A[u-v](Z_i)
	=
	m.
\end{equation*}
This implies that $u-v=m$ in $\V_\eps(Z_i)$, since otherwise $u(Z_j)-v(Z_j)<m$ for at least one $Z_j\in\V_\eps(Z_i)$, which would imply that $\A[u-v](Z_i)<m$. That is, we have shown that from $Z_i\in S$ it follows that
\begin{equation}\label{comparison-inclusion}
	\V_\eps(Z_i)\subset S.
\end{equation}
Since \eqref{partition-a>0} holds by assumption, iterating this property we can deduce that $S\cap\O_\eps\neq\emptyset$ so we reach the desired contradiction. The idea is as follows: starting from any $Z_{i_0}\in S$, we can construct a finite sequence of points $Z_{i_0},Z_{i_1},Z_{i_2},\ldots$ approaching a point $W\in\O_\eps$ contained in the same connected component of $\Omega$ in such a way that $Z_{i_k}\in\V_\eps(Z_{i_{k-1}})$ for each $k$. Then \eqref{comparison-inclusion} yields that this collection of points is in $S$. Moreover, the condition \eqref{partition-a>0} implies that at each step the point $Z_{i_k}\in\V_\eps(Z_{i_{k-1}})$ can be selected in such a way that its distance to $W$ decreases at least $\eps/2$. Therefore, $W$ is reached after a finite amount of steps, and again by \eqref{comparison-inclusion} we have that $W\in S$, and in particular $S\cap\O_\eps\neq\emptyset$. Then $m\leq 0$ and so $u\leq v$ in $\X_n$.
\end{proof}

We show that the barrier function from \Cref{sec:barrier} can be used also as a barrier in the discrete setting.

\begin{lemma}\label{lem:L+barrier}
Let $\X_n\subset\Omega$ be a random data cloud and suppose that the events \eqref{partition-a>0} and \eqref{event-a>0} hold for some fixed $a>0$. Let $r>0$ and $\xi\in\R^N\setminus\Omega$ be such that $\dist(\xi,\Omega)\geq r$,  and define $R=R(\Omega,\xi)=\max\{|x-\xi|\,:\,x\in\Omega\}\leq\diam\Omega+\dist(\xi,\Omega)$. For $\sigma>\sigma_0$, where $\sigma_0=\sigma(r,R)$ satisfies \eqref{sigma-bound-r-R}, let $\Phi_\xi:\R^N\to[0,1)$ be the smooth function defined as
\begin{equation}
\label{eq:graph-barrier}
	\Phi_\xi(x)
	=
	1-(1+|x-\xi|^2)^{-\sigma}
\end{equation}
for every $x\in\R^N$. There exists $\eps_0=\eps_0(r,R,\sigma)>0$ (satisfying \eqref{eps0-bound-1} and \eqref{eps-bound-r-R}) such that
\begin{equation}\label{L+barrier}
	\L_{\X_n,\eps}^+\big[\Phi_\xi\big|_{\X_n}\big]
	\leq
	-\frac{\sigma}{2}(R^2+1)^{-\sigma}
\end{equation}
in $\X_n\cap\Omega_{-\eps}$ for every $\eps\in(0,\eps_0)$.
\end{lemma}

\begin{proof}
After a translation of the domain, we can assume without loss of generality that $\xi=0\in\R^N\setminus\Omega$, so that $r\leq|x|\leq R=R(\Omega,0)$ for every $x\in\Omega$. We consider the function
\begin{equation*}
	\Phi(x)
	=
	\Phi_0(x)
	=
	1-(1+|x|^2)^{-\sigma}.
\end{equation*}
By the discrete-to-nonlocal inequality from \eqref{discrete-to-nonlocal-a>0} applied with $\|\Phi\big|_{\X_n}\|_{L^\infty(\X_n)}\leq1$ we have that there exists a universal constant $C>0$ such that
\begin{equation}\label{discrete-to-nonlocal-barrier}
	\L_{\X_n,\Lambda,\tau,\eps}^+\Phi(Z_i)
	\leq
	\L_{\Omega,\Lambda+\eps^2,\tau+2\eps,\eps}^+[\Phi\circ T_\eps](x)
	+
	C\eps^a
\end{equation}
for every $Z_i\in\X_n$ and $x\in\U_i=T_\eps^{-1}(Z_i)$ provided that the events \eqref{partition-a>0} and \eqref{event-a>0} hold. The idea is to estimate the right hand side with the aid of \Cref{barrier-0}. In order to do so, first we observe that
\begin{equation}\label{difference-barrier}
	\abs{\L_{\Omega,\Lambda+\eps^2,\tau+2\eps,\eps}^+[\Phi\circ T_\eps](x)-\L_{\Omega,\Lambda+\eps^2,\tau+2\eps,\eps}^+\Phi(x)}
	\leq
	4\sigma R\eps^{1+a}
\end{equation}
for $x\in\Omega$. This can be directly deduced from the elementary estimate
\begin{align*}
	|(\Phi\circ T_\eps)(x)-\Phi(x)|
	\leq
	\|\nabla\Phi\|_\infty|T_\eps(x)-x|
	\leq
	2\sigma R\eps^{3+a},
\end{align*}
where we used that $|\nabla\Phi(x)|=2\sigma(1+|x|^2)^{-\sigma-1}|x|\leq 2\sigma R$ for every $x\in\Omega\subset B_R$ and that $|T_\eps(x)-x|\leq\eps^{3+a}$. Indeed, the desired bound in \eqref{difference-barrier} is obtained after inserting the following inequalities in the definition of \eqref{L-Omega} for the corresponding minimal operator,
\begin{multline*}
	\bigg|\sup_{|z|<(\Lambda+\eps^2)\eps}\frac{\delta_{\Omega,(\tau+2\eps)\eps^2}^+(\Phi\circ T_\eps)(x,z)}{2\eps^2}
	-
	\sup_{|z|<(\Lambda+\eps^2)\eps}\frac{\delta_{\Omega,(\tau+2\eps)\eps^2}^+\Phi(x,z)}{2\eps^2}\bigg|
	\\
	\leq
	\sup_{|z|<(\Lambda+\eps^2)\eps}\frac{|\delta_{\Omega,(\tau+2\eps)\eps^2}^+[\Phi\circ T_\eps-\Phi](x,z)|}{2\eps^2}
	\leq
	4\sigma R\eps^{1+a} 
\end{multline*}
and
\begin{multline*}
	\bigg|\dashint_{B_\eps(x)}\frac{(\Phi\circ T_\eps)(y)-(\Phi\circ T_\eps)(x)}{\eps^2}\ d\mu(y)
	-\dashint_{B_\eps(x)}\frac{\Phi(y)-\Phi(x)}{\eps^2}\ d\mu(y)\bigg|
	\\
	\leq
	\dashint_{B_\eps(x)}\frac{|(\Phi\circ T_\eps)(y)-\Phi(y)|}{\eps^2}\ d\mu(y)
	+
	\frac{|(\Phi\circ T_\eps)(x)-\Phi(x)|}{\eps^2}
	\leq
	4\sigma R\eps^{1+a} .
\end{multline*}
Therefore, \eqref{discrete-to-nonlocal-barrier} together with \eqref{difference-barrier} yields
\begin{equation*}
	\L_{\X_n,\Lambda,\tau,\eps}^+\Phi(T_\eps(x))
	\leq
	\L_{\Omega,\Lambda+\eps^2,\tau+2\eps,\eps}^+\Phi(x)
	+C\eps^a
	+4\sigma R\eps^{1+a} .
\end{equation*}

Now, recalling \Cref{barrier-0} noting that $\varphi=1-\Phi$, there exists $\sigma_0=\sigma_0(r,R)>0$ (with $R=R(\Omega,0)$) such that for $\sigma>\sigma_0$ it holds that
\begin{equation*}
	\L_{\Omega,\Lambda+\eps^2,\tau+2\eps,\eps}^+\Phi(x)
	=
	-\L_{\Omega,\Lambda+\eps^2,\tau+2\eps,\eps}^-\varphi(x)
	\leq
	-\sigma\varphi(x)
	\leq
	-\sigma(1+R^2)^{-\sigma}
\end{equation*}
for $x\in\Omega\subset B_{R}\setminus B_r$, so
\begin{equation*}
	\L_{\X_n,\Lambda,\tau,\eps}^+\Phi(T_\eps(x))
	\leq
	-\sigma(1+R^2)^{-\sigma}
	+(C
	+4\sigma R)\eps^a 
\end{equation*}
and \eqref{L+barrier} follows for every $\eps\in(0,\eps_0)$ with a sufficiently small $\eps_0=\eps_0(r,R,\sigma)>0$.
\end{proof}

Next we use the comparison principle together with the explicit barriers to obtain the following maximum principle.
Later this immediately implies uniform boundedness for example for solutions to equations of the form $$\L_{\X_n,\eps} u_k(Z_j) = f(Z_j)$$ and is utilized in passing to a limit. If in this example we had $f=0$, then uniform boundedness would follow just by taking big and small constants as barriers.

\begin{lemma}[Uniform boundedness]\label{lem:uniform-bound}Let $\X_n\subset\Omega$ be a random data cloud. Suppose that the events \eqref{partition-a>0} and \eqref{event-a>0} hold and that $u\in L^\infty(\X_n)$ satisfies
\begin{equation*}
	\L_{\X_n,\eps}^+u
	\geq
	-\rho,
	\quad
	\L_{\X_n,\eps}^-u
	\leq
	\rho
\end{equation*}
in $\X_n\setminus\O_\eps$ for some $\rho>0$. Then
\begin{equation*}
	\|u\|_{L^\infty(\X_n)}
	\leq
	\|u\|_{L^\infty(\O_\eps)}+C_\Omega\rho
\end{equation*}
for some constant $C_\Omega>0$ depending exclusively on $\Omega$.
\end{lemma}

\begin{proof}
Let any $r>0$, fix any $\xi\in\R^N\setminus\Omega$ such that $\dist(\xi,\Omega)\geq 1$ and consider the explicit barrier $\Phi_\xi:\R^N\to[0,1)$ constructed in \Cref{lem:L+barrier} with $\sigma>\sigma_0(1,R)$ and $R=R(\Omega,\xi)$. We define $v\in L^\infty(\X_n)$ as
\begin{equation*}
	v(Z_i)
	=
	\|u\|_{L^\infty(\O_\eps)}+2\rho\,\frac{(R^2+1)^\sigma}{\sigma}\,\Phi_\xi(Z_i)
\end{equation*}
for every $Z_i\in\X_n$. Then by \eqref{L+barrier} we have that
\begin{equation*}
	\L_{\X_n,\eps}^+v
	=
	2\rho\,\frac{(R^2+1)^\sigma}{\sigma}\,\L_{\X_n,\eps}^+\big[\Phi_\xi\big|_{\X_n}\big]
	\leq
	-\rho
	\leq
	\L_{\X_n,\eps}^+u
\end{equation*}
in $\X_n\setminus\O_\eps$, where the last inequality holds by assumption. Moreover, since $\Phi_\xi\geq 0$, we have that
\begin{equation*}
	u
	\leq
	v
\end{equation*}
in $\O_\eps$. Thus, in view of the comparison principle from \Cref{lem:comparison-pple} and since $\Phi_\xi\leq 1$ we get that
\begin{equation*}
	u
	\leq
	v
	\leq
	\|u\|_{L^\infty(\O_\eps)}+2\rho\,\frac{(R^2+1)^\sigma}{\sigma}
\end{equation*}
in $\X_n$. Following an analogous argument, we can show that
\begin{equation*}
	u
	\geq
	-\|u\|_{L^\infty(\O_\eps)}-2\rho\,\frac{(R^2+1)^\sigma}{\sigma}
\end{equation*}
so the proof is concluded.
\end{proof}

\subsection{Asymptotic continuity and Ascoli-Arzel\`a}
\label{sec:boundary-continuity}

We say that a domain $\Omega\subset\R^N$ satisfies the \emph{uniform exterior ball condition} if there is some $r>0$ such that for every $x\in\partial\Omega$ there exists $\xi\in\R^N\setminus\Omega$ for which
\begin{equation*}
	\overline B_r(\xi)\cap\overline\Omega
	=
	\{x\}.
\end{equation*}

\begin{lemma}[Asymptotic boundary continuity]
\label{boundarycontinuity}
Let $\Omega\subset\R^N$ be a bounded domain satisfying the uniform exterior ball condition and $\X_n\subset\Omega$ be a random data cloud. Suppose that the events \eqref{partition-a>0} and \eqref{event-a>0} hold and that $u\in L^\infty(\X_n)$ satisfies
\begin{equation*}
	\begin{dcases}
	\L_{\X_n,\eps}^+u
	\geq
	-\rho,
	\quad
	\L_{\X_n,\eps}^-u
	\leq
	\rho
	& \text{ in } \X_n\setminus\O_\eps,
	\\
	u=g & \text{ in } \O_\eps,
	\end{dcases}
\end{equation*}
for some $\rho>0$, where $g\in C(\Gamma_{\eps_0})$.
For every $\eta>0$ there exists $\delta>0$ such that
\begin{equation*}
	|u(Z_i)-g(x)|
	\leq
	\eta
\end{equation*}
for every $Z_i\in\X_n$ and $x\in\partial\Omega$ with $|Z_i-x|<\delta$.
\end{lemma}

\begin{proof}
Fix $\eta>0$.
Since $g$ is uniformly continuous, there exists $r_0$ such that
\begin{equation}
\label{equicontg}
|g(x)-g(y)|<\frac{\eta}{2}
\end{equation}
for every $|x-y|<r_0$.

Let $x\in\partial\Omega$. By the uniform exterior ball condition, there exists $\tilde \xi\in\R^N\setminus\Omega$ such that
\begin{equation}
\label{inter}
	\overline B_{2r}(\xi)\cap\overline\Omega
	=
	\{x\}.
\end{equation}
We consider $\xi=\frac{\tilde\xi+x}{2}$, so we have $B_r(\xi)\subset B_{2r}(\tilde\xi)$. We construct the barrier functions $v,w\in L^\infty(\X_n)$ given by
\begin{align*}
	v(y)
	=
	~&
	g(x)+\frac{\eta}{2}+\theta\big(\Phi_\xi(y)-\Phi_\xi(x)\big),
	\\
	w(y)
	=
	~&
	g(x)-\frac{\eta}{2}-\theta\big(\Phi_\xi(y)-\Phi_\xi(x)\big),
\end{align*}
where $\Phi_\xi$ is given in \eqref{eq:graph-barrier} and $\theta>0$.

Given \eqref{inter} we have that the distance from $B_r(\xi)$ to $\Omega\setminus B_{r_0}$ is positive so $\Phi_\xi(y)-\Phi_\xi(x)>0$ for every $y\in \Omega\setminus B_{r_0}$ (observe that the bound only depends on $r$ and $r_0$).
So, for $\theta$ large enough we get $v\geq \max g$ in $\Gamma_\eps\setminus B_{r_0}$.
Combining this with \eqref{equicontg} we get $v\geq g$ in $\Gamma_\eps$, and in particular $v\big|_{\X_n}\geq u$ in $\O_\eps$. By \Cref{lem:L+barrier}, it turns that
\begin{equation*}
	\L_{\X_n,\eps}^+\big[v\big|_{\X_n}\big]
	=
	\theta\L_{\X_n,\eps}^+\big[\Phi_\xi\big|_{\X_n}\big]
	\leq
	-\theta\,\frac{\sigma}{2}(R^2+1)^{-\sigma}.
\end{equation*}
Furthermore, for sufficiently large $\theta$ the right hand side can be bounded from above by $-\rho$, which in turn is bounded by $\L_{\X_n,\eps}^+u$. Then by he comparison principle (\Cref{lem:comparison-pple}) we get that $v\big|_{\X_n}\geq u$ in $\X_n$. Repeating the argument with $w$, we similarly get that
\begin{equation*}
	w\big|_{\X_n}
	\leq
	u
	\leq
	v\big|_{\X_n}
\end{equation*}
in $\X_n$. Replacing the definitions of $v$ and $w$ we immediately get that
\begin{equation*}
	|u(Z_i)-g(x)|
	\leq
	\frac{\eta}{2}+\theta\big|\Phi_\xi(Z_i)-\Phi_\xi(x)\big|
\end{equation*}
for every $Z_i\in\X_n$. Then the result follows by selecting small enough $\delta>0$ such that right hand side above is smaller than $\eta$ for $|Z_i-x|<\delta$.
\end{proof}

Next, combining our regularity result (\Cref{thm:main-graph-holder-regularity}) and the regularity near the boundary just obtained, the asymptotic continuity follows.

\begin{lemma}[Asymptotic continuity]
\label{lem:asymptoticcontinuity}
Let $\Omega\subset\R^N$ be a bounded domain satisfying the uniform exterior ball condition and $\X_n\subset\Omega$ be a random data cloud. Suppose that the events \eqref{partition-a>0} and \eqref{event-a>0} hold and that $u\in L^\infty(\X_n)$ satisfies
\begin{equation*}
	\begin{dcases}
	\L_{\X_n,\eps}^+u
	\geq
	-\rho,
	\quad
	\L_{\X_n,\eps}^-u
	\leq
	\rho
	& \text{ in } \X_n\setminus\O_\eps,
	\\
	u=g & \text{ in } \O_\eps,
	\end{dcases}
\end{equation*}
for some $\rho>0$, where $g\in C(\Gamma_{\eps_0})$.
For every $\eta>0$ there exists $\delta>0$ such that
\begin{equation*}
	|u(Z_i)-u(Z_j)|
	\leq
	\eta
\end{equation*}
for every $Z_i,Z_j\in\X_n$  with $|Z_i-Z_j|<\delta$.
\end{lemma}

\begin{proof}
Given $\eta>0$ we consider $\delta_1>0$ such that \Cref{boundarycontinuity} holds for $\eta/2$.
Then consider $Z_i,Z_j\in\X_n$ with $|Z_i-Z_j|<\delta_1/2$.
Suppose that at least one of them is at distance smaller than $\delta_1/2$ from the boundary, let say it is $Z_i$. Then there exists $x\in\partial\Omega$ such that $|Z_i-x|<\delta_1/2$.
Then $|Z_j-x|<|Z_j-Z_i|+|Z_i-x|<\delta_1$ and we get
\[
|u(Z_i)-u(Z_j)|
	\leq
|u(Z_i)-g(x)|+|g(x)-u(Z_j)|
\leq 
\eta/2+\eta/2=	
	\eta.
\]

Next we consider the case when both $Z_i,Z_j\in\X_n$ are at distance $\delta_1/2$ at least from the boundary.
Then, by \Cref{thm:main-graph-holder-regularity} with a standard covering argument, we get
\[
	|u(Z_i)-u(Z_j)|
	\leq
	C(|Z_i-Z_j|^\gamma+\eps^\gamma).
\]
So the result follows for if we take $\eps_0$ and $\delta<\delta_1$ small enough such that the right hand side is smaller $\eta$.
\end{proof}

Finally to obtain convergence of a subsequence we recall a Ascoli-Arzel\`a type lemma, see Lemma~4.2 in \cite{manfredipr12}.

\begin{lemma}
\label{lem:ascoliarzela}
Let $v_k:\overline{\Omega}\to\R$ be a sequence of functions such that
\begin{enumerate}
\item
there exists $C>0$ such that $|v_k(x)|<C$ for every $k$ and $x\in\overline{\Omega}$,

\item
given $\eta>0$ there exists $r_0$ and $k_0$ such 
\[
|v_k(x)-v_k(y)|<\eta
\] 
for every $k\geq k_0$ and $x,y\in\overline{\Omega}$ with $|x-y|<r_0$.
\end{enumerate}
Then, there exists $v:\overline{\Omega}\to\R$ and a subsequence still denoted by $v_k$ such that
\[
v_k\to v  \quad \text{uniformly in } \overline{\Omega}
\]
as $k\to \infty$.
\end{lemma}

Finally, \Cref{lem:uniform-bound} and \Cref{lem:asymptoticcontinuity}
allow us to verify that the sequence of functions satisfies the hypothesis of the previous lemma.

\begin{theorem}
\label{convergence}
Let $\{n_k\}_k$ be a sequence of positive integers and $\{\eps_k>0\}_k$ such that
\[
n_k\to \infty, \quad \eps_k\to 0 \quad \text{and}\quad
\liminf_k\big( 2n_k\exp\{-c_0n_k\eps_k^{3N+4+(N+2)a}\}\big)<1.
\] 
For each $k\in\N$, let $\X_{n_k}^{(k)}\subset\Omega$ be a random data cloud in $\Omega$ and denote $\O_{\eps_k}=\X_{n_k}^{(k)}\cap\Gamma_{\eps_k}$.
Let $u_k\in L^\infty(\X_{n_k}^{(k)})$ be such that
\begin{equation*}
	\begin{dcases}
	\L_{\X_{n_k},\eps_k}^+u_k
	\geq
	-\rho,
	\quad
	\L_{\X_n,\eps_k}^-u_k
	\leq
	\rho
	& \text{ in } \X_{n_k}^{(k)}\setminus\O_{\eps_k},
	\\
	u_k=g & \text{ in } \O_{\eps_k},
	\end{dcases}
\end{equation*}
Then, with probability $1$, there exists $u:\Omega\to \R$ and a subsequence still denoted by $u_k$ such that
\[
u_k \circ T_{\eps_k}\to u  \quad \text{uniformly in } \overline{\Omega}
\]
as $k\to \infty$.
\end{theorem}

\begin{proof}
Since
\begin{equation*}
	\limsup_k\big( 1-2n_k\exp\{-c_0n_k\eps_k^{3N+4+(N+2)a}\}\big)>0,
\end{equation*}
by \Cref{lem:dens-eps} we have that \eqref{partition-a>0} and \eqref{event-a>0} hold for infinity many $k$ with probability $1$.
We consider the sequence for those values and apply
\Cref{lem:uniform-bound} and \Cref{lem:asymptoticcontinuity} to them.
Observe that the bound obtained for $u_k$ immediately translates to $u_k \circ T_{\eps_k}$. 
Also the asymptotic continuity translates to $u_k \circ T_{\eps_k}$ by \eqref{partition-a>0}.
Finally, we can apply \Cref{lem:ascoliarzela} to $u_k \circ T_{\eps_k}$ to obtain the convergent subsequence.
\end{proof}

Observe that if the election of the points is independent for each $k$, by Borel-Cantelli lemma it is enough to have
\[
\sum_k 1-2n_k\exp\{-c_0n_k\eps_k^{3N+4+(N+2)a}\}=+\infty.
\]

\section{Convergence to viscosity solutions}
\label{sec:asymptotic}

In this section we establish as an application that one can pass to the limit and that the limit under suitable assumptions satisfies a PDE. 
So far we have assumed that the probability density $\dens$ is a Lipschitz function but in order to talk about the limit PDE we assume that $\dens\in C^1$.
We consider $u_k$ such that
\begin{equation*}
	\begin{dcases}
	\L_{\X_{n_k},\eps_k}^+u_k=f
	& \text{ in } \X_n\setminus\O_{\eps_k},
	\\
	u_k=g & \text{ in } \O_{\eps_k},
	\end{dcases}
\end{equation*}
for a bounded continuous function $f:\Omega\to\R$.
We assume that $n_k$ and $\eps_k$ are such that the hypothesis of \Cref{convergence} are satisfied for $\rho=\|f\|_\infty$.

\begin{lemma}[Asymptotic expansions]\label{lem:asymptotics}
Let $v\in C^2(\Omega)$. Then
\begin{align*}
	\L^+v(x)
	:\,=
	\lim_{\eps\to 0}\L_{\Omega,\eps}^+v(x)
	=
	~&
	\frac{\alpha\Lambda}{2}\lambda_N(D^2v(x))
	+
	\frac{\beta}{2(N+2)}\Delta v(x)
	\\
	~&
	+
	\frac{\alpha\tau}{2}|\nabla v(x)|
	+
	\frac{\beta}{N+2}\nabla v(x)^\top\frac{\nabla\dens(x)}{\dens(x)},
	\\
	\L^-v(x)
	:\,=
	\lim_{\eps\to 0}\L_{\Omega,\eps}^-v(x)
	=
	~&
	\frac{\alpha\Lambda}{2}\lambda_1(D^2v(x))
	+
	\frac{\beta}{2(N+2)}\Delta v(x)
	\\
	~&
	-
	\frac{\alpha\tau}{2}|\nabla v(x)|
	+
	\frac{\beta}{N+2}\nabla v(x)^\top\frac{\nabla\dens(x)}{\dens(x)},
\end{align*}
where $\lambda_1(M)$ and $\lambda_N(M)$ stand for the smallest and the largest eigenvalue of a real symmetric matrix $M$.

\end{lemma}

\begin{proof}
Let $v\in C^2(\Omega)$ and $x\in\Omega$, then by the second order Taylor's expansion of $v$ we have
\begin{equation*}
	v(x+\eps z+\eps^2h)
	=
	v(x)+\eps\nabla v(x)^\top z+\eps^2\Big(\nabla v(x)^\top h+\frac{1}{2}z^\top D^2v(x)z\Big)+o(\eps^2).
\end{equation*}
Thus,
\begin{align*}
	v(x+\eps z)+v(x-\eps z+\eps^2h)-2v(x)=
	~&
	\eps^2\big(\nabla v(x)^\top h+z^\top D^2v(x)z\big)+o(\eps^2),
\end{align*}
and taking supremum and infimum over $|h|<\tau$ and $|z|<\Lambda$ and dividing by $2\eps^2$ we get
\begin{align*}
	\inf_{|z|<\Lambda\eps}\frac{\delta_{\tau\eps^2}^-v(x,z)}{2\eps^2}
	=
	~&
	\frac{1}{2}\bigg[-\tau|\nabla v(x)|+\lambda_1(D^2v(x))\bigg]+o(\eps^0),
	\\
	\sup_{|z|<\Lambda\eps}\frac{\delta_{\tau\eps^2}^+v(x,z)}{2\eps^2}
	=
	~&
	\frac{1}{2}\bigg[+\tau|\nabla v(x)|+\lambda_N(D^2v(x))\bigg]+o(\eps^0).
\end{align*}

On the other hand, using that $\dens$ is the density function of $\mu$ we can write
\begin{multline*}
	\dashint_{B_\eps(x)}\frac{v(y)-v(x)}{\eps^2}\ d\mu(y)
	=
	\frac{|B_\eps|}{\mu(B_\eps(x))}\bigg(\dashint_{B_\eps(x)}\frac{v(y)-v(x)}{\eps}\,\frac{\dens(y)-\dens(x)}{\eps}\ dy
	\\
	+\dens(x)\dashint_{B_\eps(x)}\frac{v(y)-v(x)}{\eps^2}\ dy\bigg),
\end{multline*}
where
\begin{align*}
	\frac{|B_\eps|}{\mu(B_\eps(x))}
	=
	\bigg(\dashint_{B_\eps(x)}\dens(y)\ dy\bigg)^{-1}
	=
	\frac{1}{\dens(x)}+o(\eps).
\end{align*}
Replacing the asymptotic expansions and computing the integrals we get
\begin{align*}
	\dashint_{B_\eps(x)}&\frac{v(y)-v(x)}{\eps^2}\ d\mu(y)
	\\
	&=
	\frac{1}{\dens(x)}\dashint_{B_1}\nabla v(x)^\top y\,\nabla\dens(x)^\top y\ dy+\frac{1}{2}\dashint_{B_1}y^\top D^2v(x)y\ dy+o(\eps^0)
	\\
	&=
	\trace\bigg\{\Big(\frac{1}{\dens(x)}\nabla v(x)\otimes\nabla\dens(x)+\frac{1}{2}D^2v(x)\Big)\dashint_{B_1}y\otimes y\ dy\bigg\}+o(\eps^0)
	\\
	&=
	\frac{1}{2(N+2)}\Big(2\nabla v(x)^\top\frac{\nabla\dens(x)}{\dens(x)}+\Delta v(x)\Big)+o(\eps^0).
\end{align*}
Thus
\begin{align*}
	\lim_{\eps\to 0}\L_{\Omega,\eps}^+ v(x)
	=
	~&
	\frac{\alpha}{2}\bigg(\tau|\nabla v(x)|+\Lambda\lambda_N(D^2v(x))\bigg)
	\\
	~&
	+
	\frac{\beta}{2(N+2)}\bigg(2\nabla v(x)^\top\frac{\nabla\dens(x)}{\dens(x)}+\Delta v(x)\bigg)
\end{align*}
so the proof is concluded.

\end{proof}

Observe that if $u\in C^3(\Omega)$, since every point has one in the data cloud at distance at most $\eps^{3+a}$ we get that the difference
\[
\max_{Z_j\in\V_{\Lambda\eps}(Z_i)}\max_{Z_k\in\V_{\tau\eps^2}(2Z_i-Z_j)}\frac{u(Z_j)+u(Z_k)}{2}
-\sup_{|z|<\Lambda\eps}\sup_{|h|<\tau\eps^2}\frac{u(Z_i+z)+u(Z_i-z+h)}{2}
\]
is of order $O(\eps^{3})$.
Combining this with \eqref{eq:discrete-to-nonlocal-averages-a}, we get
\begin{equation}\label{difference-L+=o(eps)}
\L_{\Omega,\eps}^+u-\L_{\X_{n},\eps}^+u=O(\eps^a).
\end{equation}
Therefore the asymptotic expansions in \Cref{lem:asymptotics} also hold for $\L_{\X_{n},\eps}^+$.

Next, we recall the definition of a viscosity solution for the limiting operator $\L^+$ in \Cref{lem:asymptotics} (the definition is analogous for $\L^-$).

\begin{definition} \label{def.sol.viscosa}
A continuous function $v:\Omega \to \R$  verifies
$$
\L^+v  = 0
$$
in the viscosity sense if
\begin{enumerate}
\item
for every $\xi\in C^2$ such that $v-\xi$
has a strict minimum at the point $x \in \Omega$
with $v(x)=\xi(x)$,
we have
$$
\L^+ \xi (x) \leq 0.
$$

\item
for every $\psi\in C^2$ such that $v-\psi$
has a strict maximum at the point $x \in \Omega$
with $v(x)=\psi(x)$,
we have
$$
\L^+ \psi (x) \geq 0.
$$
\end{enumerate}
\end{definition}
For more details on viscosity solutions, see for example \cite{koike04} or \cite{katzourakis15}.

\begin{theorem}
\label{thm:visc-sol}
Let $\{n_k\}_k$ be a sequence of positive integers and $\{\eps_k>0\}_k$ such that
\[
n_k\to \infty, \quad \eps_k\to 0 \quad \text{and}\quad
\liminf_k\big( 2n_k\exp\{-c_0n_k\eps_k^{3N+4+(N+2)a}\}\big)<1.
\] 
For each $k\in\N$, let $\X_{n_k}^{(k)}\subset\Omega$ be a random data cloud in $\Omega$ and denote $\O_{\eps_k}=\X_{n_k}^{(k)}\cap\Gamma_{\eps_k}$.
Let $u_k\in L^\infty(\X_{n_k}^{(k)})$ be such that
\begin{equation*}
	\begin{dcases}
	\L_{\X_{n_k}^{(k)},\eps_k}^+u_k=f
	& \text{ in } \X_n\setminus\O_{\eps_k},
	\\
	u_k=g & \text{ in } \O_{\eps_k},
	\end{dcases}
\end{equation*}
Then, with probability 1, there exists $u:\Omega\to \R$ and a subsequence still denoted by $u_k$ such that
\[
u_k \circ T_{\eps_k}\to u  \quad \text{uniformly in } \overline{\Omega}
\]
as $k\to \infty$.
The function $u$ is a  viscosity solution to

\begin{equation*}
	\begin{cases}
	\L^+u=f
	& \text{ in } \Omega,
	\\
	u=g & \text{ in } \partial\Omega.
	\end{cases}
\end{equation*}
\end{theorem}

\begin{proof}
The convergence follows from \Cref{convergence}, it remains to prove that the limit function is a solution to the PDE. Let $\xi\in C^{2}$ be such that $u-\xi$ has a strict minimum at the point $x \in \Omega$ with $u(x)=\xi(x)$. We want to show that
\[
\L^+\xi (x)\leq 0.
\]
As $\tilde u_k=u_k \circ T_{\eps_k} \to u$ uniformly in $\overline{\Omega}$ we have the existence of a sequence
$x_k$ such that $x_k \to x$ as $k \to \infty$ and 
\[
\tilde u_k(y)-\xi(y) \geq \tilde u_k(x_k)-\xi(x_k)- \eps_k^3
\]
Since $u_k$ is a solution to
\[
\L_{\X_{n_k}^{(k)},\eps_k}^+ u_k(x_k)=f(x_k),
\]
and since the asymptotic expansion holds also for the discrete operator (see \Cref{lem:asymptotics} and \eqref{difference-L+=o(eps)}), we obtain that $\xi$ verifies the inequality
$$
f(x_k) \geq \L^+\xi(x_k) - \eps_k^3,
$$
which in the limit gives us the desired result.
\end{proof}

Similarly, recalling the tug-of-war type graph operators from \Cref{example:tug-of-war}, we obtain that the limit is a viscosity solution to a weighted $p$-Laplace equation with $p>2$,
\begin{align*}
\operatorname{div}(\phi^2\abs{\nabla u}^{p-2}\nabla u)&=\abs{\nabla u}^{p-2}\Big(\operatorname{div}(\phi^2 u)+(p-2)\phi^2\langle D^2 u \frac{\nabla u}{\abs{\nabla u}},\frac{\nabla u}{\abs{\nabla u}} \rangle\Big )\\
&=\abs{\nabla u}^{p-2}\Big(\operatorname{div}(\phi^2 u)+(p-2)\phi^2\Delta^{N}_{\infty}u\Big )=0,
\end{align*}
where we recall that $\phi$ is the $C^1$ probability density of the arriving data points bounded between positive constants. Moreover, $\Delta^{N}_{\infty}u$ is normalized or game theoretic infinity Laplacian and we also recall that we assume exterior sphere condition for $\Om$. The following result extends Calder's result obtained for $p>n$ in \cite[Theorem 1]{calder19b} for all $p>2$. This is due to a fact that the local regularity argument in that paper is the one where $p>n$ appears.  

\begin{theorem}
\label{thm:visc-sol-p}
Under the assumptions of \Cref{thm:visc-sol}, let $u_k\in L^\infty(\X_{n_k}^{(k)})$ satisfy
\begin{equation*}
	\begin{dcases}
	\L_{\X_{n_k}^{(k)},\eps_k}u_k=0
	& \text{ in } \X_n\setminus\O_{\eps_k},
	\\
	u_k=g & \text{ in } \O_{\eps_k},
	\end{dcases}
\end{equation*}
for each $k\in\N$, where $\L_{\X_{n_k}^{(k)},k}$ stands for the tug-of-war operator on the random graph $\X_{n_k}^{(k)}$ (see \Cref{example:tug-of-war}). Then, the limit function $u$ is a viscosity solution to
\begin{equation*}
	\begin{cases}
	\operatorname{div}(\phi^2\abs{\nabla u}^{p-2}\nabla u)=0
	& \text{ in } \Omega,
	\\
	u=g & \text{ in } \partial\Omega,
	\end{cases}
\end{equation*}
with probability $1$.
\end{theorem}

The proof for the convergence is exactly the same as in the previous theorem using our regularity result,
apart from using the exterior ball condition and barrier to guarantee the boundary regularity and boundedness. The claim that the limit is a viscosity solution follows in a similar way as in the previous theorem by using Taylor's expansion for a smooth test function touching the limit. For details one can consult \cite[Theorems 5 and 6] {calder19b} by observing that also these results work for all $p>2$.

\appendix
\section{Histogram estimation of the probability density} 
\label{sec:histogram}

We begin by stating the multiplicative Chernoff bound, see Corollary 4.6 in \cite{mitzenmacher}.

\begin{lemma}
\label{HoeffdingBernoulli}
Let $X_i,\dots,X_n$ be i.i.d. Bernoulli random variables with $p=\E[X_i]$.
Consider $S_n = X_1 + \cdots + X_n$.
Then
\[
\P \left(\left |S_n - pn \right | \geq pnt \right) \leq 2\exp \left(-\frac{npt^2}{3} \right)
\]
for every $0<t<1$.
\end{lemma}

Next we state and prove the lemma from which we deduce \Cref{lem:dens-eps}. We assume a kind of a measure density condition: Let $\Omega\subset \R^N$ be an open set such that there exists $c>0$ such that for every $\delta>0$ there exists $B_1,\dots,B_M$ a partition of $\Omega$ such that $B_i\subset B_\delta(x_i)$ for some $x_i\in B_i$ and $|B_i|\geq c\delta^N$.

The above condition is ensured for example when $\Omega$ satisfies the following interior ball condition: there exists $r>0$ such that for every $x\in\Omega$, there exist a sphere $B_r(y)\subset\Omega$ such that $x\in B_r(y)$. To prove that this interior ball condition implies the preceding hypothesis, we consider the grid $G=c\delta\Z^N$ and the points $X=G\cap \Omega$. We take the Voronoi partition associated to $X$, that is
\[
B_x=\{y\in\Omega: x \text{ is the closest point to } y \text{ in } X \}
\]
Any tie-breaking criterion is acceptable for our purposes. One can check that this partition verifies the condition. Also recall that $\dens$ is a Lipschitz probability measure density such that $\dens_0\leq \dens\leq \dens_1$ for some constants $\dens_1,\dens_0>0$.

\begin{lemma}
\label{lem:apendix}
Let $\delta>0$ and $n$ a positive integer.
Let $X_1,\dots,X_n$ be a sequence of i.i.d. random variable with density $\dens$.
Then for every $0< \lambda\leq 1$, there exist $\dens_\delta$ and a partition $\U_1,\dots,\U_n$ such that 
\begin{equation}
\label{eqUi}
\U_i\subset B_\delta(X_i) \quad \text{and} \quad \int_{\U_i} \dens_\delta\ dx=\frac{1}{n} 
\end{equation}
for every $i=1,\dots,n$ and
\begin{equation}
\label{proba}
\|\dens_\delta-\dens\|_{\Omega}\leq C_0(\lambda+\delta)
\end{equation}
at least with probability 
\[
1-2n\exp \left(-c_0\lambda^2 n\delta^{N} \right)
\]
where $C_0=C_0(\dens)$ and $c_0=c_0(\dens,\Omega)$.
\end{lemma}

\begin{proof}
By the hypothesis on $\Omega$ for $\delta/2>0$ we have that
there exists $B_1,\dots,B_M$ a partition of $\Omega$ such that $B_i\subset B_{\delta/2}(x_i)$ for some $x_i\in B_i$ and $|B_i|\geq c2^{-N}\delta^N$.
Let $\dens_\delta$ be the histogram density estimator
\[
\dens_\delta(x)=
\frac{1}{n}\sum_{i=1}^{M}\frac{n_i}{|B_i|}\ind_{B_i}(x),
\]
where $n_i$ is the number of points from $X_1,\dots, X_n$ in $B_i$.
We have
\[
\int_\Omega\dens_\delta\ dx=\frac{1}{n}\sum_{i=1}^{M}n_i=1.
\]

Observe that
\[
n_i=\sum_{j=1}^n\ind_{B_i}(X_j) 
=\card{\X_n\cap B_i}
\]
and $\ind_{B_i}(X_j)$ is a Bernoulli of parameter $p_i=\int_{B_i}\dens\ dx$.
Then, by \Cref{HoeffdingBernoulli}, we have
\[
\P \left(\left |n_i - np_i \right | \geq \lambda np_i \right) \leq 2\exp \left(-\frac{\lambda^2}{3} np_i \right)
\]

We have
\[
p_i=\int_{B_i}\dens\ dx\geq \dens_0|B_i|\geq \dens_0c2^{-N}\delta^N.
\]
So, we get
\[
\P \left(\left |\frac{n_i}{n} - p_i \right | \geq \lambda p_i \right) \leq 
2\exp \left(-\frac{\lambda^2}{3} n\dens_0c2^{-N}\delta^N \right)
\]
By using this estimate for every $i$, we get
\begin{equation}
\label{nibound}
\left|\frac{n_i}{n} - p_i \right | < \lambda p_i
\end{equation}
for all $i=1,\dots,n$ with probability at least
\[
1-2n\exp \left(-\frac{\lambda^2}{3} n\dens_0c2^{-N}\delta^N \right)
=
1-2n\exp \left(-c_0\lambda^2 n\delta^{N} \right)
\]
where $c_0=\frac{1}{3}\dens_0c2^{-N}$.

In particular \eqref{nibound} implies that $n_i>0$, that is, every $B_i$ contains at least one $X_j$.
Let $k_{i,1},\dots, k_{i,n_i}$ denote the indices of
the random variables that fall in $B_i$.
We construct a partition $B_{i,1},\dots , B_{i,n_i}$ of each $B_i$ consisting of sets of equal measure.
We set $\U_{k_{i,j}} = B_{i,j}$.

Let $j\leq n$ and consider $i\leq M$ and $k\leq n_i$ such that $\U_j=B_{i,k}$, we have
\[
\int_{U_j}\dens_\delta\ dx=\frac{1}{n}\frac{n_i}{|B_i|}|B_{i,k}|=\frac{1}{n}
\] 
and since $X_j\in B_i\subset B_{\delta/2}(x_i)$ we get
\[
\U_j\subset B_i\subset B_{\delta/2}(x_i) \subset B_\delta(X_j).
\]
We have proved \eqref{eqUi}.

To prove \eqref{proba} we bound $p_i$ in \eqref{nibound} by
\[
p_i=\int_{B_i}\dens\ dx\leq \dens_1|B_i|.
\]
Then, given $x\in U$ we consider $i$ such that $x\in B_i$, we have
\[
\begin{split}
|\dens(x)-\dens_\delta(x)|
&=\left|\dens(x)-\frac{n_i}{n|B_i|}\right|\\
&=\left|\dens(x)-\int_{B_i}\dens\ dx+\int_{B_i}\dens\ dx-\frac{n_i}{n|B_i|}\right|\\
&\leq\left|\dens(x)-\frac{1}{|B_i|}\int_{B_i}\dens\ dx\right|+\left|\frac{1}{|B_i|}\int_{B_i}\dens\ dx-\frac{n_i}{n|B_i|}\right|\\
&\leq \delta L_{\dens}+\lambda\dens_1
\end{split}
\]
where $L_{\dens}$ is the Lipschitz constant of $\dens$.
We have proved the results for $C_0=\max\{L_{\dens},\dens_1\}$.
\end{proof}

Observe that in \Cref{lem:apendix} we have two factors at play: the precision of the event and the likelihood of its occurrence. On the one hand the result is stronger as $\lambda$ and $\delta$ are smaller, as shown in \eqref{eqUi} and \eqref{proba}. On the other hand the result is meaningless if the probability is not greater than $0$. That is, we want  $\lambda$ and $\delta$ to be small enough for our purposes, but $n$ to be large so that $n\exp(-c_0\lambda^2n\delta^{2N})$ is small.


\begin{thebibliography}{PSSW09}


\bibitem[ABP22]{arroyobp22}
{\'A}.~Arroyo, P.~Blanc, and M.~Parviainen.
\newblock Local regularity estimates for general discrete dynamic programming equations.
\newblock {\em J.\ Math.\ Pures Appl.}, 167:225--226, 2022.

\bibitem[ABP23]{arroyobp23}
{\'A}.~Arroyo, P.~Blanc, and M.~Parviainen.
\newblock H{\"o}lder regularity for stochastic processes with bounded and measurable increments.
\newblock {\em Ann.\ Inst.\ H.\ Poincar\'e Anal.\ Non Lin\'eaire.}, 40:215--258, 2023.

\bibitem[AP24]{arroyop}
{\'A}.~Arroyo and M.~Parviainen.
\newblock H\"older estimate for a tug-of-war game with $1<p<2$ from Krylov-Safonov regularity theory.
\newblock  {\em Rev. Mat. Iberoam.}, 40(3): 1023--1044, 2024.

\bibitem[BR19]{blancr19b}
P.~Blanc and J.~D. Rossi.
\newblock {\em Game Theory and Partial Differential Equations}, volume~31 of {\em De Gruyter Series in Nonlinear Analysis and Applications}.
\newblock De Gruyter, 2019.

\bibitem[BLM13]{boucheronlm13}
S.~Boucheron, G.~Lugosi, and P.~Massart.
\newblock Concentration inequalities: A nonasymptotic theory of independence.
\newblock {\em Oxford University Press}, 2013.

\bibitem[BCR23]{bungertcr23}
L.~Bungert, J.~Calder, and T.~Roith.
\newblock Uniform convergence rates for {L}ipschitz learning on graphs.
\newblock {\em IMA Journal of Numerical Analysis}, 43(4):2445--2495, 2023.

\bibitem[BR23]{bungertr23}
L.~Bungert and T.~Roith.
\newblock Continuum limit of {L}ipschitz learning on graphs.
\newblock {\em Found. Comput. Math.}, 23(2):393--431, 2023.

\bibitem[Caf89]{caffarelli89}
L.~A. Caffarelli.
\newblock Interior a priori estimates for solutions of fully nonlinear equations.
\newblock {\em Ann. of Math. (2)}, 130(1):189--213, 1989.

\bibitem[Cal19a]{calder19}
J.~Calder.
\newblock Consistency of {L}ipschitz learning with infinite unlabeled data and finite labeled data.
\newblock {\em SIAM Journal on Mathematics of Data Science}, 1(4):780--812, 2019.

\bibitem[Cal19b]{calder19b}
J.~Calder.
\newblock The game theoretic {$p$}-{L}aplacian and semi-supervised learning with few labels.
\newblock {\em Nonlinearity}, 32(1):301--330, 2019.

\bibitem[CG22]{calderg22}
J.~Calder and N.~Garc\'{\i}a Trillos.
\newblock Improved spectral convergence rates for graph Laplacians on $\varepsilon$-graphs and $k$-NN graphs.
\newblock {\em Appl. Comput. Harmon. Anal.}, 60:123--175, 2022.

\bibitem[CGTL22]{caldergl22}
J.~Calder, N.~Garc\'{\i}a~Trillos, and M.~Lewicka.
\newblock Lipschitz regularity of graph {L}aplacians on random data clouds.
\newblock {\em SIAM J. Math. Anal.}, 54(1):1169--1222, 2022.

\bibitem[CS09]{caffarellis09}
L.~Caffarelli and L.~Silvestre.
\newblock Regularity theory for fully nonlinear integro-differential equations.
\newblock {\em Comm. Pure Appl. Math.}, 62(5):597--638, 2009.

\bibitem[CSZ10]{chapellesz06}
O.~Chapelle, B.~Sch{\"o}lkopf, and A.~Zien.
\newblock {\em Semi-supervised Learning}.
\newblock Adaptive computation and machine learning. MIT Press, 2010.


\bibitem[EDT17]{elmoatazdt17}
A.~Elmoataz, X.~Desquesnes, and M.~Toutain.
\newblock On the game $p$-{L}aplacian on weighted graphs with applications in image processing and data clustering.
\newblock {\em European Journal of Applied Mathematics}, 28(6):922--948, 2017.

\bibitem[EG15]{evansg15}
L.C.~Evans and R.F.~Gariepy.
\newblock Measure theory and fine properties of functions.
\newblock Textbooks in Mathematics, Revised Edition. CRC Press, Boca Raton, FL, 2015.

\bibitem[Kat15]{katzourakis15}
N.~Katzourakis.
\newblock {\em An introduction to viscosity solutions for fully nonlinear {PDE}
  with applications to calculus of variations in {$L^\infty$}}.
\newblock SpringerBriefs in Mathematics. Springer, Cham, 2015.

\bibitem[Koi04]{koike04}
S.~Koike.
\newblock {\em A beginner's guide to the theory of viscosity solutions},
  volume~13 of {\em MSJ Memoirs}.
\newblock Mathematical Society of Japan, Tokyo, 2004.
\bibitem[KS79]{krylovs79}
N.~V. Krylov and M.~V. Safonov.
\newblock An estimate for the probability of a diffusion process hitting a set of positive measure.
\newblock {\em Dokl. Akad. Nauk SSSR}, 245(1):18--20, 1979.


\bibitem[KT90]{kuot90}
H.~J. Kuo and N.~S. Trudinger.
\newblock Linear elliptic difference inequalities with random coefficients.
\newblock {\em Math. Comp.}, 55(191):37--53, 1990.

\bibitem[KT96]{kuot96}
H.~J. Kuo and N.~S. Trudinger.
\newblock Positive difference operators on general meshes.
\newblock {\em Duke Math. J.}, 83(2):415--433, 1996.

\bibitem[Lew20]{lewicka20}
M.~Lewicka.
\newblock {\em A Course on Tug-of-War Games with Random Noise}.
\newblock Universitext. Springer-Verlag, Berlin, 2020.
\newblock Introduction and Basic Constructions.

\bibitem[LP18]{luirop18}
H.~Luiro and M.~Parviainen.
\newblock Regularity for nonlinear stochastic games.
\newblock {\em Ann.\ Inst.\ H.\ Poincar{\'e} Anal.\ Non Lin{\'e}aire}, 35(6):1435--1456, 2018.

\bibitem[LPS13]{luirops13}
H.~Luiro, M.~Parviainen, and E.~Saksman.
\newblock Harnack's inequality for $p$-harmonic functions via stochastic games.
\newblock {\em Comm.\ Partial Differential Equations}, 38(11):1985--2003, 2013.

\bibitem[MPR12]{manfredipr12}
J.J. Manfredi, M.~Parviainen, and J.D. Rossi.
\newblock On the definition and properties of $p$-harmonious functions.
\newblock {\em Ann. Scuola Norm. Sup. Pisa Cl. Sci.}, 11(2):215--241, 2012.

\bibitem[MU17]{mitzenmacher}
M. Mitzenmacher and E. Upfal.
\newblock Probability and computing: Randomization and probabilistic techniques in algorithms and data analysis.
\newblock Cambridge university press, 2017.

\bibitem[Par24]{parviainen24}
M.~Parviainen.
\newblock {\em Notes on tug-of-war games and the $p$-{L}aplace equation}.
\newblock SpringerBriefs on PDEs and Data Science. Springer Nature, Singapore, 2024.

\bibitem[PS08]{peress08}
Y.~Peres and S.~Sheffield.
\newblock Tug-of-war with noise: a game-theoretic view of the {$p$}-{L}aplacian.
\newblock {\em Duke Math. J.}, 145(1):91--120, 2008.

\bibitem[PSSW09]{peresssw09}
Y.~Peres, O.~Schramm, S.~Sheffield, and D.~B. Wilson.
\newblock Tug-of-war and the infinity {L}aplacian.
\newblock {\em J. Amer. Math. Soc.}, 22(1):167--210, 2009.



\bibitem[Tru80]{trudinger80}
N.~S. Trudinger.
\newblock Local estimates for subsolutions and supersolutions of general second order elliptic quasilinear equations.
\newblock {\em Invent. Math.}, 61(1):67--79, 1980.

\end{thebibliography}

\def\cprime{$'$} \def\cprime{$'$} \def\cprime{$'$}

\end{document}